\documentclass[reqno]{amsart}
\numberwithin{equation}{section}

\usepackage{amssymb}
\usepackage{graphicx}
\usepackage[utf8]{inputenc}
\usepackage[hidelinks]{hyperref}

\newtheorem{theorem}{Theorem}
\newtheorem{definition}{Definition}[section]
\newtheorem{proposition}{Proposition}[section]
\newtheorem{lemma}{Lemma}[section]
\newtheorem{corollary}{Corollary}[section]
\newtheorem{rmk}{Remark}[section]

\newenvironment{remark}{\begin{rmk}\rm}{\end{rmk}}
\newtheorem{notat}{Notation}
\newenvironment{notation}{\begin{notat}\rm}{\end{notat}}

\renewcommand{\ne}{\not =}
\newcommand{\obra}[3]{{\textbf #1}: {``\emph{#2}''.\/}  {#3}.}

\title[Truncated Local Uniformization]{Truncated Local Uniformization of Formal Integrable Differential Forms}

\author{F. Cano}
\address{Universidad de Valladolid}
\email{fcano@agt.uva.es}

\author{M. Fern\'{a}ndez-Duque}
\address{Universidad Autónoma de México}
\email{mfduque@im.unam.mx}

\date{ \today}
\begin{document}
\maketitle

\begin{abstract} We prove the existence of Local Uniformization for rational codimension one foliations along rational rank one valuations, in any ambient dimension. This result is consequence of the Truncated Local Uniformization of integrable formal differential $1$-forms, that we also state and prove in the paper. Thanks to the truncated approach, we perform a classical inductive procedure, based both in the control of the Newton Polygon and in the po\-ssi\-bi\-li\-ty of avoiding accumulations of values, given by the existence of suitable Tschirnhausen transformations.
\end{abstract}
\tableofcontents

\section{Introduction}
In this work, we obtain Local Uniformization for rational foliations in any am\-bient dimension, along a rational valuation of rank one, as well as a Truncated Local Uniformization of formal integrable differential $1$-forms.

Let us consider a birational class $\mathcal C$ of projective varieties over a characteristic zero base field $k$. That is, we take a field extension $K/k$, where $K$ is the field of rational functions $K=k(M)$ of any projective model $M\in {\mathcal C}$. A {\em rational foliation $\mathcal F$ over $K/k$} is any one-dimensional $K$-vector subspace of the Kähler differentials $\Omega_{K/k}$, whose elements satisfy Frobenius integrability condition $\omega\wedge d\omega=0$.

 In this paper we prove:
\begin{theorem}
 \label{teo:mainmain}
 Let ${\mathcal F}\subset \Omega_{K/k}$ be a rational foliation over $K/k$ and consider a rank one $k$-rational valuation ring $R$ of $K$. There is a birational projective model $M$ of $K$ such that $\mathcal F$ is simple at the center $P$ of $R$ in $M$.
\end{theorem}

This is a result of Local Uniformization in the sense of Zariski \cite{Zar1}, where the objects to be considered are  rational foliations. Let us note that the case of a rational function $\phi\in K$ is included, when we consider the differential $\omega=d\phi$; in this case we obtain a classical Local Uniformization of the divisor $\operatorname{div}(\phi)$.

The reduction of singularities of codimension one foliations is an open problem in dimension bigger of equal than four. We have positive answers by Seidenberg \cite{Sei} in 1968, for the two-dimensional case,  and by Cano \cite{Can2} in 2004, for the three-dimensional case. This is in contrast with Hironaka's results  \cite{Hir1} of 1964, that provide a reduction of singularities for algebraic varieties in characteristic zero and any dimension.

We have to deal with the possibility of having nonzero objects with infinite value. This is  a reason for making the hard part of the proof in terms of formal objects
and in a ``truncated way''. The infinite value for a nonzero formal function only comes when it is a non rational function, see for instance the works on the implicit ideals \cite{Her-O-S-T, Cut-E}. However, when we are considering $1$-differential forms, the property of having infinite value may appear even with differential forms having polynomial coefficients. This is the case of the well known Euler's Equation in dimension two.

The meaning of ``simple integrable $1$-differential form'' has been established in previous works. In dimension two by Seidenberg \cite{Sei} and in any dimension by Cerveau-Mattei \cite{Cer-M}, Mattei-Moussu \cite{Mat-M}, Mattei  \cite{Mat},  Cano-Cerveau \cite{Can-C} and Cano \cite{Can2} among others. The definition contains the case of the differential of a monomial (normal crossings) and several versions of saddle-nodes.

As it is standard in reduction of singularities, a normal crossings divisor is present in the definition of ``final'' points after reduction of singularities. It can be an exceptional divisor created along the reduction of singularities process, or also an originally prescribed divisor (for instance, in arguments working by induction). Let us recall the local definition of simple point for the case of a  function. A function $f$ is simple with respect to a normal crossings divisor $x_1x_2\cdots x_r=0$ if one of the following properties holds:
\begin{itemize}
\item The function is a monomial in $\boldsymbol{x}$ times a unit $U$. That is $f=U\boldsymbol{x^a}$.
\item There is a new local coordinate $y$ and a unit $U$ such that $f= Uy^b\boldsymbol{x^a}$.
\end{itemize}
In the first case, that we call the {\em corner case}, the zeroes of $f$ are contained in the divisor; in the second case, that we call {\em trace case}, the  set of zeroes of $ f $ contains $y=0$.  When $f$ is a rational function, we can avoid trace points along the valuation ring $R$; otherwise,  the value of $f$ would be infinity. Nevertheless, it is possible to get trace points for a rational differential $1$-form; in the two dimensional case, this indicates the presence of a formal non-rational invariant curve.

 The definition of a simple formal integrable differential $1$-form $\omega$ is compatible with the above one given for functions, in the sense that that if $f$ is not a unit and it is simple, then $df$ will be simple. The precise definition, the existence of normal formal forms and other properties may be found in \cite{Can1,Can2,Can-C,FeD}, where ``simple points'' are ``pre-simple points'' with a diophantine additional condition, that is automatically satisfied in the case of functions. By a paper of Fernández-Duque \cite{FeD}, it is possible to get only simple singularities when we start with pre-simple ones in any dimension; in other words, we can globally obtain the diophantine condition that makes the difference between pre-simples and simples. This means that it is only necessary to obtain pre-simple points in  Theorem \ref{teo:mainmain}. We recall that $\omega$ is {\em pre-simple } with respect to the divisor $x_1x_2\cdots x_r=0$, if it can be written in a logarithmic way as $\omega=f\omega^*$, where
\[
\omega^*=
\sum_{i=1}^r a_i\frac{dx_i}{x_i}+\sum_{j=1}^{m-r}b_jdy_j
\]
and one of the following properties holds:
\begin{itemize}
\item[a)] There is a unit among the coefficients $a_1,a_2,\ldots,a_r$.
\item[b)] There is a unit among the coefficients $b_j$ or there is a coefficient $a_i$ whose linear part is not a linear combination of $x_1,x_2,\ldots,x_r$.
\end{itemize}
In case a), we have a {\em pre-simple corner}; otherwise, we  have a {\em pre-simple trace point.}

There are several results concerning the reduction of singularities of codimension one singular foliations and dynamical systems given by vector fields in dimension bigger that two. For vector fields over three dimensional ambient spaces: we have  Local Uniformization type results in \cite{Can1,Can-R-S}, a global reduction of singularities in the real case by Panazzolo \cite{Pan} and in terms of stacks and orbifolds, by Panazzolo-McQuillan \cite{Pan-McQ}. In general dimension, there are papers of Belotto \cite{Bel2,Bel3} where he performs reduction of singularities of ideals and varieties conditioned to simple foliations. Related problems are the monomialization of morphisms by Cutkosky \cite{Cut}, the reduction of first integrals of dynamical systems \cite{Bel1} or the works of Abramovich \cite{Abr}.  Some of the ``extra'' technical difficulties in that problems are of the same nature as some of those we encounter in the case of foliations. Surprisingly, this is also true with respect to the known results for the three dimensional case of schemes in positive or mixed characteristic \cite{Cos-P1, Cos-P2}.

In this paper, we consider $k$-rational  valuations of rank one. In classical Zariski's approach, this is the case where the problem is concentrated:  one can pass  to  the case of a general valuation, following for instance the paper \cite{Nov-S}. In the case or rational foliations, there are new difficulties when we consider general valuations, that we plan to solve in a forthcoming paper.

Anyway, the $k$-rational valuations of rank one have a geometrical and dynamical interpretation in the real case, that has been considered in \cite{Can-M-R}. The fact of being $k$-rational means that each time we blow-up, the center of the valuation is a $k$-rational point, hence a ``true'' point. In the real case, we have valuations of this type given by transcendental non oscillating curves, see \cite{Can-M-S1,Can-M-S2,Can-M-R}. The property of being of rank one means that we cannot decompose the valuation;  it may be interpreted by saying that the curve has no flat contact with any hypersurface.

Let us give a few comments on the technical structure of this paper. First of all, as it is ubiquitous in problems of local uniformization, for instance for the positive characteristic case \cite{Tei}, we have to deal with the possibility of having ``bad accumulation of values''. In zero characteristic, the classical Tschirnhausen transformation is usually the tool that allows us to avoid the accumulation of values.  A big part of the paper is devoted to facilitate the use of Tschirnhausen transformations in the case of integrable forms.

Instead of giving a direct proof of Theorem \ref{teo:mainmain}, we prove a statement of  Truncated Local Uniformization and we derive Theorem \ref{teo:mainmain} from it. This has several advantages. The first one is that we can deal with infinite value objects and in fact the result is valid for formal differential forms.  Moreover, the structure of our induction process is simplified, since it is possible to ``decompose'' a formal series with respect to groups of variables and then we can apply induction to the coefficients.

Let us explain the Truncated Local Uniformization statements.
Let $\Gamma$ be the value group of the valuation $\nu$ associated to $R$. We know that $\Gamma\subset{\mathbb R}$, since we are dealing with a rank one valuation. The truncated statements are relative to a ``truncation value'' $\gamma\in \Gamma$.

Recall that we are always working with respect to a list \[
\boldsymbol{x}=(x_1,x_2,\ldots,x_r)
\] of parameters that represents the prescribed divisor, that we call the {\em independent parameters}. We require  the values  $\nu(x_i)$ to be a ${\mathbb Q}$-basis of $\Gamma\otimes_{\mathbb Z}{\mathbb Q}$. The $\mathbb Q$-dimension $r$ of $\Gamma\otimes_{\mathbb Z}{\mathbb Q}$ is the so-called {\em rational rank of the valuation}. We complete $\boldsymbol{x}$ with a list of ``dependent parameters''
\[\boldsymbol{y}=(y_1,y_2,\ldots,y_{m-r})
\]
to obtain a regular system of parameters of the local ring ${\mathcal O}_{M,P}$ of $M$ at the center $P$ of the valuation (this will always be possible, up to blow-ups). In this way we have a {\em parameterized local model ${\mathcal A}=({\mathcal O}_{\mathcal A}, \boldsymbol{x}, \boldsymbol{y})$}, where ${\mathcal O}_{\mathcal A}={\mathcal O}_{M,P}$. Most of the technical results in the paper are stated in terms of parameterized local models.

Let us give the definition of {\em truncated $\gamma$-final} formal functions and  {\em truncated $\gamma$-final} formal differential $1$-forms.
Given a formal function
\[
f=\sum_{I}\boldsymbol{x}^{I}f_I(\boldsymbol{y}),
\]
we define the {\em explicit value} $\nu_{{\mathcal A}}(f)$ by $\nu_{{\mathcal A}}(f)=\min\{\nu(\boldsymbol{x}^{I});\, f_I(\boldsymbol{y})\ne 0\}$.  We say that {\em $f$ is $\gamma$-final} if one of the following properties holds:
\begin{itemize}
\item{\em Recessive case:} $\nu_{{\mathcal A}}(f)>\gamma$.
\item{\em Dominant case:} $\nu_{{\mathcal A}}(f)\leq\gamma$ and $f_{I_0}(0)\ne 0$, where $\nu(\boldsymbol{x}^{I_0})=\nu_{{\mathcal A}}(f)$.
\end{itemize}
The dominant case is very close to the definiton of corner singularity: in fact, by combinatorial blow-ups, concerning only the indepedent parameters $\boldsymbol{x}$, we can obtain the additional  property that $f=\boldsymbol{x}^{I_0}U$, where $U$ is a unit. Now, let us consider a formal $1$-differential form $\omega$,  that we write in a logarithmic way as
\[
\omega=\sum_{i=1}^rf_i\frac{dx_i}{x_i}+\sum_{j=1}^{m-r}g_jdy_j.
\]
The {\em explicit value $\nu_{{\mathcal A}}(\omega)$} is defined by
$
\nu_{{\mathcal A}}(\omega)=\min\{\{\nu_{{\mathcal A}}(f_i)\}_{i=1}^r,
\{\nu_{{\mathcal A}}(g_j)\}_{j=1}^{m-r}\}
$. We say that $\omega$ is {\em $\gamma$-final} if one of the following properties holds:
\begin{itemize}
\item{\em Recessive case:} $\nu_{{\mathcal A}}(\omega)>\gamma$.
\item{\em Dominant case:} $\nu_{{\mathcal A}}(\omega)\leq\gamma$ and there is a coefficient $f_i$ with $\nu_{{\mathcal A}}(f_i)=\nu_{{\mathcal A}}(\omega)$ such that $f_i$ is $\gamma$-final dominant.
\end{itemize}
The proof of the $\gamma$-truncated local uniformization goes by induction on the number $m-r$ of dependent variables $\boldsymbol{y}$. In particular, we need a truncated version of
 Frobenius integrability condition $\omega\wedge d\omega=0$, compatible with the induction procedure. We say that $\omega$ satisfies the {\em $\gamma$-truncated Frobenius integrability condition} if
\[
\nu_{{\mathcal A}}(\omega\wedge d\omega)\geq 2\gamma.
\]
Here we see one of the advantages of the case of a formal function $f$.  The differential $df$ trivially satisfies the integrability condition. Moreover, when we decompose $f$ as a power series in the last dependent variable, the coefficients are also formal functions that also satisfies the integrability condition. Note that Section 7 is devoted to the preparation procedure. It should not be necessary for the case of a formal  function, since in this case the preparation is a direct consequence of the induction hypothesis.

The Truncated Local Uniformization may be stated as follows:
\begin{theorem}
 \label{teo:truncatedmainmain}
 Consider a nonsingular algebraic variety $M$ over a characteristic zero field $k$ with field of rational functions $K=k(M)$, and let $R$ be a rank one $k$-rational valuation ring of $K$. Fix a value $\gamma$ in the value group of $R$.  For any formal differential $1$-form $\omega$ in the center $P$ of $R$ in $M$ satisfying the $\gamma$-truncated Frobenius integrability condition,
  there is a birational regular morphism $\pi:M'\rightarrow M$, that is a composition of blow-ups with non-singular centers, such that the transform $\pi^*\omega$ of $\omega$ is $\gamma$-final at the center $P'$ of $R$ in $M'$.
\end{theorem}
We first give a proof of the ``local statement''  Theorem \ref{teo:formalforms} in Section \ref{Truncated Local Uniformization}. To see that Theorem \ref{teo:formalforms} implies Theorem \ref{teo:truncatedmainmain}, it is enough to assure that the local centers of blow-ups can be made global non-singular ones, just by performing additional blowing-ups external to the centers of the valuation. This is a standard argument of reduction of singularities that we do not detail in the text. Theorem \ref{teo:formalforms} will be a consequence of Theorem \ref{teo:inductivestatement}, that is a stronger inductive version of it. 

 We devote Sections
\ref{Truncated Local Uniformization},
\ref{Statements of Strict Preparation},
\ref{Preparation Process} and
\ref{Control by the Critical Height} to the proof of the Theorem \ref{teo:inductivestatement}. In Section
\ref{Local Uniformization of  Rational Foliations},
we show how Theorem \ref{teo:inductivestatement} implies Theorem \ref{teo:mainmain}, with arguments directly related to the proof of  Theorem \ref{teo:inductivestatement}.

Let us give an idea of the structure of the proof of Theorem \ref{teo:inductivestatement}.
 We do induction on the number of dependent variables in $\boldsymbol{y}$ that appear in the expression of $\omega$. If  no dependent variables appear, we are done, since $\omega$ is automatically $\gamma$-final. We rename the dependent variables that appear in the expression of $\omega$ as
$
(\boldsymbol{y}_{\leq \ell},z)=(y_1,y_2,\ldots,y_\ell,z) 
$
and we assume that we have $\gamma$-truncated local uniformization when only $\ell$ dependent variables appear. Then, we make a decomposition of $\omega$ as
\[
\omega=\sum_{s\geq 0}z^s\omega_s,\quad \omega_s=\eta_s+h_s\frac{dz}{z},
\quad \eta_s=\sum_{i=1}^rf_{is}\frac{dx_i}{x_i}+\sum_{j=1}^\ell g_{js}dy_j,
\]
where $f_{is},g_{js},h_s\in k[[\boldsymbol{x}, \boldsymbol{y}]]$. Each $\omega_s$ is the $s$-level of $\omega$. Note that we could apply induction to $h_s$ and $\eta_s$, but for this we need to control the explicit value of $\eta_s\wedge d\eta_s$. Note also that $h_0=0$, since $\omega$ has no poles along $z=0$; this feature is interesting in  some of our preparation arguments. Now, we draw a Newton Polygon ${\mathcal N}_\omega\subset{\mathbb R}^2_{\geq 0}$ from the cloud of points
$
(s,\nu_{{\mathcal A}}(\omega_s))
$, for $s\geq 0$.
The importance of the Newton Polygon in the induction step is due to the following remark:

\vspace{5pt}
``If ${\mathcal N}_\omega$ has the only vertex $(\rho,0)$ and $\omega_0=\eta_0$ is $\gamma$-final, then $\omega$ is $\gamma$-final''.

\vspace{5pt}
\noindent
The objective is to obtain $\omega$ with the above property, or such that $\nu_{{\mathcal A}}(\omega)>\gamma$, after suitable transformations ${\mathcal A}\rightarrow{\mathcal A}'$ between parameterized local models.

 Let us say a word about the transformations we use. They are of three types:
 \begin{itemize}
 \item Blow-ups in the independent variables.
 \item Nested coordinate changes.
 \item Puiseux's packages.
 \end{itemize}
 
  The blow-ups in the independent variables are blow-ups with centers $x_i=x_j=0$ given by two independent variables. They are combinatorial along the valuation; we can use them, for instance, to principalize ideals given by monomials in the variables $\boldsymbol{x}$.

The nested coordinate changes do not affect to the ambient space. They have the form $z'=z+f$, where $f\in k[[\boldsymbol{x},\boldsymbol{y}]]\cap {\mathcal O}_{\mathcal A}$ with $\nu_{{\mathcal A}}(f)\geq \nu(z)$. These coordinate changes are necessary in order to avoid problems of accumulation of values.

 We have already considered {Puiseux's packages} in previous works \cite{Can-R-S}. They are related with Perron's transformations, the key polynomials of a valuation and binomial ideals, see for instance \cite{Dec-M-S, Tei, Zar1}. In dimension two, they are close to the Puiseux's pairs of a plane branch. We can understand them through the rational contact function $\Phi=z^d/\boldsymbol{x^p}$, which satisfies $\nu(\Phi)=0$. The Puiseux's package can be interpreted as a local uniformization  of the hypersurface
 \[
 z^d=\lambda\boldsymbol{x^p},
 \]
 where $\lambda\in k$ is such that $\nu(\Phi-\lambda)>0$. The new variable $z'$ is $z'=\Phi-\lambda$. The number $d>0$ is called the {\em ramification index}
 of the Puiseux's package. When $d=1$, the above hypersurface is non-singular and thae equations defining the transformation have an appropriate form;  in fact,  this is the case we encounter at the end of the proof.

 Thus, we try to control the evolution of the Newton polygon after performing the above transformations and taking into account the induction hypothesis. We proceed in two steps. First, we perform a {\em $\gamma$-preparation of $\omega$}.
Second, once $\omega$ is $\gamma$-prepared, we provide a control of the evolution of the {\em critical height $\chi_{\mathcal A}(\omega)$}, under Puiseux's packages and nested coordinate changes followed by $\gamma$-preparations.

 Roughly speaking, the $\gamma$-preparation of $\omega$ consists in obtaining a situation where the relevant levels are $\rho$-final, with respect to the abscissa $\rho$ determined by the polygon. That is, the ``important'' part of the level may be read in the coordinates $\boldsymbol{x}$. In order to obtain the $\gamma$-preparation, we need to apply induction to some of the forms $\eta_s$. The truncated integrability properties of $\eta_s$ do not allow to get a direct $\gamma$-preparation. We perform first an {\em approximate preparation}, thanks to the properties of truncated integrability of the differential parts $\eta_s$ of the  levels.
We complete the preparation, that we call {$\gamma$-strict preparation}, thanks to some additional properties that are consequence of the hypothesis of truncated integrability and De Rham-Saito type results of truncated division. This part of the proof goes along Sections
\ref{Truncated Local Uniformization},
\ref{Statements of Strict Preparation} and
\ref{Preparation Process}.

Once we have obtained a $\gamma$-preparation, we devote Section \ref{Control by the Critical Height} to the control of the behaviour of the Newton Polygon under  Puiseux's packages and nested coordinate changes, followed by new $\gamma$-preparations. Roughly speaking, we do what is ne\-ce\-ssa\-ry in order to obtain a situation modelled on the behaviour of Newton Polygon of plane branches, when we perform a sequence of blow-ups associated to a Puiseux's pair. The shape of Newton Polygon is described by means of numerical invariants. The most important for us is the {\em critical height $\chi_{\mathcal A}(\omega)$}. The {\em critical segment} is the vertex or segment of contact between the Newton Polygon and the lines of slope $-1/\nu(z)$. The {\em critical vertex} is the highest vertex in the critical segment, and $\chi_{\mathcal A}(\omega)$ is the ordinate of the critical vertex.

The new critical height, after a Puiseux's package, is lower or equal that the preceding one. If we are able to obtain strict inequalities, we are done. Thus, We have to look at the cases when the critical height stabilizes at $\chi=\chi_{{\mathcal A}'}(\omega)$ for any ``normalized transformation'' $({\mathcal A},\omega)\rightarrow ({\mathcal A}',\omega)$. The stabilization implies the presence of resonance conditions, that we call \textbf{r1}, \textbf{r2a} and \textbf{r2b}-$\boldsymbol{\upsilon}$. The resonance \textbf{r1} occurs only when $\chi=1$ and we can show that it happens ``at most once''. The resonances \textbf{r2} imply that the ramification index, after the necessary $\gamma$-preparations, is equal to one; this is a necessary property for our arguments. In the case when $\chi=2$, we arrive to show quite directly the existence of a Tschirnhausen coordinate change that allows us to ``cross the limit imposed by $\gamma$''. In the case $\chi=1$, we follow the same general ideas, but it is necessary for us to use truncated cohomological results, namely a generalized and truncated version of Poincaré's Lemma that we include in Section \ref{Truncated Cohomological Statements}.

{\em Acknowledgements: The authors are grateful with O. Piltant, M. Spivakovsky and B. Teissier  for many fruitful conversations on the subject. This work has been supported by the Spanish Research Project  MTM2016-77642-C2-1-P. }

\section{Formal Differential Forms in the Center of a Valuation}
 Let  $K/k$ be a field of rational functions over a base field $k$ of characteristic zero. Along all this paper, we consider a rational $k$-valuation ring $k\subset R\subset K$ of rank one.  That is, the following properties hold:
 \begin{itemize}
   \item The natural mapping $k\rightarrow \kappa$ over the residual field $\kappa$ of $R$ is an isomorphism. In particular, for any birational model $M$ of $K$ the center of $R$ in $M$  is a $k$-rational point of $M$.
   \item The value group $\Gamma=K^*/R^*$ of the valuation $\nu$ associated to $R$ is isomorphic to a subgroup of $({\mathbb R},+)$. Once for all, we fix an immersion $\Gamma\subset {\mathbb R}$.
 \end{itemize}
 \begin{definition}  {\em A locally parameterized model ${\mathcal A}$}  is given by ${\mathcal A}=({\mathcal O}_{\mathcal A},\boldsymbol{x},\boldsymbol{y})$, where ${\mathcal O}_{\mathcal A}$ is a regular local ring and
 \[
 (\boldsymbol{x},\boldsymbol{y})=(x_1,x_2,\ldots,x_r, y_{1},y_{2},\ldots,y_{m-r})
 \]
 is a regular system of parameters of ${\mathcal O}_{\mathcal A}$ satisfying the following properties:
 \begin{itemize}
 \item[1)] There is a projective model $M$ of $K$ such that ${\mathcal O}_{\mathcal A}={\mathcal O}_{M,P}$, where $P$ is the center of $R$ in $M$.
 \item[2)] The values $\nu(x_1),\nu(x_2),\ldots,\nu(x_r)$ give a ${\mathbb Q}$-basis of $\Gamma\otimes_{\mathbb Z}{\mathbb Q}$.
 \end{itemize}
The parameters $\boldsymbol{x}$ are {\em the independent parameters} and $\boldsymbol{y}$ are the {\em dependent parameters} of $\mathcal A$.
\end{definition}

\begin{remark} Locally parameterized models exist, see \cite{Can-R-S}.
\end{remark}

 Let us consider a locally parameterized model $\mathcal A=({\mathcal O}_{\mathcal A}, \boldsymbol{x},\boldsymbol{y})$.

 Denote by $
\Omega^1_{{\mathcal O}_{\mathcal A}/k}[\log \boldsymbol{x}]
$ the ${\mathcal O}_{\mathcal A}$-module of $\boldsymbol{x}$-logarithmic Kähler differentials, that is the logarithmic Kähler differentials with respect to the monomial $x_1x_2\cdots x_r$. We recall that
\[
\Omega^1_{{\mathcal O}_{\mathcal A}/k}[\log \boldsymbol{x}]\supset \Omega^1_{{\mathcal O}_{\mathcal A}/k}
\]
and both are free modules of finite rank $m=\dim M$, see for instance \cite{Eis}.
 Let $\Omega^1_{\mathcal A}$ denote the $\widehat{\mathcal O}_{\mathcal A}$-module
\[
\Omega^1_{\mathcal M}=\Omega^1_{{\mathcal O}_{\mathcal A}/k}[\log \boldsymbol{x}]
\otimes_{{\mathcal O}_{\mathcal A}} \widehat{\mathcal O}_{\mathcal A}\ ,
\]
where $\widehat{\mathcal O}_{\mathcal A}$ is the completion of ${\mathcal O}_{\mathcal A}$ with respect to its maximal ideal ${\mathfrak m}_{\mathcal A}$. The elements of $\Omega^1_{\mathcal A}$ are called {\em $\mathcal A$-formal logarithmic differential $1$-forms} or simply  {\em formal logarithmic differential $1$-forms}, it the reference to $\mathcal A$ is obvious.
\begin{remark} The modules $\Omega^1_{{\mathcal O}_{\mathcal A}/k}[\log \boldsymbol{x}]$ and $\Omega^1_{\mathcal A}$ depend only on the ideal $(\boldsymbol{x^1})$ generated by the monomial $\boldsymbol{x^1}=x_1x_2\cdots x_r$.
\end{remark}

For any $p\geq 0$, the free $\widehat{\mathcal O}_{\mathcal A}$-module $\Omega^p_{\mathcal A}$ is the $p$-th exterior power
\[
\Omega^p_{\mathcal A}=\Lambda^p \Omega^1_{\mathcal A}\quad
(\Omega^0_{\mathcal A}=\widehat{\mathcal O}_{\mathcal A}).
\]
The elements of $\Omega^p_{\mathcal A}$ are the {\em $\mathcal A$-formal logarithmic differential $p$-forms}. Let $\Omega^{\bullet}_{\mathcal A} $ denote  the direct sum $\oplus_{p \geq 0}\Omega^p_{\mathcal A}$. We have a well defined exterior derivative
\[
d: \Omega^p_{\mathcal A}\longrightarrow \Omega^{p+1}_{\mathcal A},
\]
as well as an exterior product $\alpha,\beta\mapsto \alpha\wedge\beta$
in $\Omega^{\bullet}_{\mathcal A}$.
\begin{remark}
 \label{rk:definicionofnuf}
 Note that the valuation $\nu$ is defined for the elements $f\in K\setminus \{0\}$. In particular, we have $\nu(f)<\infty$ for any $0\ne f\in {\mathcal O}_{\mathcal A}$, since ${\mathcal O}_{\mathcal A}\subset K$. For formal functions $f\in\Omega^0_{\mathcal A}\setminus {\mathcal O}_{\mathcal A}$, the value $\nu(f)$ is not defined.
\end{remark}
\section{Explicit Values and Final Truncated Differential Forms}
Let us consider a  parameterized local model ${\mathcal A}=({\mathcal O}_{\mathcal A};\boldsymbol{x},\boldsymbol{y})$.
\subsection{Explicit Values}
\label{subsec:explicitvalues}
Given  $\delta\in {\mathbb R}_{\geq 0}$, we define the ideals ${\mathcal I}^\delta_{\mathcal A}$ and ${\mathcal I}^{\delta+}_{\mathcal A}$ of ${\mathcal O}_{{\mathcal A}}$ by
\[
{\mathcal I}^\delta_{\mathcal A}=({\boldsymbol x}^I;\;  \nu({\boldsymbol x}^I)\geq \delta)
\quad\mbox{and}\quad
{\mathcal I}^{\delta+}_{\mathcal A}=({\boldsymbol x}^I;\;  \nu({\boldsymbol x}^I)>\delta),
\]
where ${\boldsymbol x}^I=x_1^{i_1}x_2^{i_2}\cdots x_r^{i_r}$, for $I=(i_1,i_2,\ldots,i_r)\in {\mathbb Z}_{\geq 0}^r$.
We obtain a filtration of ${\mathcal O}_{{\mathcal A}}$ that we call the {\em explicit filtration}. It depends only on the ideal  $(\boldsymbol{x^1})$. In the same way, the {\em explicit filtration} of $\widehat{\mathcal O}_{{\mathcal A}}=\Omega^0_{\mathcal A}$ is given by the family of ideals
\[
\widehat{\mathcal I}^\delta_{\mathcal A}={\mathcal I}^\delta_{\mathcal A}\Omega^0_{\mathcal A}\quad
\mbox{and}
\quad
\widehat{\mathcal I}^{\delta+}_{\mathcal A}={\mathcal I}^{\delta+}_{\mathcal A}\Omega^0_{\mathcal A}\ .
\]

\begin{remark} Let ${\mathcal V}_{\mathcal A}\subset \Gamma\subset {\mathbb R}$ be the set of values
$\nu({\boldsymbol x}^I)$ for $I\in {\mathbb Z}_{\geq 0}^r$. It is a well ordered subset of ${\mathbb R}_{\geq 0}$. In particular, for any $\delta\in {\mathbb R}_{\geq 0}$ there are unique $\delta_0,\delta_1\in {\mathcal V}_{\mathcal A}$ with $\delta\leq \delta_0\leq\delta_1$ such that
\[
{\mathcal I}^\delta_{\mathcal A}={\mathcal I}^{\delta_0}_{\mathcal A}\quad  \mbox{and}\quad
{\mathcal I}^{\delta+}_{\mathcal A}={\mathcal I}^{\delta_1}_{\mathcal A}.
\]
We have that $\delta=\delta_0<\delta_1$ if $\delta\in {\mathcal V}_{\mathcal A}$  and $\delta<\delta_0=\delta_1$ if $\delta\notin {\mathcal V}_{\mathcal A}$. Thus the family of ideals
$
\{{\mathcal I}^\delta_{\mathcal A}; \delta\in {\mathcal V}_{\mathcal A}\}
$
gives the explicit filtration. Let us also note that for any $0\ne f\in\Omega^0_{\mathcal A}$, we have
$\min\{\delta\in {\mathcal V}_{\mathcal A};\; f\notin \widehat{\mathcal I}^{\delta+}_{\mathcal A}\}=\max\{\delta\in {\mathcal V}_{\mathcal A};\; f\in \widehat{\mathcal I}^\delta_{\mathcal A}\}
$.
\end{remark}
\begin{definition} The {\em explicit value} $\nu_{\mathcal A}(f)$
of $f\in \Omega^0_{\mathcal A}$, with $f\ne 0$, is
\[
\nu_{\mathcal A}(f)=\min\{\delta\in {\mathcal V}_{\mathcal A};\; f\notin \widehat{\mathcal I}^{\delta+}_{\mathcal A}\}=\max\{\delta\in {\mathcal V}_{\mathcal A};\; f\in \widehat{\mathcal I}^\delta_{\mathcal A}\}.
\]
We put $\nu_{\mathcal A}(0)=\infty$ and   $
\nu_{\mathcal A}(S)=\min\{\nu_{\mathcal A}(f);\, f\in S\}
$, for any subset $S\subset \widehat{\mathcal O}_{{\mathcal A}}$.
\end{definition}
\begin{remark}Assume that ${\mathcal O}_{\mathcal A}={\mathcal O}_{M,P}$, where $M$ is a projective model of $K$ and $P$ is the center of $R$ in $M$.
Since $P$ is a $k$-rational point of $M$, there is a natural identification
$
\Omega^0_{{\mathcal A}}=k[[\boldsymbol{x},\boldsymbol{y}]]
$,
where $k[[\boldsymbol{x},\boldsymbol{y}]]$ stands for the formal series ring in the variables $\boldsymbol{x},\boldsymbol{y}$ with coefficients in $k$. Write $f\in \Omega^0_{\mathcal A}$ as a formal series
\[
f=\sum_{I}f_{I}\boldsymbol{x}^I, \quad  f_{I}\in k[[\boldsymbol{y}]].
\]
We have that $\nu_{\mathcal A}(f)=\min\{\nu(\boldsymbol{x}^I);\; f_I\ne 0\}$. For any $f\in {\mathcal O}_{{\mathcal A}}$ we have $\nu_{\mathcal A}(f)\leq \nu(f)$. However, $\nu(f)$ is not defined for $f\in \Omega^0_{{\mathcal A}}\setminus {\mathcal O}_{{\mathcal A}}$, see Remark  \ref{rk:definicionofnuf}.
\end{remark}

Let $N$ be a  free $\Omega^0_{{\mathcal A}}$-module of finite rank.
We extend the explicit filtration to $N$ by considering the family of submodules
$\widehat{\mathcal I}^\delta_{\mathcal A}N$
 and
$\widehat{\mathcal I}^{\delta+}_{\mathcal A}N$.
The {\em explicit value $\nu_{\mathcal A}(a)$ of $a\in N$} is
\[
\nu_{\mathcal A}(a)=\min\{\delta\in {\mathcal V}_{\mathcal A}; a\notin {\mathcal I}^{\delta+}_{\mathcal A}N\}.
\]
 If $e_1,e_2,\ldots, e_m$ is a $\Omega^0_{{\mathcal A}}$-basis of $N$ and
\[
a=f_1e_1+f_2e_2+\cdots+f_me_m,
\]
we have that $\nu_{\mathcal A}(a)=\nu_{\mathcal A}(\{f_1,f_2,\ldots,f_m\})=\nu_{\mathcal A}({\mathfrak a})$, where ${\mathfrak a}=\sum_{i=1}^mf_i\Omega^0_{{\mathcal A}}$. Note that the ideal ${\mathfrak a}\subset\Omega^0_{\mathcal A}$ does not depend on the choice of the basis.

From now on, we denote by $\nu_{\mathcal A}$ the explicit order in the free $\Omega^0_{{\mathcal A}}$-modules $\Omega^p_{\mathcal A}$.

We have the following standard valuative properties:
\begin{itemize}
\item[1)] For any $\alpha,\beta\in \Omega^p_{\mathcal A}$  we have $\nu_{\mathcal A}(\alpha+\beta)\geq \nu_{\mathcal A}(\alpha)+\nu_{\mathcal A}(\beta)$ and equality holds if $\nu_{\mathcal A}(\alpha)\ne\nu_{\mathcal A}(\beta)$.
\item[2)] For any $f\in \Omega^0_{{\mathcal A}}$ and $\alpha\in \Omega^p_{\mathcal A}$ we have $\nu_{\mathcal A}(f\alpha)=\nu_{\mathcal A}(f)+\nu_{\mathcal A}(\alpha)$.
\item[3)] For any $\alpha\in \Omega^p_{\mathcal A}$ and $\beta\in \Omega^q_{\mathcal A}$  we have $\nu_{\mathcal A}(\alpha\wedge\beta)\geq \nu_{\mathcal A}(\alpha)+\nu_{\mathcal A}(\beta)$.
\end{itemize}
Next, we give the definition of the main technical objects in this paper:
\begin{definition}
\label{def:gammatruncatedformalfoliatiedspace} Let ${\mathcal A}$ be a locally parameterized model. Consider $\omega\in \Omega^1_{\mathcal A}$ and a real number $\gamma\in {\mathbb R}$.
We say that $({\mathcal A},\omega)$ is a {\em $\gamma$-truncated formal foliated space} if and only if $\omega$ satisfies the $\gamma$-truncated integrability condition
\begin{equation*}
\nu_{\mathcal A}(\omega\wedge d\omega)\geq 2\gamma.
\end{equation*}
\end{definition}
A $\gamma$-truncated formal foliated space $({\mathcal A},\omega)$ is also $\gamma'$-truncated for any $\gamma'\leq \gamma$. If $\omega$ satisfies Frobenius integrability condition $\omega\wedge d\omega=0$, then $({\mathcal A},\omega)$ is a $\gamma$-truncated foliated space, for any $\gamma\in {\mathbb R}$; in this case, we say that $({\mathcal A},\omega)$ is a {\em formal foliated space}. In particular, we have a  formal foliated space $({\mathcal A},df)$ associated to a given formal function $f\in \Omega^0_{\mathcal A}$.

\subsection{Final truncated differential forms}
For any $\delta\in \mathcal{V}_{\mathcal A}$, let ${\mathcal G}_{\mathcal A}^\delta$ denote the quotient
$
{\mathcal G}_{\mathcal A}^\delta= \widehat{\mathcal I}^{\delta}_{\mathcal A}/ \widehat{\mathcal I}^{\delta+}_{\mathcal M}
$.
The graded algebra ${\mathcal G}_{\mathcal A}=\oplus_\delta {\mathcal G}^\delta_{\mathcal A}$,
associated to the explicit filtration of $\Omega^0_{\mathcal A}$, can be identified with a weighted polynomial algebra
\[
{\mathcal G}_{\mathcal A}\simeq \Omega^0_{{\mathcal A}} / (\boldsymbol{x}) [X_1,X_2,\ldots,X_r] \ ,
\]
where we attach the weight $\nu(x_i)$ to each variable $X_i$. Take $f\in \Omega^0_{\mathcal A}$. For any $\delta\leq \nu_{\mathcal A}(f)$, the {\em explicit $\delta$-initial form $\operatorname{In}^\delta_{\mathcal A}(f)$} is
\[
\operatorname{In}^\delta_{\mathcal A}(f)=f+ \widehat{\mathcal I}^{\delta+}_{\mathcal A} \in {\mathcal G}^\delta_{\mathcal A}.
\]
We have $\operatorname{In}^\delta_{\mathcal A}(f)\ne 0$ if and only if $\delta=\nu_{\mathcal A}(f)$.

Note that the $0$-degree ring ${\mathcal G}^0_{\mathcal A}=\Omega^0_{{\mathcal A}} / (\boldsymbol{x})$ of ${\mathcal G}_{\mathcal A}$ is a complete regular local ring with maximal ideal $\overline{\mathfrak{m}}_{{\mathcal A}} = \widehat{\mathfrak{m}}_{{\mathcal A}} / (\boldsymbol{x})$ and residual field naturally isomorphic to $k$. In particular, we have that
\[
{\mathcal G}_{\mathcal A}/\overline{\mathfrak{m}}_{{\mathcal A}} {\mathcal G}_{\mathcal A}\simeq k[X_1,X_2,\ldots,X_r].
\]
Let us introduce the definition of final formal functions.
\begin{definition}
\label{def:truncated_final_function}
Consider a real number $\gamma\in {\mathbb R}$ and a formal function $f\in \Omega^0_{\mathcal A}$. We say that $({\mathcal A},f)$ is {\em $\gamma$-final recessive} if $\delta>\gamma$, where $\delta=\nu_{\mathcal A}(f)$. We say that
$({\mathcal A},f)$ is {\em $\gamma$-final dominant} if $\delta\leq\gamma$ and we have that
\[
\operatorname{In}^\delta_{\mathcal A}(f)\notin \overline{\mathfrak m}_{\mathcal A}{\mathcal G}_{\mathcal A}.
\]
We say that $({\mathcal A},f)$ is {\em $\gamma$-final} if it is $\gamma$-final dominant or $\gamma$-final recessive.
\end{definition}
Let us introduce some notations to facilitate the generalization of the above Definition \ref{def:truncated_final_function} to higher order forms.  We denote by ${\mathcal C}_{\mathcal A}^0(f)$ the ${\mathcal G}^0_{\mathcal A}$-submodule of ${\mathcal G}^\delta_{\mathcal A}$ generated by $\operatorname{In}^\delta_{\mathcal A}(f)$, when $\delta=\nu_{\mathcal A}(f)$. We also denote by ${\mathfrak c}^0_{\mathcal A}(f)$ the image of ${\mathcal C}^0_{\mathcal A}(f)$ under the natural morphism
\begin{equation}
\label{eq:cornermorphism}
{\mathcal G}^\delta_{\mathcal A}\rightarrow {\mathcal G}^\delta_{\mathcal A}/ \overline{\mathfrak{m}}_{{\mathcal A}} {\mathcal G}^\delta_{\mathcal A}\subset
{\mathcal G}_{\mathcal A}/ \overline{\mathfrak{m}}_{{\mathcal A}} {\mathcal G}_{\mathcal A}=k[X_1,X_2,\ldots,X_r].
\end{equation}
Thus $({\mathcal A},f)$ is $\gamma$-final dominant if and only if $\delta\leq \gamma$ and ${\mathfrak c}^0_{\mathcal M}(f)\ne 0$.

\begin{remark} Consider a pair $({\mathcal A},f)$ where $0\ne f\in \Omega^0_{\mathcal A}=k[[\boldsymbol{x},\boldsymbol{y}]]$. Let $\delta$  denote the value $\nu_{\mathcal A}(f)$.
We can write $f$ as
\[
f= \boldsymbol{x}^{I_0}f_{I_0}+\sum_{\nu(\boldsymbol{x}^I)>\delta } \boldsymbol{x}^If_I,
\quad  \nu(\boldsymbol{x}^{I_0})=\delta,
\]
where $f_{I_0}, f_I \in k[[\boldsymbol{y}]]$.
The pair $({\mathcal A},f)$ is $\gamma$-final dominant if and only if $\delta\leq\gamma$ and $f_{I_0}$
is a unit in $k[[\boldsymbol{y}]]$.
\end{remark}

Consider an element $\omega \in \Omega^1_{\mathcal A}$ and write it as
\begin{equation*}
  \omega = \sum_{i=1}^rf_i\frac{dx_i}{x_i}+\sum_{j=1}^{m-r}g_jdy_j.
\end{equation*}
We know that
$
  \nu_{\mathcal A}(\omega)=\nu_{\mathcal A}(\{f_1,f_2,\ldots,f_r,g_1,g_2,\ldots,g_{m-r}\})
$.
If $\delta=\nu_{\mathcal A}(\omega)$, we denote by
  ${\mathcal C}^1_{\mathcal A}(\omega)$ the ${\mathcal G}^0_{\mathcal A}$-submodule of ${\mathcal G}^\delta_{\mathcal A}$ generated by the initial forms
  \[
  \operatorname{In}_{\mathcal A}^\delta(f_i)\ ,\quad i=1,2,\ldots,r.
  \]
   Let ${\mathfrak c}^1_{\mathcal A}(\omega)$ be the image of ${\mathcal C}^1_{\mathcal A}(\omega)$ under the natural morphism in Equation \eqref{eq:cornermorphism}.
We
extend Definition \ref{def:truncated_final_function} to formal differential $1$-forms
$\omega\in\Omega^1_{\mathcal A}$ as follows:
  \begin{definition}
   Fix $\gamma \in \Gamma$. Take $\omega\in\Omega^1_{\mathcal A}$ with $\nu_{\mathcal A}(\omega)=\delta$. We say that $({\mathcal A},\omega)$ is
  \begin{itemize}
    \item {\em $\gamma$-final dominant}, if $\delta\leq \gamma$ and
    ${\mathfrak c}_{\mathcal A}^1(\omega)\ne 0$.
    \item {\em $\gamma$-final recessive}, if $\delta>\gamma$.
  \end{itemize}
   We say that $({\mathcal A},\omega)$ is {\em $\gamma$-final} if it is $\gamma$-final dominant or $\gamma$-final recessive.
  \end{definition}

\subsection{Explicit values under Differentiation}  We compare here the explicit value of $d\alpha\in \Omega^{p+1}_{\mathcal A}$ with the explicit value of $\alpha\in \Omega^{p}_{\mathcal A}$.

\begin{lemma}
\label{lema:explicitvalueofaderivative}
Let $\partial:\Omega^0_{\mathcal A}\rightarrow \Omega^0_{\mathcal A}$ be a $k$-derivation of the form
\[
\partial=\sum_{i=1}^ra_ix_i\frac{\partial}{\partial x_i}+\sum_{j=1}^{m-r}b_j\frac{\partial}{\partial y_j}, \quad a_i, b_j\in\Omega^0_{\mathcal A}.
\]
For any $f\in \Omega^0_{\mathcal A}$, we have
$\nu_{\mathcal A}(\partial f)\geq \nu_{\mathcal A}(f)$.
\end{lemma}
\begin{proof}Write $\delta=\nu_{\mathcal A}(f)$. Then
$f=\sum_{\nu(\boldsymbol{x}^I) \geq \delta}\boldsymbol{x}^If_I$ and
$\partial f=\sum_{\nu(\boldsymbol{x}^I) \geq \delta}\boldsymbol{x}^Ig_I$, where
$
g_I=\partial f_I+ f_I\sum_{j=1}^ri_ja_j$ and $I=(i_1,i_2,\ldots,i_r)
$.
Thus, we have that $\nu_{\mathcal A}(\partial f)\geq \delta$.
\end{proof}

\begin{proposition}
\label{prop:explicitvalueofthedifferential}
Consider $f\in \Omega^0_{\mathcal A}$ and $df\in \Omega^1_{\mathcal A}$.
 We have
$\nu_{\mathcal A}(df)\geq \nu_{\mathcal A}(f)$.
In the case that $f\in {\mathfrak{m}}_{{\mathcal A}}{\Omega^0_{{\mathcal A}}}$, then $\nu_{\mathcal A}(df)=\nu_{\mathcal A}(f)$. More generally, for any $\alpha\in \Omega^p_{{\mathcal A}}$ we have that $\nu_{\mathcal A}(d\alpha)\geq \nu_{\mathcal A}(\alpha)$.
\end{proposition}
\begin{proof} The first statement is a consequence of Lemma \ref{lema:explicitvalueofaderivative}, since
\[
df=\sum_{i=1}^r x_i\frac{\partial f}{\partial x_i}\frac{dx_i}{x_i}+\sum_{j=1}^{m-r}\frac{\partial f}{\partial y_j} dy_j \ .
\]
Write $\delta=\nu_{\mathcal A}(f)$ and $f=\sum_{\nu(\boldsymbol{x}^I) \geq \delta}\boldsymbol{x}^If_I$. Assume that $0\ne f\in {\mathfrak{m}}_{{\mathcal A}}{\Omega^0_{{\mathcal A}}}$. If $\delta=0$, we can write $f=f_0+\tilde f$, where $\nu_{\mathcal A}(\tilde f)>0$ and $0\ne f_0\in {\mathfrak{m}}_{{\mathcal A}}{\Omega^0_{{\mathcal A}}}$ has a series expansion
\[
f_0=\sum_{\vert J\vert>0}\lambda_J\boldsymbol{y}^{J},\quad  \lambda_J\in k.
\]
Then $df_0\ne 0$ and $\nu_{\mathcal A}(df_0)=0$. Since $df=df_0+d\tilde f$ and $\nu_{\mathcal A}(d\tilde f)>0$, we conclude that $\nu_{\mathcal A}(df)=0$.
If $\delta>0$, there is $I_0\ne 0$ such that $\nu(\boldsymbol{x}^{ I_0})=\delta$ and $\nu_{\mathcal A}(f_{I_0})=0$. We write
\[
f_{I_0}=g_0+\tilde g_0,\quad \nu_{\mathcal A}(\tilde g_0)>0,\quad 0\ne g_0\in k[[\boldsymbol{y}]].
\]
By the previous arguments, it is enough to see that $\nu_{\mathcal A}(d(x^{I_0}g_0))=\delta$. Let us write $I_0=(i^0_1,i^0_2,\ldots,i^0_r)\ne 0$. We have
\[
d(x^{I_0}g_0)=x^{I_0} \left(
g_0 \sum_{s=1}^r i^0_s\frac{dx_s}{x_s}+dg_0
\right)= x^{I_0}\left(
g_0\sum_{s=1}^r i^0_s\frac{dx_s}{x_s}+\sum_{j=1}^{m-r}\frac{\partial g_0}{\partial y_j}dy_j
\right).
\]
Then $\nu_{\mathcal A}(d(x^{I_0}g_0))=\delta$.
The last assumption is direct from the formulas for $d\alpha$.
\end{proof}

\begin{corollary}
\label{cor:gammafinalfdf}
For any $f\in {\mathfrak{m}}_{{\mathcal A}}{\Omega^0_{{\mathcal A}}}$, with $f\ne 0$,
 we have that ${\mathfrak c}^0_{\mathcal A}(f)\ne 0$ if and only if
 ${\mathfrak c}^1_{\mathcal A}(df)\ne 0$. Thus $({\mathcal A}, f)$ is $\gamma$-final if and only if $({\mathcal A}, df)$ is $\gamma$-final.
\end{corollary}
\begin{proof} Put $\delta=\nu_{\mathcal A}(f)=\nu_{\mathcal A}(df)$. We can write
\[
f=\boldsymbol{x}^{I}f_I+\tilde f;\quad \nu(\boldsymbol{x}^{I})=\delta,\quad \nu_{\mathcal A}(\tilde f)>\delta,\quad  f_I\in {\mathbb C}[[\boldsymbol{y}]].
\]
Then ${\mathfrak c}^0_{\mathcal A}(f)\ne 0$ is equivalent to saying that $\lambda\ne 0$, where $f_I=\lambda+\sum_{\vert J\vert>0}f_{I,J}\boldsymbol{y}^J$.

On the other hand, we have
$
df=f_Id(\boldsymbol{x}^I)+x^{I}df_I+d\tilde f
$. Since $\nu_{\mathcal A}(d\tilde f)>\delta$ and $d(\boldsymbol{x}^{I})= \boldsymbol{x}^{I}\sum_{j=1}^ri_jdx_j/x_j$ with $I\ne 0$, we have that ${\mathfrak c}^1_{\mathcal M}(df)\ne 0$ is also equivalent to $\lambda\ne 0$.
\end{proof}

\section{Truncated Cohomological Statements}\label{Truncated Cohomological Statements}

Next result is a truncated version of De Rham-Saito division \cite{Sai}.
\begin{proposition}
	\label{prop:trucateddivision}
	Let $\alpha\in\Omega^1_{\mathcal A}$ be a $0$-final dominant $ 1 $-form. For any $\beta\in\Omega^1_{\mathcal A}$ with $\nu_{\mathcal A}(\alpha\wedge\beta)\geq\rho$, there are $H\in \Omega^0_{\mathcal A}$ and $\tilde\beta\in \Omega^1_{\mathcal A}$, with $\nu_{\mathcal A}(\tilde\beta)\geq\rho$, such that $\beta=H\alpha+\tilde\beta$.
\end{proposition}
\begin{proof} Let us write
	\[
	\alpha = \sum_{i=1}^r a_i\frac{dx_i}{x_i}+\sum_{j=1}^{m-r}b_jdy_j \quad \text{and} \quad \beta = \sum_{i=1}^r f_i\frac{dx_i}{x_i}+\sum_{j=1}^{m-r}g_jdy_j \ .
	\]
	Since $ \alpha $ is $ 0 $-final dominant,  there is at least an unit between the coefficients $a_1, a_2,\dots,a_r $. Without lost of generality, we can assume that $ a_1 $ is a unit. Since  $\nu_{\mathcal A}(\alpha\wedge\beta)\geq\rho$ we have that
	\begin{align*}
		\nu_{\mathcal A}(a_1f_i - a_if_1)\geq\rho & \quad \text{for all} \quad i=2,3,\dots,r \ , \\
		\nu_{\mathcal A}(a_1g_j - b_jf_1)\geq\rho & \quad \text{for all} \quad j=1,2,\dots,m-r \ .
	\end{align*}
	The condition $ \nu_{\mathcal A}(a_1f_i - a_if_1)\geq\rho $ is equivalent to $ \nu_{\mathcal A}(f_i - H a_i)\geq\rho $ where \[ H=f_1/a_1.\] So, there are functions $ \tilde{f}_i $ with $ \nu_{\mathcal A}(\tilde{f}_i)\geq \rho $ such that $ f_i = H a_i + \tilde{f}_i $ for $ i=2,3,\dots,r $. In the same way, there are functions $ \tilde{g}_j $ with $ \nu_{\mathcal A}(\tilde{g}_j)\geq \rho $, such that $ g_j = H b_j + \tilde{g}_j $, for $ j=1,2,\dots,n-r $. The $ 1 $-form
	\[
	\tilde{\beta} = \sum_{i=2}^r \tilde{f}_i\frac{dx_i}{x_i}+\sum_{j=1}^{m-r}\tilde{g}_j dy_j
	\]
	satisfies that $ \nu_{\mathcal A}(\tilde{\beta}) \geq \rho $ and
	$ \beta = H \alpha + \tilde{\beta} $.
\end{proof}

Given a $p$-form $\xi\in \Omega^p_{\mathcal A}$ and a vector
$
\boldsymbol{\mu}=(\mu_1,\mu_2,\ldots,\mu_r)\in {\mathbb C}^r
$, we define the {\em $\boldsymbol{\mu}$-exterior derivative $d_{\boldsymbol{\mu}}\xi\in \Omega^{p+1}_{\mathcal A}$} by the formula
\[
d_{\boldsymbol{\mu}}\xi=
\frac{d{\boldsymbol{x}}^{\boldsymbol{\mu}}}{{\boldsymbol{x}}^{\boldsymbol{\mu}}}\wedge \xi+d(\xi),\quad  (\text{ where }
\frac{d{\boldsymbol{x}}^{\boldsymbol{\mu}}}{{\boldsymbol{x}}^{\boldsymbol{\mu}}}
=\sum_{i=1}^r\mu_i\frac{dx_i}{x_i}
).
\]
In the case $\boldsymbol{\mu}=\boldsymbol{0}$ we recover the usual exterior derivative $d_{\boldsymbol{0}}=d$.
\begin{remark} The definition of $d_{\boldsymbol{\mu}}$ comes from the formal relation
	$
	d({\boldsymbol{x}}^{\boldsymbol{\mu}}\xi)= {\boldsymbol{x}}^{\boldsymbol{\mu}}
	d_{\boldsymbol{\mu}}(\xi)
	$.
	Note that ${\boldsymbol{x}}^{-\boldsymbol{\mu}}$ is an integrant factor of $\xi$ if and only if $d_{\boldsymbol{\mu}}(\xi)=0$, see \cite{Cer-M}.
\end{remark}

Let us consider a ``value of truncation'' $\rho\in{\mathbb R}\cup \{+\infty\}$, where the case $\rho=+\infty$ means ``no truncation''. In Proposition \ref{prop:logpoincare2} below
we state and prove a truncated $\boldsymbol{\mu}$-multivaluated and logarithmic generalization of formal Poincaré's Lemma.

Consider a $1$-form $\eta\in \Omega^1_{\mathcal A}$ and write it as $ \eta=\sum_{I}\boldsymbol{x}^I\eta_I $, where
\begin{equation*}
\eta_I=\sum_{i=1}^ra_I\frac{dx_i}{x_i}+
\sum_{j=1}^{n-r}b_Idy_j;\quad a_I,b_I\in k[[\boldsymbol{y}]].
\end{equation*}
We have a splitting $\eta=\eta'+\eta''$ with
\begin{equation*}
\eta'=\sum_{\nu(\boldsymbol{x}^I)<\rho}\boldsymbol{x}^I\eta_I;\quad
\eta''=\sum_{\nu(\boldsymbol{x}^I)\geq \rho}\boldsymbol{x}^I\eta_I.
\end{equation*}
In the case $\rho=\infty$, we put $\eta'=\eta$ and $\eta''=0$.
By Proposition \ref{prop:explicitvalueofthedifferential},
we know that $\nu_{\mathcal A}(\eta'')\geq \rho$; we also have that
$\nu_{\mathcal A}(d\eta'')\geq \rho$.

\begin{proposition}
	\label{prop:logpoincare2}
	Assume that
	$\nu_{\mathcal A}(d_{\boldsymbol{\mu}}\eta)\geq \rho$.
	There are $f\in {\Omega}^0_{{\mathcal A}}$, residues $\boldsymbol{\lambda}\in k^r$, and $\tilde \eta\in {\Omega}^1_{\mathcal A}$ such that
	\begin{equation*}
	\eta=d_{\boldsymbol{\mu}}f+
	\boldsymbol{x}^{-\boldsymbol{\mu}} \frac{d \boldsymbol{x}^{\boldsymbol{\lambda}}}{\boldsymbol{x}^{\boldsymbol{\lambda}}}+\tilde \eta,\quad \nu_{\mathcal A}(\tilde \eta)\geq \rho,
	\end{equation*}
	where $\boldsymbol{\lambda}=\boldsymbol{0}$  if $-\boldsymbol{\mu}\notin {\mathbb Z}_{\geq 0}^r$.
\end{proposition}
\begin{proof}  Let us note that $d_{\boldsymbol{\mu}}(\eta)=d_{\boldsymbol{\mu}}(\eta')+d_{\boldsymbol{\mu}}(\eta'')$ and $\nu_{\mathcal A}(d_{\boldsymbol{\mu}}(\eta''))\geq \rho$. Hence $\nu_{\mathcal A}(d_{\boldsymbol{\mu}}(\eta'))\geq \rho$. Moreover,
	\[
	d_{\boldsymbol{\mu}}(\eta')=
	\sum_{\nu(\boldsymbol{x}^I)<\rho}d_{\boldsymbol{\mu}}(\boldsymbol{x}^I\eta_I)
	= \sum_{\nu(\boldsymbol{x}^I)<\rho}\boldsymbol{x}^I d_{I+\boldsymbol{\mu}}(\eta_I)  .
	\]
	Since the coefficients of $d_{\boldsymbol{I+\mu}}(\eta_I)$ belong to $k[[\boldsymbol{y}]]$, the fact that $\nu_{\mathcal A}(d_{\boldsymbol{\mu}}(\eta'))\geq \rho$
	implies that $d_{\boldsymbol{\mu}}(\eta')=0$. Moreover, since $\nu_{\mathcal A}(d_{\boldsymbol{\mu}}(\eta''))\geq \rho$, it is enough to show that there are $\lambda_i$ and $f$ such that
	\[
	\eta'=d_{\boldsymbol{\mu}}f+\frac{d \boldsymbol{x}^{\boldsymbol{\lambda}}}{\boldsymbol{x}^{\boldsymbol{\lambda}}} .
	\]
	In other words, we have reduced the problem to the case $\rho=\infty$.
	
	We know that $d_{I+\boldsymbol{\mu}}(\eta_I)=0$ for any $I$ such that
	$\nu(\boldsymbol{x}^I)<\rho$. Let us consider one such $I$ and put
	$\boldsymbol{\rho}=I+\boldsymbol{\mu}$.  Let us split $\eta_I=\tilde\eta_I+\eta^*_I$
	where
	\[
	\tilde\eta_I=\sum_{i=1}^rf_i\frac{dx_i}{x_i},\quad \eta_I^*=\sum_{j=1}^{n-r}g_jdy_j,
	\]
	where $f_i,g_j\in k[[\boldsymbol{y}]]$ for any $i,j$. Note that
	\begin{equation}
	\label{eq:rohomologica}
	0=d_{\boldsymbol{\rho}}\eta_I=
	\frac{d \boldsymbol{x}^{\boldsymbol{\rho}}}{{\boldsymbol{x}}^{\boldsymbol{\rho}}}\wedge(\tilde\eta_I
	+\eta^*_I)+d\tilde\eta_I+d\eta_I^*
	\end{equation}
	implies that $d\eta^*_I=0$, since the coefficients of $d_{\boldsymbol{\rho}}\eta_I$ for $dy_j\wedge dy_s$ correspond exactly to the coefficients of $d\eta^*_I$. By applying the standard formal Poincaré's Lemma, we find $F^*\in k[[y]]$ such that $dF^*=\eta^*_I$. Now, we have two cases to consider:
	
	{\em Case $\boldsymbol{\rho}\neq  \boldsymbol{0}$:} By Equation \eqref{eq:rohomologica}, we deduce that
	\[
	\frac{d \boldsymbol{x}^{\boldsymbol{\rho}}}{{\boldsymbol{x}}^{\boldsymbol{\rho}}}
	\wedge\tilde\eta_I=0,
	\]
	just by looking to the coefficients of $d_{\boldsymbol{\rho}}\eta_I$ of the terms in $(dx_i/x_i)\wedge(dx_\ell/x_\ell)$. Since $\boldsymbol{\rho}\ne 0$,the above proportionality implies that there is $f_I\in k[[\boldsymbol{y}]]$ such that
	\[
	\tilde\eta_I=f_I\frac{d \boldsymbol{x}^{\boldsymbol{\rho}} }{{\boldsymbol{x}}^{\boldsymbol{\rho}}}.
	\]
	Looking in  Equation \eqref{eq:rohomologica} to the coefficients of $(dx_i/x_i)\wedge dy_j$ we find that
	\[
	\frac{d \boldsymbol{x}^{\boldsymbol{\rho}}}{{\boldsymbol{x}}^{\boldsymbol{\rho}}}\wedge
	d(F^*-f_I)=0.
	\]
	We conclude that $dF^*=df_I$, since $F^*,f_I\in k[[\boldsymbol{y}]]$.  Finally, in this case we have that $\eta_I=d_{\boldsymbol{\rho}}(f_I)$ and hence
	\begin{equation*}
	\boldsymbol{x}^I\eta_I=d_{\boldsymbol{\mu}}(\boldsymbol{x}^If_I).
	\end{equation*}
	
	{\em Case $\boldsymbol{\mu}=-I$:} We have that $d\eta_I=d\tilde\eta_I+d\eta^*_I=0$. Recalling that $d\eta^*_I=0$, we conclude that $d\tilde\eta_I=0$. This implies that $df_i=0$ for any $i=1,2,\ldots,r$ and hence
	$f_i=\lambda_i\in k$ for any $i=1,2,\ldots,r$. That is
	\[
	\tilde\eta_I=\frac{d \boldsymbol{x}^{\boldsymbol{\lambda}}}{\boldsymbol{x}^{\boldsymbol{\lambda}}} ;\quad \eta^*_I=df_I,
	\]
	where $f_I=F^*$ in this case. More precisely, we have
	\begin{equation*}
	\boldsymbol{x}^I\eta_I= \boldsymbol{x}^I \frac{d \boldsymbol{x}^{\boldsymbol{\lambda}}}{\boldsymbol{x}^{\boldsymbol{\lambda}}}+ d_{\boldsymbol{\mu}}(\boldsymbol{x}^If_I) ,
	\end{equation*}
	where we recall that $ \boldsymbol{\mu}=-I $. Finally, let us put $f=\sum_{
		\nu_{\mathcal A}(\boldsymbol{x}^I)<\rho
	}\boldsymbol{x}^If_I$. We have
	\[
	\eta'= d_{\boldsymbol{\mu}}(f)+
	\boldsymbol{x}^{-\boldsymbol{\mu}} \frac{d \boldsymbol{x}^{\boldsymbol{\lambda}}}{\boldsymbol{x}^{\boldsymbol{\lambda}}},
	\]
	where $\boldsymbol{\lambda}=\boldsymbol{0}$ if $-\boldsymbol{\mu}\notin {\mathbb Z}_{\geq 0}^r$ (or $\nu_{\mathcal A}(\boldsymbol{x}^{-\boldsymbol{\mu}})\geq \rho$). This ends the proof.
\end{proof}

\begin{remark}
	\label{rk:integrabilidadpoincare}
	Let us note that
	$\theta=d_{\boldsymbol{\mu}}f+
	\boldsymbol{x}^{-\boldsymbol{\mu}} d \boldsymbol{x}^{\boldsymbol{\lambda}} / \boldsymbol{x}^{\boldsymbol{\lambda}}$ satisfies Frobenius' integrability condition $\theta\wedge d\theta=0$.
\end{remark}

In a more general way, let us consider a $0$-final dominant $\alpha\in \Omega^1_{\mathcal A}$ such that $d\alpha=0$. For any $\xi\in \Omega^p_{\mathcal A}$ we define {\em the $\alpha$-exterior derivative $d_\alpha(\xi)$} by
\begin{equation*}
d_\alpha(\xi)=\alpha\wedge\xi+d\xi.
\end{equation*}
By Proposition \ref{prop:logpoincare2}, the fact that $d\alpha=0$ implies that there is $h\in {\Omega}^0_{{\mathcal A}}$ and $\boldsymbol{\mu}\in {\mathbb C}^r$ such that
\[
\alpha=\frac{
	d\boldsymbol{x}^{\boldsymbol{\mu}}
}
{
	\boldsymbol{x}^{\boldsymbol{\mu}}
}+dh ,
\]
where $\boldsymbol{\mu}\ne\boldsymbol{0}$ since $\alpha$ is $0$-final dominant. We call $d\boldsymbol{x^\mu}/\boldsymbol{x^\mu}$ the {\em residual part of $\alpha$} in $\mathcal A$.
Take an index $i_0\in\{1,2,\ldots,r\}$ such that $\mu_{i_0}\ne 0$ and consider the unit $H=\exp(h/\mu_{i_0})\in \widehat{\mathcal O}_{{\mathcal A}}$. Let us put
\[
\boldsymbol{\tilde \mu}=\frac{\boldsymbol{\mu}}{\mu_{i_0}} , \quad  \tilde x_{i_0}=Hx_{i_0} \quad \text{and} \quad  \tilde x_i=x_i \quad \text{for} \quad i\ne i_0.
\]
We have
\[
\frac{\alpha}{\mu_{i_0}}=\frac{d\boldsymbol{x}^{\boldsymbol{\tilde \mu}}}{\boldsymbol{x}^{\boldsymbol{\tilde \mu}}}+\frac{dH}{H}= \frac{d\boldsymbol{\tilde x}^{\boldsymbol{\tilde \mu}}}{\boldsymbol{\tilde x}^{\boldsymbol{\tilde \mu}}}
\Longrightarrow \alpha= \frac{d\boldsymbol{\tilde x}^{\boldsymbol{\mu}}}{\boldsymbol{\tilde x}^{\boldsymbol{\mu}}}.
\]

\begin{corollary}
	\label{cor:logpoincare3}
	Let $\alpha\in \Omega^1_{\mathcal A}$ be $0$-final dominant with $d\alpha=0$ and residual part $d\boldsymbol{x^\mu}/\boldsymbol{x^\mu}$. Let $\eta\in \Omega^1_{\mathcal A}$ be such that
	$\nu_{\mathcal A}(d_{\alpha}\eta)\geq \rho$. We can decompose $\eta$ as
	$
	\eta=\theta+\tilde\eta
	$,
	where $ \nu_{\mathcal A}(\tilde \eta)\geq \rho$ and there is a formal function $F\in \Omega^0_{\mathcal A}$ such that $\theta$ has the form
	\begin{equation*}
	\theta=d_{\alpha}F+\boldsymbol{{\tilde x}^{-\mu}}
	\frac{d \boldsymbol{x}^{\boldsymbol{\lambda}}}{\boldsymbol{x}^{\boldsymbol{\lambda}}} ,
	\end{equation*}
	where there is an index $i_0$ with $\mu_{i_0}\ne 0$ and $\tilde x_{i_0}=Wx_{i_0}$, for a formal unit $W$ and $\tilde x_i=x_i$ for $i\ne i_0$. Moreover, in the case that $-\boldsymbol{\mu}\notin {\mathbb Z}_{\geq 0}$ we have
	$\lambda_i=0$, for $i=1,2,\ldots,r$.
	In particular, we have $\theta\wedge d\theta=0$.
\end{corollary}
\begin{proof} Follows from
	Proposition \ref{prop:logpoincare2}, noting that
	$
	\alpha= {d\boldsymbol{\tilde x}^{\boldsymbol{\mu}}}/{\boldsymbol{\tilde x}^{\boldsymbol{\mu}}}
	$.
	That is, the derivative $d_\alpha$ is expressed in the new independent coordinates $\boldsymbol{\tilde x}$ as $d_{\boldsymbol{\tilde\mu}}$ and moreover,  we recall that $\nu_{\mathcal A}$ is independent of the coordinate change $\boldsymbol{x}\mapsto \boldsymbol{\tilde x}$.
\end{proof}
\section{Transformations of Locally Parameterized Models}
\label{sec:Transformations of Locally Parameterized Models}
The {\em allowed transformations}
\[
{\mathcal A}=({\mathcal O}_{\mathcal A};\boldsymbol{x},\boldsymbol{y})
\rightarrow{\mathcal A}'=({\mathcal O}_{{\mathcal A}'};\boldsymbol{x}',\boldsymbol{y}'),
\]
between  locally parameterized models are finite compositions of the following types of {\em elementary allowed transformations}:
\begin{itemize}
\item {\em $\ell$-coordinate changes}.
\item {\em Independent blow-ups of locally parameterized models}.
\item{\em $\ell$-Puiseux's packages of locally parameterized models}.
\end{itemize}
\subsection{Elementary allowed transformations}
\label{Elementary allowed transformations}
Let us describe the three types of elementary allowed transformations:
\vspace{3pt}
\paragraph{ $\bullet$ \em Coordinate changes}  Consider an integer $1\leq \ell\leq m-r$ and an element
\[
f\in k[[\boldsymbol{x},y_1,y_2,\ldots,y_{\ell-1}]]\cap {\mathcal O}_{M,P},
\]
such that $\nu_{\mathcal A}(f)\geq \nu(y_\ell)$. Let us put
\[
y'_\ell=y_\ell+f,\quad y'_j=y_j,\, j\in \{1,2,\ldots,m-r\}\setminus \{\ell\}.
\]
 We obtain a new locally parameterized model ${\mathcal A}'=({\mathcal O}_{{\mathcal A}'}; \boldsymbol{x}',\boldsymbol{y}')$, where $\boldsymbol{x}'=\boldsymbol{x}$ and ${\mathcal O}_{{\mathcal A}'}={\mathcal O}_{{\mathcal A}}$.
Any transformation ${\mathcal A}\rightarrow {\mathcal A}'$ as above is called an {\em $\ell$-coordinate change}, or a {\em coordinate change in the dependent variable $y_\ell$}.
\vspace{3pt}
\paragraph{$\bullet$ \em Independent Blow-ups} Consider a set $\{x_i,x_j\}$ of two independent variables. Let us note that $\nu(x_i)\ne\nu(x_j)$. Assume that $\nu(x_i)<\nu(x_j)$. Let us select a projective model $M$ of $K$ associated to ${\mathcal A}$ in the sense that ${\mathcal O}_{\mathcal A}={\mathcal O}_{M,P}$, where $P$ is the center of $R$ in $M$. Denote $Y\subset M$ the irreducible subvariety of $M$ defined locally a $P$ by the equations $x_i=x_j=0$. Let us consider the blow-up $\pi:M'\rightarrow M$ with center $Y$ and let $P'\in M'$ be the center of $R$ in $M'$. We obtain a locally parameterized model ${\mathcal A}'$ associated to $M'$, where ${\mathcal O}_{{\mathcal A}'}={\mathcal O}_{M',P'}$, $\boldsymbol{y}'=\boldsymbol{y}$ and
\[
x'_j=x_j/x_i, \quad x'_s=x_s,\; s\in \{1,2,\ldots,r\}\setminus\{j\}.
\]
The transformation ${\mathcal A}\rightarrow {\mathcal A}'$ is called {\em the independent blow-up of ${\mathcal A}$ with center in the coordinates $\{x_i,x_j\}$}.

\vspace{3pt}
\paragraph{$\bullet$ \em Puiseux`s Packages} Let us consider an integer number $1\leq \ell\leq m-r$. An {\em $\ell$-Puiseux's package} ${\mathcal A}\rightarrow {\mathcal A}'$ is the composition of two transformations
\[
{\mathcal A}\rightarrow\widetilde{\mathcal A}\rightarrow {\mathcal A}'
\]
where $\widetilde{\mathcal A}\rightarrow {\mathcal A}'$ is an {\em $\ell$-blow-up with translation} and ${\mathcal A}\rightarrow\widetilde{\mathcal A}$ is a finite composition of independent blow-ups and {\em $\ell$-combinatorial blow-ups}. Let us define these types of blow-ups:
\begin{itemize}
\item[1)]  An {\em $\ell$-combinatorial blow-up ${\mathcal A}\rightarrow{\mathcal A}^{*}$} is given by the choice of an $x_i$ such that $\nu(y_\ell)\ne\nu(x_i)$. As before, we take a projective model $M$ associated to $\mathcal A$ and the blow-up $\pi:M^*\rightarrow M$ of $M$ with the center $Y=(x_i=y_\ell=0)$. Let $P^*$ be the center of $R$ in $M^*$ and put ${\mathcal O}_{{\mathcal A}^*}={\mathcal O}_{M^*,P^*}$. To obtain the regular system of parameters $\boldsymbol{x}^*,\boldsymbol{y}^*$ we have two cases:
    \begin{itemize}
    \item[a)] If $\nu(x_i)<\nu(y_\ell)$, we put $\boldsymbol{x}'=\boldsymbol{x}$, $y'_{\ell}=y_\ell/x_i$ and $y'_j=y_j$, for $j\ne \ell$.
    \item[b)] If $\nu(x_i)>\nu(y_\ell)$, we put $\boldsymbol{y}'=\boldsymbol{y}$, $x'_{i}=x_i/y_\ell$ and $x'_s=x_s$, for $s\ne i$.
    \end{itemize}
    The case (a), where $\nu(x_i)<\nu(y_\ell)$, is called {\em of transversal type}.
\item[2)]  An {\em $\ell$-blow-up ${\mathcal A}\rightarrow{\mathcal A}^{*}$ with translation} is given by the choice of an $x_i$ such that $\nu(y_\ell)=\nu(x_i)$.
     We take a projective model $M$ associated to $\mathcal A$ and the blow-up $\pi:M^*\rightarrow M$ of $M$ with the center $Y=(x_i=y_\ell=0)$. Let $P^*$ be the center of $R$ in $M^*$ and put ${\mathcal O}_{{\mathcal A}^*}={\mathcal O}_{M^*,P^*}$. We obtain the regular system of parameters $\boldsymbol{x}^*,\boldsymbol{y}^*$ as follows. Since $\nu(y_\ell/x_i)=0$ and $R$ defines a $k$-rational valuation, there is a unique $0\ne\lambda\in k$ such that $\nu(y_\ell/x_i-\lambda)>0$. We put $\boldsymbol{x}^*=\boldsymbol{x}$, $y^*_\ell=y_\ell/x_i-\lambda$ and $y^*_j=y_j$ for $j\ne \ell$.
\end{itemize}
\begin{remark} Let us note that an  $\ell$-combinatorial blow-up is not an allowed transformation itself. Nevertheless, an $\ell$-blow-up with translation defines an $\ell$-Puiseux's package.
\end{remark}

The existence of at least one $\ell$-Puiseux's package ${\mathcal A}\rightarrow{\mathcal A}'$ is proved in \cite{Can-R-S}. It is related with the so called {\em $\ell$-contact rational function $\Phi$ of $\mathcal A$}. Since the values of the independent parameters define a basis for $\Gamma\otimes_{\mathbb Z}{\mathbb Q}$, there are unique integers $d>0$ and $\boldsymbol{p}=(p_1,p_2,\ldots,p_r)\in {\mathbb Z}^r$ such that $d,p_1,p_2,\ldots,p_r$ are coprime and
\[
\nu(y_\ell^d/\boldsymbol{x^p})=0.
\]
We put $\Phi=y_\ell^d/\boldsymbol{x^p}$. The number $d\geq 1$ is called the {\em $\ell$-ramification index of ${\mathcal A}$}.
There is a unique $0\ne \lambda\in k$ such that $\nu(\Phi-\lambda)>0$. Note that, in the case of a blow-up with translation, the $\ell$-contact rational function is given by $y_\ell/x_i$.

\subsection{Equations for Puiseux's Packages}
\label{sec:EquationsforPuiseuxPackages} Let $\Phi=y_\ell^d/\boldsymbol{x^p}$ be the $\ell$-contact rational function of ${\mathcal A}$ and take $\lambda\in k$ with $\nu(\Phi-\lambda)>0$. Consider an $\ell$-Puiseux's package
\[
{\mathcal A}\rightarrow \widetilde{\mathcal A}\rightarrow {\mathcal A}',
\]
where
$
{\mathcal A}=({\mathcal O}_{\mathcal A},\boldsymbol{x},\boldsymbol{y})$,
$\widetilde{\mathcal A}=({\mathcal O}_{\widetilde{\mathcal A}},\widetilde{\boldsymbol{x}},\widetilde{\boldsymbol{y}})$,
$
{\mathcal A}'=({\mathcal O}_{{\mathcal A}'},\boldsymbol{x}',\boldsymbol{y}')
$. In this subsection we recall formulas in \cite{Can-R-S} relating the parameters $\boldsymbol{x},\boldsymbol{y}$, $\widetilde{\boldsymbol{x}},\widetilde{\boldsymbol{y}}$ and $\boldsymbol{x}',\boldsymbol{y}'$.

The relationship between $\widetilde{\boldsymbol{x}},\widetilde{\boldsymbol{y}}$ and $\boldsymbol{x}',\boldsymbol{y}'$ is given by $\boldsymbol{x}'=\widetilde{\boldsymbol{x}}$, $y'_\ell=\Phi-\lambda$, where $\Phi=\tilde y /\tilde x_{i_0}$ and
\begin{equation*}
y'_j=\tilde y_j \quad \text{for} \quad  j=1,2,\ldots,m-r,\; j\ne \ell.
\end{equation*}
The relationship between $\boldsymbol{x},\boldsymbol{y}$ and $\widetilde{\boldsymbol{x}},\widetilde{\boldsymbol{y}}$ is given by a $(r+1)\times(r+1)$ matrix $B=(b_{ij})$, with nonnegative integer coefficients and $\det B=1$, such that
\begin{equation*}
 \begin{array}{cclllll}
 x_1&=&\tilde x_1^{b_{11}}&\tilde x_2^{b_{12}}&\cdots &\tilde x_r^{b_{1r}}&\tilde y_\ell^{b_{1 r+1}}\\
 x_2&=&\tilde x_1^{b_{21}}&\tilde x_2^{b_{22}}&\cdots &\tilde x_r^{b_{2r}}&\tilde y_\ell^{b_{2,r+1}}\\
 &\cdots&\\
 x_r&=&\tilde x_1^{b_{r1}}&\tilde x_2^{b_{r2}}&\cdots &\tilde x_r^{b_{rr}}&\tilde y_\ell^{b_{r,r+1}}\\
 y_\ell&=&\tilde x_1^{b_{r+1,1}}&\tilde x_2^{b_{r+1,2}}&\cdots &\tilde x_r^{b_{r+1,r}}&\tilde y_\ell^{b_{r+1,r+1}} ,
 \end{array}
 \end{equation*}
and moreover $y_j=\tilde y_j$ for any $j\in \{1,2,\ldots,m-r\}\setminus \{\ell\}$. Recalling that $\Phi=\tilde y_\ell/\tilde x_{i_0}=y'_\ell+\lambda$, we have
\begin{equation}
\label{eq:puiseux2}
 \begin{array}{cclllll}
 x_1&=&{x'}_1^{c_{11}}&{x'}_2^{c_{12}}&\cdots &{x'}_r^{c_{1r}}&\Phi_\ell^{c_{1,r+1}}\\
 x_2&=&{x'}_1^{c_{21}}&{x'}_2^{c_{22}}&\cdots &{x'}_r^{c_{2r}}&\Phi_\ell^{c_{2,r+1}}\\
 &\cdots&\\
 x_r&=&{x'}_1^{c_{r1}}&{x'}_2^{c_{r2}}&\cdots &{x'}_r^{c_{rr}}&\Phi_\ell^{c_{r,r+1}}\\
 y_\ell&=&{x'}_1^{c_{r+1,1}}&{x'}_2^{c_{r+1,2}}&\cdots &{x'}_r^{c_{r+1,r}}&\Phi_\ell^{c_{r+1,r+1}}
 \end{array}
 \end{equation}
where $c_{si_0}=b_{si_0}+b_{s,r+1}$ for any $s=1,2,\ldots,r+1$ and $c_{ji}=b_{ji}$, if $i\ne i_0$. Note that we also have that $\det C=1$ and the coefficients $c_{ji}$ of $C$ are nonnegative integer numbers.
\begin{remark}
\label{rk:propiedadesecuacionespuiseux}
 We have the following properties:
\begin{itemize}
\item[1)] Denote
$
(-p_1,-p_2,\ldots,-p_r,d)=(0,0,\ldots,0,1)C^{-1}
$.

\item[2)] Consider the $r\times r$ sub-matrix $C_0$ of $C$ given by $C_0=(c_{ij})_{1\leq  i,j\leq r}$. Then $\det C_0 \ne 0$ since it gives a base change matrix for the bases
    \[
    \{ \nu(x_1),\nu,(x_2),\ldots,\nu(x_r)\},\text{ and } \{ \nu(x'_1),\nu,(x'_2),\ldots,\nu(x'_r)\}
    \]
    of $\Gamma\otimes_{\mathbb Z}{\mathbb Q}$ in view of Equation \eqref{eq:puiseux2}.
\item[3)] We have $\nu(y_\ell)>0$ and hence $y_\ell$ is not a unit in ${\mathcal O}_{{\mathcal A}'}$. More precisely, we have that $\nu_{{\mathcal A}'}(y_\ell^d)=\nu(\boldsymbol{x}^{\boldsymbol{p}})>0$ and thus $\nu_{{\mathcal A}'}(y_\ell)>0$. As a consequence of Proposition \ref{prop:explicitvalueofthedifferential}, we also have that $\nu_{{\mathcal A}'}(dy_\ell)=\nu_{{\mathcal A}'}(y_\ell)>0$.
\item[4)] The case $d=1$ is relevant for our computations. The following properties are equivalent
    \begin{itemize}
    \item[i)] $d=1$.
    \item[ii)] The $\ell$-combinatorial blow-ups in ${\mathcal A}\rightarrow\widetilde{\mathcal A}$ are all of transversal type.
    \item[iii)] The matrix $C$ has the form
\begin{equation}
\label{eq:ecuaciondigulauno}
C=
\left(
\begin{array}{ccc}
{C_0}&\vline&{\boldsymbol{0}}\\
\hline
\boldsymbol{\tilde p}&\vline& 1\\
\end{array}
\right),\quad \boldsymbol{\tilde p}=\boldsymbol{p}C_0.
\end{equation}
 \end{itemize}
\item[5)]
 The relationship between the differential forms are given by:
\begin{equation*}
\left(\frac{dx_1}{x_1}, \frac{dx_2}{x_2},\cdots,\frac{dx_r}{x_r}, \frac{dy_\ell}{y_\ell}\right)=
\left(\frac{d\tilde x_1}{\tilde x_1}, \frac{d\tilde x_2}{\tilde x_2},\cdots,\frac{d\tilde x_r}{\tilde x_r}, \frac{d\tilde y_\ell}{\tilde y_\ell}\right)B^t
\end{equation*}
 and
 \begin{equation}
 \label{eq:puiseux4}
\left(\frac{dx_1}{x_1}, \frac{dx_2}{x_2},\cdots,\frac{dx_r}{x_r}, \frac{dy_\ell}{y_\ell}\right)=
\left(\frac{dx'_1}{ x'_1}, \frac{dx'_2}{x'_2},\cdots,\frac{dx'_r}{ x'_r}, \frac{d \Phi}{\Phi}\right)C^t,
\end{equation}
where we remark that $d\Phi/\Phi=(1/(y'_\ell+\lambda))dy'_\ell$.
\end{itemize}
\end{remark}

\subsection{Stability Results}
The use of allowed transformations in the Truncated Local Uniformization is justified by the results in this subsection. We show that the critical values are not decreasing under allowed transformations and that the stabilization of the critical value characterize final situations.
\begin{remark}
\label{rk:transforms}
Any allowed transformation ${\mathcal A}\rightarrow {\mathcal A}'$ gives and injective morphism of local rings ${\mathcal O}_{\mathcal A}\rightarrow {\mathcal O}_{{\mathcal A}'}$ and also an injective morphism of complete local rings $\Omega^0_{\mathcal A}\rightarrow\Omega^0_{{\mathcal A}'}$.  We also have induced inclusions
\[
\Omega^p_{\mathcal A}\subset\Omega^p_{{\mathcal A}'}
\]
that are compatible with the sum, the exterior product and the exterior differentiation.
In order to simplify notations along the paper, when we have $\omega\in \Omega^p_{{\mathcal A}}$, we write $\omega\in \Omega^p_{{\mathcal A}'}$ to denote the transform of $\omega$ by the considered allowed transformation.
 \end{remark}

\begin{lemma}
\label{lema:transformationofamonomial}
Let
$
{\mathcal A}
\rightarrow{\mathcal A}'
$ be an elementary allowed transformation. Consider a monomial $\boldsymbol{x}^I$. We have $ \boldsymbol{x}^I=U'{\boldsymbol{x}'}^{I'} $ where $U'\in {\mathcal O}_{{\mathcal A}'}$ is a unit. Moreover, if ${\mathcal A}\rightarrow {\mathcal A}'$ is an $\ell$-Puiseux's package, then $U'$ has the form
\[
U'=(y'_\ell+\lambda)^b\in k[[y'_\ell]] ,
\]
where $b\in {\mathbb Z}$ and $0\ne \lambda\in k$. Otherwise, we have that $U'=1$.
\end{lemma}
\begin{proof} This result is a direct consequence of the description, given in Subsections \ref{Elementary allowed transformations} and \ref{sec:EquationsforPuiseuxPackages},  of the elementary allowed transformations.
\end{proof}
\begin{proposition}
\label{prop:stabilityofcriticalvalue}
Let
$
{\mathcal A}
\rightarrow{\mathcal A}'
$ be an allowed transformation  and consider
$\alpha\in \Omega^p_{\mathcal A}$. We have
$
\nu_{{\mathcal A}'}(\alpha)\geq\nu_{\mathcal A}(\alpha).
$
Moreover, if $p\in \{0,1\}$ and ${\mathfrak c}_{\mathcal A}^{p}(\alpha)\ne 0$, then we have that  $\nu_{{\mathcal A}'}(\alpha)=\nu_{\mathcal A}(\alpha)$ and ${\mathfrak c}_{{\mathcal A}'}^{p}(\alpha)\ne 0$.
\end{proposition}
\begin{proof} If ${\mathcal A}\rightarrow {\mathcal A}'$ is an $\ell$-coordinate change, we are done. Thus, we have only to solve the cases of an independent blow-up or an $\ell$-Puiseux's package.

Put $\delta=\nu_{\mathcal A}(\alpha)$. Let us prove that $\nu_{{\mathcal A}'}(\alpha)\geq \delta$.
By the properties of the explicit order, we have only to consider the case
$\alpha=f\in \Omega^0_{\mathcal A}$. Write $f$ as a finite sum
\[
f=\sum_{\nu(\boldsymbol{x}^I)\geq \delta}\boldsymbol{x}^Ih_I,\quad   h_I\in k[[\boldsymbol{x},\boldsymbol{y}]].
\]
Now, it is enough to show that
$\nu_{{\mathcal A}'}(\boldsymbol{x}^I)=\nu(\boldsymbol{x}^I)$. By Lemma \ref{lema:transformationofamonomial}, there is a unit $U'$ in ${\mathcal O}_{{\mathcal A}'}$
such that
 $
 \boldsymbol{x}^I= U'{\boldsymbol{x}'}^{I'}
 $.
This implies that $\nu_{{\mathcal A}'}(\boldsymbol{x}^I)= \nu({\boldsymbol{x}'}^{I'})= \nu(\boldsymbol{x}^I)$.

Let us prove the second statement.
Assume that $\alpha=f\in \widehat{\mathcal O}_{M,P}$. We know that $\nu_{\mathcal A}(f)=\delta$ and
${\mathfrak c}^0_{\mathcal A}(f)\ne 0$ if and only if we can write
\[
f=\boldsymbol{x}^If_I+\tilde f,
\]
where $f_I\in k[[\boldsymbol{x},\boldsymbol{y}]]$ is a unit,
$\nu(\boldsymbol{x}^I)=\delta$  and $\nu_{\mathcal A}(\tilde f)>\delta$. By Lemma \ref{lema:transformationofamonomial} we can write
\[
f={\boldsymbol{x}'}^{I'}(f_IU')+\tilde f
\]
and we are done since $f_IU'$ is a unit in $k[[\boldsymbol{x}',\boldsymbol{y}']]$ and $\nu_{{\mathcal A}'}(\tilde f)>\delta$.

Assume now that $\alpha\in \Omega^1_{\mathcal A}$. Saying that $\nu_{\mathcal A}(\alpha)=\delta$ and
${\mathfrak c}_{\mathcal A}^1(\alpha)\ne 0$ is equivalent to saying that if $\alpha=\boldsymbol{x}^I\beta+\tilde\alpha$,  $\nu_{\mathcal A}(\tilde \alpha)>\delta$ and $\nu(\boldsymbol{x}^I)=\delta$, then  $\nu_{\mathcal A}(\beta)=0$ and ${\mathfrak c}_{\mathcal A}^1(\beta)\ne 0$. Write $\alpha=\boldsymbol{x}^I\beta+\tilde\alpha$ as before. By Lemma \ref{lema:transformationofamonomial} and the previous result, the only thing we have to prove is that $\nu_{{\mathcal A}'}\beta=0$ and ${\mathfrak c}_{{\mathcal A}'}^1(\beta)\ne 0$.
Write
\[
\beta=\sum_{s=1}^rf_s\frac{dx_s}{x_s}+\sum_{j=1}^{m-r}g_jdy_j.
\]
We know that there is a unit among the coefficients $f_s$. In the case of an independent blow-up, we have
\[
\beta=\sum_{s=1}^rf'_s\frac{dx'_s}{x'_s}+\sum_{j=1}^{m-r}g_jdy'_j,
\]
where $(f'_1,f'_2,\ldots,f'_r)=(f_1,f_2,\ldots,f_r)A$ and $A$ is a matrix with $\det A=1$ and nonnegative integer coefficients. Then, there is a unit between the coefficients $f'_s$ and hence $\nu_{E'}(\beta)=\delta$ and ${\mathfrak c}^1_{{\mathcal A}'}(\beta)\ne 0$. In the case of
an $\ell$-Puiseux's package, let us write
\[
\beta=\sum_{s=1}^rf_s\frac{dx_s}{x_s}+g_\ell y_\ell\frac{ dy_\ell}{y_\ell}+\sum_{j\ne \ell }g_jdy_j= \sum_{s=1}^rf'_s\frac{dx'_s}{x'_s}+g'_\ell\frac{ dy'_\ell}{y'_\ell+\lambda}+\sum_{j\ne \ell }g'_jdy'_j.
\]
By Equation \eqref{eq:puiseux4}, we have that
$
(f'_1,f'_2,\ldots,f'_r,g'_\ell)= (f_1,f_2,\ldots,f_r,y_\ell g_\ell)C
$.
Let us denote $\bar f=f+{{\mathfrak m}'}\Omega^0_{{\mathcal A}'}\in k$. We know that $\bar{y}_\ell=0$. Thus
\[
(\bar f'_1,\bar f'_2,\ldots,\bar f'_r,\bar g'_\ell)= (\bar f_1,\bar f_2,\ldots,\bar f_r,0)C,
\]
 and hence $(\bar f'_1,\bar f'_2,\ldots,\bar f'_r)= (\bar f_1,\bar f_2,\ldots,\bar f_r)C_0$. Since $\det C_0\ne 0$, we conclude that there is a unit among the $f'_s$ and thus $\nu_{{\mathcal A}'}(\beta)=0$ and ${\mathfrak c}^1_{{\mathcal A}'}(\beta)\ne 0$.
\end{proof}

\begin{corollary}[Stability]
\label{cor:stabilityofcriticalvalue}
Consider
$\alpha\in \Omega^p_{\mathcal A}$ with $p\in \{0,1\}$ and a real number $\gamma\in {\mathbb R}$. Let
$
{\mathcal A}
\rightarrow{\mathcal A}'
$ be an allowed transformation. Then
\begin{itemize}
\item If $({\mathcal A},\alpha)$ is $\gamma$-final recessive, then $({\mathcal A}',\alpha)$ is also $\gamma$-final recessive.
\item If $({\mathcal A},\alpha)$ is $\gamma$-final dominant with $\nu_{\mathcal A}(\alpha)=\delta$, then $({\mathcal A}',\alpha)$ is also $\gamma$-final dominant with $\nu_{{\mathcal A}'}(\alpha)=\delta$.
\end{itemize}
\end{corollary}

\begin{proposition}
\label{prop:notfinalformsandcriticalvalue}
Let $0\ne \alpha\in \Omega^p_{\mathcal A}$ with $p\in \{0,1\}$ be such that ${\mathfrak c}^p_{\mathcal A}(\alpha)=0$.
Consider a sequence of allowed transformations
\[{\mathcal A}={\mathcal A}_0\rightarrow {\mathcal A}_1\rightarrow \cdots \rightarrow {\mathcal A}_{m-r}={\mathcal A}'\]
such that each ${\mathcal A}_{\ell-1}\rightarrow{\mathcal A}_{\ell}$ is an $\ell$-Puiseux's package. We have $\nu_{{\mathcal A}'}(\alpha)>\nu_{\mathcal A}(\alpha)$.
\end{proposition}
\begin{proof} Put $\delta=\nu_{\mathcal A}(\alpha)$, note that $\delta<\infty$, since $\alpha\ne 0$. Let us write $\alpha$ as
\[
\alpha=\boldsymbol{x}^I\alpha_I+\tilde \alpha, \quad \nu_{\mathcal A}(\tilde \alpha)>\delta,\quad \nu(\boldsymbol{x}^I)=\delta,
\]
where, in addition, we write $\alpha_I$ as follows:
\[
\alpha_I=\sum_{s=1}^{m-r}y_s\beta_s,\;\text{ if } p=0;\quad
\alpha_I=\sum_{s=1}^{m-r}(y_s\beta_s+h_sdy_s),\; \text{ if } p=1.
\]
This is possible because ${\mathfrak c}_{\mathcal A}^p(\alpha)=0$. The transforms in ${\mathcal A}_{s-1}$ of $y_s\beta_s$ and $h_sy_s$ have the form
\[
y^{(s-1)}_s\beta^{(s-1)}_s \text{ and } h^{(s-1)}_sdy^{(s-1)}_s.
\]
By Remark \ref{rk:propiedadesecuacionespuiseux}, after performing the $s$-Puiseux's package ${\mathcal A}_{s-1}\rightarrow {\mathcal A}_s$ we have that
$
\nu_{{\mathcal A}_s}(y^{(s-1)}_s)>0$ and $\nu_{{\mathcal A}_s}(dy^{(s-1)}_s)>0
$.
In this way, we obtain that $\nu_{{\mathcal A}'}(\alpha_I)>0$. Hence $\nu_{{\mathcal A}'}(\alpha)>\delta$.
\end{proof}
\begin{remark}
Let us fix a real number $\gamma\geq 0$, a locally parameterized model ${\mathcal A}$ and $0\ne\alpha\in \Omega^p_{{\mathcal A}}$, $p\in \{0,1\}$. In order to obtain a $\gamma$-truncated local uniformization of $\alpha$, a rough idea is to increase the explicit value $\nu_{\mathcal A}(\alpha)$ while $ \alpha $ is not final dominant. By Proposition \ref{prop:notfinalformsandcriticalvalue}, we can increase the value, but we have to deal with the possibility of an accumulation of the explicit values before arriving to the truncation limit given by $\gamma$. This is one of the classical difficulties in Local Uniformization problems.
\end{remark}
\subsection{Final forms and Independent Blow-ups}
It is a classical result \cite{Spi1, Hir1, Zar1, Cut} that any monomial ideal in the independent variables becomes a principal ideal under a suitable sequence of independent blow-ups. Let us state here these results in a useful way for the proof of the truncated local uniformization of $1$-forms.

\begin{proposition}[Monomial Local Principalization]
\label{prop:monlocprincipalization}
Let
$
{\mathcal A}
$ be a  locally parameterized model. Consider a nonempty family
\[
{\mathcal L}=\{m_I=U_I\boldsymbol{x}^I\}_{I\in A},\quad A\subset{\mathbb Z}_{\geq 0}^r,
\]
where $U_I$ is a unit in $\Omega^0_{\mathcal A}$ for any $I\in A$. Let $I_0$ be the index in $A$ such that $\nu(\boldsymbol{x}^{I_0})$ is minimum. There is a transformation ${\mathcal A}\rightarrow {\mathcal A'}$ obtained as a composition of finitely many independent blow-ups such that the transformed list
\[{\mathcal L}'=\{m'_I=U'_I{\boldsymbol{x}'}^{I'}\}_{I\in A}\]
has the property that ${\boldsymbol{x}'}^{I'_0}$ divides any other $m'_I$.
\end{proposition}
\begin{proof}(See the above references). The proof may be done by succesive elimination of vertices of the Newton Polyhedron of $\mathcal L$. We leave the details to the reader.
\end{proof}
\begin{corollary}
\label{cor:principalizacion}
 Consider a $\rho$-final dominant $\omega\in \Omega^1_{\mathcal A}$. There is an allowed transformation ${\mathcal A}\rightarrow {\mathcal A'}$, obtained as a composition of finitely many independent blow-ups, such that $\omega$ has the form
\[
\omega={{\boldsymbol x}'}^{I'_0} \omega',\quad \omega'\in \Omega^1_{{\mathcal A}'}
\]
where $({\mathcal A}',\omega')$ is $0$-final dominant  and $\nu({{\boldsymbol x}'}^{I'_0})=\rho$.
\end{corollary}
\begin{proof} Write
$
\omega= {\boldsymbol x}^{I_0}\omega_{I_0}+\sum_{\nu({\boldsymbol x}^{I})>\rho} {\boldsymbol x}^{I}\omega_I,
$
where $\nu({\boldsymbol x}^{I_0})=\rho$ and $\omega_{I_0}$ is $0$-final dominant and apply Proposition \ref{prop:monlocprincipalization} to the list ${\mathcal L}=\{\boldsymbol{x}^I; \nu({\boldsymbol x}^{I})\geq \rho\}$.
\end{proof}

\section{Truncated Local Uniformization}
\label{Truncated Local Uniformization}
In Theorem  \ref{teo:formalforms} we state the Truncated Local Uniformization for $1$-Forms:
\begin{theorem}
 \label{teo:formalforms}
 Consider a $\gamma$-truncated formal foliated space $({\mathcal A},\omega)$. There is an allowed transformation
$
{\mathcal A}\rightarrow {\mathcal A}'
$
such that  $({\mathcal A}',\omega)$
is $\gamma$-final.
\end{theorem}
We get also a Truncated Local Uniformization for Formal Functions as follows:
\begin{theorem}
 \label{teo:functions}
 Consider a formal function $f\in \Omega^0_{\mathcal A}$ and a real number $\gamma\in {\mathbb R}$. There is an allowed transformation
$
{\mathcal A}\rightarrow {\mathcal A}'
$
such that $({\mathcal A}',f)$ is $\gamma$-final.
\end{theorem}
Theorem \ref{teo:functions} may be considered as an avatar of the classical Zariski's Local Uniformization in \cite{Zar1}. Note that Theorem \ref{teo:functions} is a consequence of our main result Theorem \ref{teo:formalforms}, when we consider the $1$-form $\omega=df$.

\subsection{Induction Structure}
\label{Induction Structure}
The proof of Theorem \ref{teo:formalforms}
 goes by induction on the
{\em  number $\operatorname{I}_{\mathcal A}(\omega)$ of dependent variables involved in the formal $1$-form $\omega$}. To be precise,  given  $s\in {\mathbb Z}_{\geq 0}$ we have
$\operatorname{I}_{\mathcal A}(\omega)\leq s$ if and only
\[
\omega=\sum_{i=1}^sf_i\frac{dx_i}{x_i}+\sum_{j=1}^{m-r}g_jdy_j
\]
satisfies that
\begin{itemize}
\item $f_i,g_j\in k[[\boldsymbol{x},y_1,y_2,\ldots,y_s]]$, for $i=1,2,\ldots,r$, $j=1,2,\ldots, m-r$.
\item $g_j=0$, for any $j=s+1,s+2,\ldots,m-r$.
\end{itemize}
\begin{remark}
\label{rk:startinginduction}
If $\operatorname{I}_{\mathcal A}(\omega)=0$, then $\omega$ is $\gamma$-final for any $ \gamma \in \mathbb{R} $. Namely, write $\omega$ as
\[
\omega= \boldsymbol{x}^{I_0} \frac{d \boldsymbol{x}^{I_0}}{\boldsymbol{x}^{I_0}}+
\sum_{\nu( \boldsymbol{x}^I)>\delta } \boldsymbol{x}^I \frac{d \boldsymbol{x}^{I_0}}{\boldsymbol{x}^{I_0}} , \quad \nu_{\mathcal A}(\omega)=\delta,
\]
where $\nu(\boldsymbol{x}^{I_0})=\delta$, $\lambda_I\in k^r$ and $\lambda_{I_0}\neq 0$. Then $\omega$ is $\gamma$-final dominant when $\gamma\leq \delta$ and it is $\gamma$-final recessive when $\gamma>\delta$.
 \end{remark}
 \begin{definition}
 \label{def:nestedtransformation} An allowed transformation ${\mathcal A}\rightarrow{\mathcal A}'$ is called a {\em $\ell$-nested transformation} when it is a finite composition of independent blow-ups, $\ell'$-coordinate changes and $\ell'$-Puiseux's packages,  with $\ell'\leq\ell$.
\end{definition}
\begin{remark} If ${\mathcal A}\rightarrow {\mathcal A}'$ is a $I_{\mathcal A}(\omega)$-nested transformation, then
$I_{{\mathcal A}'}(\omega)\leq I_{\mathcal A}(\omega)$.
\end{remark}
The inductive version of Theorem \ref{teo:formalforms} that we are going to prove is the following one:
\begin{theorem}
\label{teo:inductivestatement}
Consider a $\gamma$-truncated parameterized formal foliated space $({\mathcal A},\omega)$.
There is an $\operatorname{I}_{\mathcal A}(\omega)$-nested transformation ${\mathcal A}\rightarrow{\mathcal A}'$ such that $({\mathcal A}',\omega)$ is $\gamma$-final.
\end{theorem}
\subsection{Starting Situation for the Inductive Step}
 \label{subsection:initialsituation}
 Next sections are devoted to the proof of Theorem \ref{teo:inductivestatement} by induction on $\operatorname{I}_{\mathcal A}(\omega)$. This  is the main technical part in this paper. The proof is divided in two parts. The first one is the {\em preparation procedure}. In the second part, we provide a control of the critical height, in a prepared situation, under ``normalized'' Puiseux's packages and coordinate changes.

 The rest of this paper, except for the last section, is devoted to the inductive proof of Theorem \ref{teo:inductivestatement}. By Remark \ref{rk:startinginduction}, we know that Theorem \ref{teo:inductivestatement} is true when we have $I_{\mathcal A}(\omega)=0$. Thus, we supose that $I_{\mathcal A}(\omega)=\ell+1\geq 1$ and we take the {\em induction hypothesis} that says that Theorem \ref{teo:inductivestatement} is true when $I_{\mathcal A}(\omega)\leq \ell$.

 In order to avoid heavy notations, let us denote
\begin{equation*}
y_{\ell+1}=z,\quad \boldsymbol{y}_{\leq \ell}=(y_1,y_2,\ldots,y_{\ell}),\quad \boldsymbol{y}_{\geq\ell+2}=(y_{\ell+2},y_{\ell+3},\ldots,y_{n-r}),
\end{equation*}
in such a way that
$
\boldsymbol{y}=(\boldsymbol{y}_{\leq \ell},z,\boldsymbol{y}_{\geq \ell+2})
$. We decompose $\omega$ into levels as:
\begin{equation}
\label{eq:omegalevels}
\omega=\sum_{s\geq 0}z^s\omega_s,\quad
\omega_s=\eta_s+h_s\frac{dz}{z},\quad \eta_s=\sum_{i=1}^rf_{si}\frac{dx_i}{x_i}+\sum_{j=1}^{\ell}g_{sj}dy_j,
\end{equation}
$h_0=0$. Moreover, we have that $\eta_s\in \Omega^1_{\mathcal A}$ with $I_{\mathcal A}(\eta_s)\leq \ell$, so we can apply the induction hypothesis to $\eta_s$.
\begin{notation}
The induction process assigns a special consideration to the $(\ell+1)$-th dependent variable in $\mathcal A$. For the sake of simplicity, and when no confusion arises, we frequently say {\em rational contact function } or {\em ramification index} in reference to the {\em $(\ell+1)$-rational contact function } or to the {\em $(\ell+1)$-ramification index}.
\end{notation}

\subsection{Newton Polygons}
\label{sec:newtonpolygon}
The $1$-form $\omega_s$ in Equation \eqref{eq:omegalevels} is called {\em the $s$-level of $\omega$ in $\mathcal A$}. Let us note that $z\omega_s\in \Omega^1_{\mathcal A}$ and
\[
\nu_{\mathcal A}(z\omega_s)=\min\{\nu_{\mathcal A}(\eta_s),\nu_{\mathcal A}(h_s)\}.
\]
The {\em cloud of points $\operatorname{Cl}_{\mathcal A}({\omega})$} is defined by
\[
\operatorname{Cl}_{\mathcal A}({\omega})=\{(\nu_{\mathcal A}(z\omega_s),s); s\in {\mathbb Z}_{\geq 0}\}\subset {\mathbb R}_{\geq 0}\times {\mathbb Z}_{\geq 0}\subset {\mathbb R}_{\geq 0}^2.
\]
The {\em Newton Polygon ${\mathcal N}_{\mathcal A}(\omega)$} is the positively convex hull of $\operatorname{Cl}_{\mathcal A}({\omega})$ in ${\mathbb R}_{\geq 0}^2$. That is ${\mathcal N}_{\mathcal A}(\omega)$ is the convex hull of
$
\operatorname{Cl}_{\mathcal A}({\omega})+{\mathbb R}_{\geq 0}^2
$
in ${\mathbb R}_{\geq 0}^2$. It has finitely many vertices, each one belonging to the cloud of points.
The {\em main vertex } is the vertex with highest ordinate, it is also the vertex with minimal abscissa. We call {\em main height} to the ordinate of the main vertex. Let us note that the abscissa of the main vertex is equal to
$
\nu_{\mathcal A}(\omega)
$.

Proposition \ref{pro:onevertex} below highlights the inductive role of the Newton Polygon:
\begin{proposition}
\label{pro:onevertex}
Assume that ${\mathcal N}_{\mathcal A}(\omega)$ has the only vertex $(\rho,0)$. Then $\omega$ is $\gamma$-final dominant, respectively recessive, if and only if $\eta_0$ is $\gamma$-final dominant, respectively recessive. More precisely, we have $\nu_{\mathcal A}(\omega)=\nu_{\mathcal A}(\eta_0)=\rho$
and ${\mathfrak c}^1_{\mathcal A}(\omega)={\mathfrak c}^1_{\mathcal A}(\eta_0)$.
\end{proposition}
\begin{proof} We already know that $\rho=\nu_{\mathcal A}(\omega)$. Since $h_0=0$, we  have that $\rho=\nu_{\mathcal A}(\eta_0)$. Write $\omega=\eta_0+\tilde\omega$. We have that $\nu_{\mathcal A}(\tilde\omega)\geq \lambda$ and since
\[
\tilde\omega=z\sum_{s\geq 1}z^{s-1}\eta_s+\sum_{s\geq 0}z^sh_{s+1}dz
\]
we conclude that ${\mathfrak c}^1_{\mathcal A}(\eta_0+\tilde\omega)={\mathfrak c}^1_{\mathcal A}(\eta_0)$.
\end{proof}

The {\em critical value $\varsigma_{\mathcal A}(\omega)$} is defined by
\[
\varsigma_{\mathcal A}(\omega)=\min\{\alpha+s\nu(z);\; (\alpha,s)\in \operatorname{Cl}_{\mathcal A}(\omega)\}.
\]
The {\em critical line $\operatorname{L}_{\mathcal A}(\omega)$} is
$
\operatorname{L}_{\mathcal A}(\omega)=\{(a,b)\in{\mathbb R}^2;a+b\nu(z)=\varsigma_{\mathcal A}(\omega)\}
$,
the {\em critical segment $\operatorname{C}_{\mathcal A}(\omega)$} is the intersection
 $\operatorname{C}_{\mathcal A}(\omega)=\operatorname{L}_{\mathcal A}(\omega)\cap {\mathcal N}_{\mathcal A}(\omega)$ and finally, the {\em critical height $\chi_{\mathcal A}(\omega)$} is the ordinate of the highest vertex in the critical segment. This vertex has the form
 $
 (\varsigma_{\mathcal A}(\omega)-\nu(z)\chi_{\mathcal A}(\omega), \chi_{\mathcal A}(\omega))
 $
and is called {\em the critical vertex}.
\section{Statements of Strict Preparation}
\label{Statements of Strict Preparation}
We introduce here the definitions and first results concerning the preparation process of  the $\gamma$-truncated formal foliated space $({\mathcal A},\omega)$.

\begin{definition}
\label{def:finallevels}
Given a real number $\delta\in {\mathbb R}$ we say that the $s$-level $\omega_s$ of $({\mathcal A},\omega)$ is
 \begin{itemize}
 \item {\em $\delta$-dominant}, if $\nu_{\mathcal A}(z\omega_s)\leq \delta$ and both $({\mathcal A},\eta_s)$ and $(\mathcal A,h_s)$ are
 $\nu_{\mathcal A}(z\omega)$-final.
 \item {\em $\delta$-recessive}, if $\nu_{\mathcal A}(z\omega_s)>\delta$.
 \end{itemize}
It is called {\em $\delta$-final} if it is $\delta$-dominant or $\delta$-recessive.
\end{definition}

\begin{remark}
  \label{rk:stabilityoflevels}
  Let ${\mathcal A}\rightarrow {\mathcal A}'$ be an $\ell$-nested transformation. Then the decomposition into levels given in Equation \eqref{eq:omegalevels} is valid both for
  $({\mathcal A},\omega)$  and $({\mathcal A}',\omega)$. Applying Proposition \ref{prop:stabilityofcriticalvalue} to each level, we have that
 \[
 {\mathcal N}_{{\mathcal A}'}(\omega)\subset{\mathcal N}_{\mathcal A}(\omega).
 \]
If
 $\omega_s$ is $\lambda_s$-dominant
 for each vertex $(\lambda_s,s)$, we have the equality $
 {\mathcal N}_{{\mathcal A}'}(\omega)={\mathcal N}_{\mathcal A}(\omega)
 $.
 \end{remark}
\begin{definition}
\label{def:preparado}
We say that
  $({\mathcal A},\omega)$ is {\em $\gamma$-prepared} if and only if each level $\omega_s$ is $(\tilde\gamma-s\nu(z))$-final, where $\tilde\gamma=\min\{\gamma, \varsigma_{\mathcal A}(\omega)\}$.
\end{definition}
\begin{remark}
  \label{rk:stability of preparation} If $({\mathcal A},\omega)$ is $\gamma$-prepared and ${\mathcal A}\rightarrow{\mathcal A}'$ is an $\ell$-nested transformation, then $({\mathcal A}',\omega)$ is also $\gamma$-prepared. Moreover, we have same critical value and same critical segment, that is
  $
  \varsigma_{{\mathcal A}'}(\omega)=\varsigma_{\mathcal A}(\omega)
  $
  and $C_{{\mathcal A}'}(\omega)= C_{{\mathcal A}}(\omega)$. This is a consequence of Proposition \ref{prop:stabilityofcriticalvalue}.
\end{remark}

 \subsection{Maximally Dominant Truncated Foliated Spaces}
  \label{Dominant Preparation}
  The first step of the preparation process is to obtain a ``maximally dominant polygon''. To be precise, we say that $({\mathcal A},\omega)$ is {\em $\rho$-maximally dominant} if and only if, for any $s\geq 0$,  one of the following properties holds:
 \begin{itemize}
 \item The $s$-level $\omega_s$ is $(\rho-s\nu(z))$-dominant.
 \item There is no $\ell$-nested transformation ${\mathcal A}\rightarrow {\mathcal A}'$ such that the $s$-level becomes $(\rho-s\nu(z))$-dominant.
 \end{itemize}
 \begin{remark}
 Let us note that if $\rho'\leq\rho$ and $({\mathcal A},\omega)$
is $\rho$-maximally dominant, then it is also $\rho'$-maximally dominant.
\end{remark}

 In view of Propositions \ref{prop:stabilityofcriticalvalue} and  \ref{prop:notfinalformsandcriticalvalue} and Remark \ref{rk:stabilityoflevels} the property for an $s$-level of being $(\rho-s\nu(z))$-dominant is stable under any $\ell$-nested transformation ${\mathcal A}\rightarrow {\mathcal A}'$. In particular, the property of being $\rho$-maximally dominant is stable under further $\ell$-nested transformations.
 \begin{remark} Any $\gamma$-prepared $({\mathcal A},\omega)$ is $\tilde\gamma$-maximally dominant, where we denote  $\tilde\gamma=\min\{\gamma,\varsigma_{\mathcal A}(\omega)\}$.
 \end{remark}

 The ``dominant preparation'' is given by Proposition \ref{prop:dominantpreparation} below:

 \begin{proposition}
 \label{prop:dominantpreparation}
 There is an $\ell$-nested transformation ${\mathcal A}\rightarrow{\mathcal A}'$ such that $({\mathcal A}',\omega)$ is $\gamma$-maximally dominant.
 \end{proposition}
 \begin{proof} We have only to consider $s$-levels with $0\leq s\leq \gamma/\nu(z)$. This is a finite set. Then, it is enough to consider an $\ell$-nested transformation that produces the maximum number of $(\gamma-s\nu(z))$-dominant levels.
  \end{proof}

 \subsection{Pseudo-Prepared and Strictly Prepared Situations} We say that $({\mathcal A},\omega)$ is {\em $\gamma$-pseudo prepared} if and only if the following conditions hold:
 \begin{itemize}
 \item $({\mathcal A},\omega)$ is $\gamma$-maximally dominant.
 \item For any $s\geq 0$ we have that $h_s$ is $(\gamma-s\nu(z))$-final.
 \item The $0$-level $\eta_0$ is $\gamma$-final.
 \end{itemize}
  In view of Proposition \ref{prop:dominantpreparation} and using the induction hypothesis applied to formal functions and to $\eta_0$, there is an $\ell$-nested transformation ${\mathcal A}\rightarrow {\mathcal A}'$ such that $({\mathcal A}',\omega)$ is $\gamma$-pseudo prepared. Moreover, if $({\mathcal A},\omega)$ is $\gamma$-pseudo prepared, then $({\mathcal A}',\omega)$ is also $\gamma$-pseudo prepared for any $\ell$-nested transformation ${\mathcal A}\rightarrow {\mathcal A}'$.

  \begin{definition} We say that $({\mathcal A},\omega)$ is {\em strictly $\gamma$-prepared} if and only if it is both $\gamma$-prepared and $\gamma$-pseudo prepared.
  \end{definition}

  \begin{remark} Although we only need a result of preparation for our purposes, we shall give a proof of the existence of strict $\gamma$-preparation.
  \end{remark}

  \subsection{Preparation Theorem}
  Next Preparation Theorem \ref{teo:preparation} is the first step in inductive proof of Theorem \ref{teo:inductivestatement}.

 \begin{theorem}[Strict Preparation]
 \label{teo:preparation} Let $({\mathcal A},\omega)$ be a $\gamma$-truncated formal foliated space $({\mathcal A},\omega)$ with $I_{\mathcal A}(\omega)=\ell+1$. There is an $\ell$-nested transformation ${\mathcal A}\rightarrow {\mathcal A}'$ such that $({\mathcal A}',\omega)$ is strictly $\gamma$-prepared.
\end{theorem}
The if ${\mathcal A}\rightarrow {\mathcal A}'$ is an $\ell$-nested transformation as in Theorem \ref{teo:preparation}, we say that
  $
  ({\mathcal A},\omega)\rightarrow ({\mathcal A}',\omega)
  $
  is a {\em strict $\gamma$-preparation of $({\mathcal A},\omega)$.}

Let us note that, in order to prove Theorem \ref{teo:preparation}, we can assume that $({\mathcal A},\omega)$ is $\gamma$-pseudo prepared.
The next section is devoted to the proof of  Theorem  \ref{teo:preparation}.

\section{Preparation Process}
\label{Preparation Process}

In this section we give a proof of Theorem  \ref{teo:preparation}. In other words, we show the existence of a {\em strict $\gamma$-preparation} for a given $\gamma$-truncated formal foliated space $({\mathcal A},\omega)$, that we suppose to be $\gamma$-pseudo prepared without loss of generality.

The $\gamma$-preparation will be done in two steps. First, we show that we can approximate the Newton Polygon to the dominant Newton Polygon. Second, we use this approximation to obtain the preparation.

\begin{remark}
The preparation process is done under the induction hypothesis. More precisely, we are going to apply induction hypothesis to the levels
\[
\omega_s=\eta_s+h_sdz/z,
\]
in particular, to the $1$-forms $\eta_s$ and to the ``horizontal coefficients'' $h_s$. The induction hypothesis may be directly applied to the formal functions $h_s$, since we do not need to assure any additional property of truncated integrability (recall that $dh_s$ is Frobenius integrable, since $d(dh_s)=0$). Nevertheless, when we consider the $1$-forms $\eta_s$, we can only apply the induction hypothesis with respect to truncation values $\rho$ for which $\eta_s$ satisfy the $\rho$-truncated integrability condition. The ``previous'' approximated preparation is necessary due to this observation.
\end{remark}

 \subsection{Dominant Newton Polygon} Assume that $({\mathcal A},\omega)$ is $\gamma$-maximally dominant. The {\em $\gamma$-dominant cloud of points
 $\operatorname{Cl}^\gamma_{\mathcal A}(\omega)$
 } is defined by
 \[
 \operatorname{Cl}^\gamma_{\mathcal A}(\omega)=\{(\nu_{\mathcal A}(z\omega_s),s);\quad
 \mbox{\rm the $s$-level $\omega_s$ is $(\gamma-s\nu(z))$-dominant }\}.
 \]
 See Definition \ref{def:finallevels}.
 Note that it is possible that $\operatorname{Cl}^\gamma_{\mathcal A}(\omega)=\emptyset$. We obtain the following objects from the dominant cloud of points:
 \begin{itemize}
 \item The {\em $\gamma$-dominant Newton Polygon ${\mathcal N}^\gamma_{\mathcal A}(\omega)$}  is the positively convex hull of $\operatorname{Cl}^\gamma_{\mathcal A}(\omega)$. Note that ${\mathcal N}^\gamma_{\mathcal A}(\omega)=\emptyset $ if and only if $\operatorname{Cl}^\gamma_{\mathcal A}(\omega)=\emptyset$.
 \item The {\em $\gamma$-dominant critical value $\varsigma^\gamma_{\mathcal A}(\omega)$} is defined by
\[
\varsigma^\gamma_{\mathcal A}(\omega)=\min\{\alpha+s\nu(z);\; (\alpha,s)\in \operatorname{Cl}^\gamma_{\mathcal A}(\omega)\}.
\]
If $\operatorname{Cl}^\gamma_{\mathcal A}(\omega)\ne\emptyset$, then $\varsigma^\gamma_{\mathcal A}(\omega)\leq \gamma$. If $\operatorname{Cl}^\gamma_{\mathcal A}(\omega)=\emptyset$, we put $\varsigma^\gamma_{\mathcal A}(\omega)=\infty$.
\item The {\em $\gamma$-dominant critical line $\operatorname{L}^\gamma_{\mathcal A}(\omega)$} is
\[
\operatorname{L}^\gamma_{\mathcal A}(\omega)=\{(a,b)\in{\mathbb R}^2;a+b\nu(z)=\varsigma^\gamma_{\mathcal A}(\omega)\}.
\]
\item The {\em $\gamma$-dominant critical segment $\operatorname{C}^\gamma_{\mathcal A}(\omega)$} is the intersection
 \[
 \operatorname{C}^\gamma_{\mathcal A}(\omega)=\operatorname{L}^\gamma_{\mathcal A}(\omega)\cap {\mathcal N}^\gamma_{\mathcal A}(\omega).
 \]
 \item The {\em $\gamma$-dominant critical height $\chi^\gamma_{\mathcal A}(\omega)$} is the ordinate of the highest vertex in the $\gamma$-dominant critical segment. This vertex has the form
 \[
 (\varsigma^\gamma_{\mathcal A}(\omega)-\nu(z)\chi^\gamma_{\mathcal A}(\omega), \chi^\gamma_{\mathcal A}(\omega))
 \]
 and is called {\em the $\gamma$-dominant critical vertex}.
 \end{itemize}

\begin{remark}
Assume that $({\mathcal A},\omega)$ is $\gamma$-maximally dominant.
 If ${\mathcal A}\rightarrow{\mathcal A}'$ is an $\ell$-nested transformation,  we have $
 \operatorname{Cl}^\gamma_{{\mathcal A}'}(\omega)=
 \operatorname{Cl}^\gamma_{\mathcal A}(\omega)
 $. In particular
 \[
 {\mathcal N}^\gamma_{{\mathcal A}'}(\omega)={\mathcal N}^\gamma_{\mathcal A}(\omega), \quad
 \operatorname{L}^\gamma_{{\mathcal A}'}(\omega)=\operatorname{L}^\gamma_{\mathcal A}(\omega),\quad
 \operatorname{C}^\gamma_{{\mathcal A}'}(\omega)=\operatorname{C}^\gamma_{\mathcal A}(\omega),
 \quad
 \varsigma^\gamma_{{\mathcal A}}(\omega)=\varsigma^\gamma_{{\mathcal A}'}(\omega).
 \]
 Moreover, we have $
 \varsigma_{\mathcal A}(\omega)\leq \varsigma_{{\mathcal A}'}(\omega)\leq \varsigma^\gamma_{{\mathcal A}}(\omega)=\varsigma^\gamma_{{\mathcal A}'}(\omega)
 $.
 \end{remark}

 The {\em totally recessive case $\operatorname{Cl}^\gamma_{\mathcal A}(\omega)=\emptyset$} corresponds to the property $\varsigma^\gamma_{\mathcal A}(\omega)=\infty$. The {\em dominant case} $\operatorname{Cl}^\gamma_{\mathcal A}(\omega)\ne \emptyset$ corresponds to the property $\varsigma^\gamma_{\mathcal A}(\omega)\leq\gamma$. Let us also recall that we always have the property $\varsigma_{\mathcal A}(\omega)\leq \varsigma^\gamma_{\mathcal A}(\omega)$.

 \begin{proposition}
 \label{prop:caracterizacionpreparacion}
 Assume that $({\mathcal A},\omega)$ is $\gamma$-maximally dominant. Then $({\mathcal A},\omega)$ is $\gamma$-prepared if and only if either $\varsigma_{\mathcal A}(\omega)>\gamma$ or the following statements hold:
 \begin{itemize}
 \item[a)] $\varsigma_{\mathcal A}(\omega)=\varsigma^\gamma_{\mathcal A}(\omega)$.
 \item[b)] $\operatorname{L}\cap \operatorname{Cl}_{\mathcal A}(\omega)=
 \operatorname{L}\cap \operatorname{Cl}^\gamma_{\mathcal A}(\omega)$, where
 $\operatorname{L}=\operatorname{L}^\gamma_{\mathcal A}(\omega)= \operatorname{L}_{\mathcal A}(\omega)$.
 \end{itemize}
 In this case, we have that $
 \operatorname{C}^\gamma_{\mathcal A}(\omega)= \operatorname{C}_{\mathcal A}(\omega)
 $
 and the levels in the cloud of points contributing  to the critical segment are only dominant levels. In particular the critical vertex corresponds to a dominant level.
 \end{proposition}
 \begin{proof} If $\varsigma_{\mathcal A}(\omega)>\gamma$ we see that $({\mathcal A},\omega)$ is $\gamma$-prepared by definition. Thus, we assume $\varsigma_{\mathcal A}(\omega)\leq\gamma$.
 If $({\mathcal A},\omega)$ is $\gamma$-prepared, then each $\omega_s$ is $(\varsigma_{\mathcal A}(\omega)-s\nu(z))$-final. Take a level $\omega_s$ such that
 $\nu_{\mathcal A}(z\omega_s)=\varsigma_{\mathcal A}(\omega)-s\nu(z)$, it is  $(\varsigma_{\mathcal A}(\omega)-s\nu(z))$-dominant.
 This shows a) and b). Conversely, assume a) and b). For any $s$-level $\omega_s$ we have that
 \[
 \nu_{\mathcal A}(z\omega_s)\geq\varsigma_{\mathcal A}(\omega)-s\nu(z).
 \]
 If $\nu_{\mathcal A}(z\omega_s)>\varsigma_{\mathcal A}(\omega)-s\nu(z)$, the level is $(\varsigma_{\mathcal A}(\omega)-s\nu(z))$-recessive. If
 \[
 \nu_{\mathcal A}(z\omega_s)=\varsigma_{\mathcal A}(\omega)-s\nu(z),
 \]
the level is $(\varsigma_{\mathcal A}(\omega)-s\nu(z))$-dominant in view of b).
 \end{proof}
 \begin{corollary}
 \label{cor:dominantcriticalsegment}
  Assume that $({\mathcal A},\omega)$ is $\gamma$-maximally dominant and that one of the following two conditions holds
  \begin{equation*}
(i):\gamma\leq\varsigma_{\mathcal A}(\omega) \mbox{  and } \varsigma_{\mathcal A}^\gamma(\omega)=\infty.\quad (ii):
\varsigma_{\mathcal A}(\omega)= \varsigma_{\mathcal A}^\gamma(\omega)\leq \gamma.
  \end{equation*}
  {\em (This is equivalent to saying that $\varsigma_{\mathcal A}(\omega)\geq \min\{\gamma,\varsigma^\gamma_{\mathcal A}(\omega)\}$)}.
Then, there is an $\ell$-nested transformation ${\mathcal A}\rightarrow{\mathcal A}'$ such that $({\mathcal A}',\omega)$ is $\gamma$-prepared.
 \end{corollary}
 \begin{proof}
 If $\varsigma_{\mathcal A}(\omega)> \gamma$, we are done. Assume that $\varsigma^\gamma_{\mathcal A}(\omega)=\infty$ and $\varsigma_{\mathcal A}(\omega)=\gamma$. Applying Proposition \ref{prop:notfinalformsandcriticalvalue} we obtain the situation where $\varsigma_{\mathcal A}(\omega)>\gamma$ and we are done.

 If $\varsigma^\gamma_{\mathcal A}(\omega)<\infty$, we have $\varsigma_{\mathcal A}(\omega)=\varsigma^\gamma_{\mathcal A}(\omega)\leq \gamma$. We can also apply  Proposition \ref{prop:notfinalformsandcriticalvalue} to the non dominant levels $\omega_s$ such that $(\nu_{\mathcal A}(z\omega_s),s)\in L$, where $L=L^\gamma_{\mathcal A}(\omega)=L_{\mathcal A}(\omega)$. In this way we obtain Property b) in Proposition \ref{prop:caracterizacionpreparacion}.
 \end{proof}
 Proposition \ref{prop:caracterizacionpreparacion} above and Corollary
   \ref{cor:dominantcriticalsegment}
   show that we need to  ``approximate'' the critical segment to the $\gamma$-dominant critical segment in order to
 obtain a prepared situation.

\subsection{Truncated Integrability Condition}
\label{Truncated Integrability Condition}
 Here we develop
the $\gamma$-truncable integrability condition in terms of levels.
Consider the level decomposition given in Equation \eqref{eq:omegalevels}. Let us write
\begin{equation*}
\omega\wedge d\omega=\sum_{s\geq 0}z^s\left( \Theta_s+\frac{dz}{z}\wedge \Delta_s\right),
\end{equation*}
where:
\begin{equation}
\label{eq:condicionesinttruncada}
\Theta_s=\sum_{i+j=s}\eta_i\wedge d\eta_j , \quad \Delta_s=\sum_{i+j=s}j\eta_j\wedge\eta_i+h_id\eta_j+\eta_i\wedge dh_j.
\end{equation}
The condition $\nu_{\mathcal A}(\omega\wedge d\omega)\geq 2\gamma$ is equivalent to saying that
\begin{equation}\label{eq:condicionesinttruncada2}
\nu_{\mathcal A}(\Theta_s)\geq 2\gamma \quad \text{and} \quad \nu_{\mathcal A}(\Delta_s)\geq  2\gamma \quad \text{for any } s\geq 0.
\end{equation}
Lemma \ref{lema:aproximacion} below is our main tool in order to perform the ``aproximated preparation'':
\begin{lemma}
\label{lema:aproximacion}
 Consider an integer number $s\geq 0$ and let $ \delta $ denote the value given by
\[
2\delta=\min\{2\gamma,\min_{s\geq j\geq 1}\{\nu_{\mathcal A}(\eta_{s-j})+\nu_{\mathcal A}(\eta_{s+j})\}\}.
\]
There is an $\ell$-nested transformation ${\mathcal A}\rightarrow {\mathcal A}'$ such that $({\mathcal A}',\eta_s)$ is $\delta$-final.
\end{lemma}
\begin{proof} The fact that $\nu_{\mathcal A}(\Theta_{2s})\geq 2\gamma$ implies that
$
\nu_{\mathcal A}(\eta_s\wedge d\eta_s)\geq 2\delta
$, recall that $\nu_{\mathcal A}(d\eta_{t})\geq \nu_{\mathcal A}(\eta_t)$. We conclude by the
 induction hypothesis.
\end{proof}
\subsection{Planning Polygons}
\label{planning polygons}
 We give here useful facts concerning positively convex polygons. We need them to get the approximate preparation from  Lemma \ref{lema:aproximacion}.

  Let $N\subset {\mathbb R}_{\geq 0}^2$ be a positively convex polygon with vertices in ${\mathbb R}_{\geq 0}\times{\mathbb Z}_{\geq 0}$. For any nonnegative integer number $s$, define the abscissa $\lambda_N(s)\in {\mathbb Z}_{\geq 0}\cup \{\infty\}$ by
 \[
 \lambda_N(s)=\min\{\lambda;\; (\lambda,s)\in N\}.
 \]
 Any $N$ provides a non-increasing sequence $\{\lambda_N(s)\}_{s=0}^\infty$ of abscissas and conversely. Moreover $N=\emptyset$ if and only if $\lambda_N(s)=\infty$ for any $s\geq 0$.
The {\em sharpness $\alpha_N(s)$} is
\[
\alpha_N(s)=\lambda_N(s-1)+\lambda_N(s+1)-2\lambda_N(s),
\]
when $\lambda_N(s)\ne \infty$. We put $\alpha_N(s)=0$ when $\lambda_N(s)=\infty$. Note that $\alpha_N(s)\geq 0$ and that $\alpha_N(s)>0$ if and only $(\lambda_N(s),s)$  is a vertex of $N$. Moreover $\alpha_N(s)=\infty$ if and only if $(\lambda_N(s),s)$ is the lowest vertex of $N$.

 Given two real numbers $\rho\geq 0$ and $\delta >0$, let us denote $H^{+}_\delta(\rho)$, respectively $H_\delta^-(\rho)$, the set of points $(\alpha,\beta)$ in ${\mathbb R}^2_{\geq 0}$ such that $\alpha+\delta\beta\geq \rho$, respectively $\alpha+\delta\beta\leq \rho$. We also put $L_{\delta}(\rho)=H^{+}_\delta(\rho)\cap H^{-}_\delta(\rho)$.

\begin{lemma}
\label{lema:planning}
 Take $\rho,\delta$ as above and consider  $\epsilon>0$. Let $h$ be the smallest integer number such that $h>\rho/\delta$ and consider any $\theta>0$ with
 $
 \theta\leq 2\epsilon/(h+1)(h+2)
 $.
  Then, any positively convex polygon $N$ with vertices in ${\mathbb R}_{\geq 0}\times{\mathbb Z}_{\geq 0}$ such that $N\not\subset H^+_\delta(\rho-\epsilon)$ has a vertex $(\lambda_N(s),s)$ with $\lambda_N(s)<\rho-s\delta$
and $\alpha_N (s)\geq \theta$.
\end{lemma}
\begin{proof} Suppose that the statement is false, so there is a  positively
convex polygon $N\not\subset H^+_\delta(\rho-\epsilon)$ satisfying $\alpha_s<\theta$ for any $s$ such that $\lambda_s<\rho-s\delta$, where $\lambda_s=\lambda_{N}(s)$, $\alpha_s=\alpha_N(s)$.
Denote $\mu_s=\lambda_s-(\rho-s\delta)$ for any $s$ and consider the following two nonnegative integers
\[
b=\min\{s;\, \mu_s<\epsilon\}\quad \text{ and }\quad  t=\max\{s;\, \mu_s<0\}.
\]
We have $0<b\leq t<h$. Recall that $\mu_{t+1}\geq 0$, $\mu_{b}< -\epsilon$ and $\mu_{b}-\mu_{b-1}<0$.
 Let us note that $\alpha_s= (\mu_{s+1}-\mu_s)-(\mu_s-\mu_{s-1})$, for any $s$. We have
\begin{align*}
\mu_{t+1}&=\mu_b+(\mu_{b+1}-\mu_b)+\cdots+(\mu_t-\mu_{t-1})+(\mu_{t+1}-\mu_t)\\
&=\mu_b+(\mu_{b+1}-\mu_b)+\cdots+2(\mu_t-\mu_{t-1})+\alpha_t\\
&=\mu_b+(\mu_{b+1}-\mu_b)+\cdots+3(\mu_{t-1}-\mu_{t-2})+2\alpha_{t-1}+\alpha_t\\
&=\cdots\\
&=\mu_b+(t-b+1)(\mu_{b}-\mu_{b-1})+(t-b+1)\alpha_b+(t-b)\alpha_{b+1}+\cdots+\alpha_t.
\end{align*}
We conclude that
\[
\mu_{t+1}<-\epsilon+\frac{(t-b+1)(t-b+2)}{2}\theta\leq -\epsilon+\frac{(h+1)(h+2)}{2}\theta\leq 0.
\]
This contradicts the fact that $\mu_{t+1}\geq 0$.
\end{proof}

Let us denote $\theta_{\delta,\rho}(\epsilon)=\min\{1,2\epsilon/(h+1)(h+2)\}$.
We introduce now the kind of operations that we perform in the approximate preparation process. Consider two positively convex polygons $N, N'\subset {\mathbb R}_{\geq 0}^2$ whose vertices have integer ordinates. Given  and integer $s\geq 0$, we say that {\em $N'$ is an $s$-planning of $N$} if and only if $N'$ is contained in the positive convex hull of
\[
(\operatorname{Vert}(N)\setminus\{(\lambda_N(s),s)\})\cup \{(\lambda_N(s+1),s+1),(\lambda_N(s-1),s-1)\},
\]
where $\operatorname{Vert}(N)$ denotes the set of vertices of $N$. We say that {\em $N'$ has been obtained from $N$ by a $(\delta,\rho,\epsilon)$-planning operation} if and only if $N'$ is an $s$-planning of $N$, where $\alpha_N(s)\geq\theta_{\delta,\rho}(\epsilon)$ and $\lambda_N(s)<\rho-s\delta$.
\begin{corollary}
Take $\delta,\rho,\epsilon$ as above and a positively convex  polygon $N\subset {\mathbb R}_{\geq 0}$ whose vertices have integer ordinates. Given any sequence
\[
N\rightsquigarrow N'\rightsquigarrow \cdots \rightsquigarrow N^{(t)}
\]
of $(\delta,\rho,\epsilon)$-planning operations we have that either $N^{(t)}\subset H^+_\delta(\rho-\epsilon)$ or it is possible to perform a new $(\delta,\rho,\epsilon)$-planning operation $N^{(t)}\rightsquigarrow N^{(t+1)}$ .
\end{corollary}
\begin{proof}Direct consequence of Lemma \ref{lema:planning}.
\end{proof}

\begin{lemma}
\label{lema:finitudplanning}
Take $\delta,\rho,\epsilon$ as above. There is $b_{\delta,\rho}(\epsilon)\in {\mathbb Z}_{\geq 0}$ such that for any positively convex  polygon $N\subset {\mathbb R}_{\geq 0}$ whose vertices have integer ordinates and any sequence
\[
N\rightsquigarrow N'\rightsquigarrow \cdots \rightsquigarrow N^{(t)}
\]
of $(\delta,\rho,\epsilon)$-planning operations, we have that $t\leq b_{\delta,\rho}(\epsilon)$.
\end{lemma}
\begin{proof} Given a polygon $N$, let us consider the numbers $\beta_N(s)$ defined by
\[
\beta_N(s)=\max\{0, \delta - s \rho - \lambda_{N(s)} \}.
\]
We know that $\beta_N(s)=0$ for $s<0$ and $s\geq h$, where $h$ is the first integer number greater or equal than $\delta / \rho$. If $N\rightsquigarrow N'$ is a $(\delta,\rho,\epsilon)$-planning operation, we have that $\beta_{N'}(s)\leq \beta_N(s)$ for any $s$ and moreover, there is an index $s_0$ such that $\beta_N(s_0)>0$ and
\[\beta_{N'}(s_0)\leq \min\{0, \beta_N(s_0)-\theta_{\delta,\rho}(\epsilon) \}.\]
The level $s_0$ corresponds to the considered vertex in the operation $N\rightsquigarrow N'$. With these properties, it is enough to take
$
b_{\delta,\rho}(\epsilon)\geq (h+1)(\rho+1)/\theta_{\delta,\rho}(\epsilon)
$.
\end{proof}

\subsection{Approximate Preparation}
Proposition \ref{prop:epsilonpreparation} and   Corollary \ref{cor:approximatepreparation}  provide the starting situation to get a prepared situation. Both results are obtained by a ``planning'' of the Newton Polygon, that is possible in view of Lemma \ref{lema:aproximacion}.

 \begin{proposition}
  \label{prop:epsilonpreparation} Consider a $\gamma$-pseudo prepared
 $({\mathcal A},\omega)$ and fix $\epsilon>0$. There is an $\ell$-nested transformation ${\mathcal A}\rightarrow {\mathcal A}'$ such that  $\varsigma_{{\mathcal A}'}(\omega)>\min\{\gamma,\varsigma^\gamma_{\mathcal A}(\omega)-\epsilon\}$.
 \end{proposition}
 \begin{proof}  In order to simplify notation, let us denote $N={\mathcal N}_{\mathcal A}(\omega)$, $\delta=\nu(z)$, $\rho=\min\{\gamma,\varsigma^\gamma_{\mathcal A}(\omega)-\epsilon\}$ and
$
\lambda_s=\lambda_N(s)\leq \nu_{\mathcal A}(z\omega_s).
$
Let us note that if $\varsigma^\gamma_{\mathcal A}(\omega)=\infty$, then $\rho=\gamma$ and if
$\varsigma^\gamma_{\mathcal A}(\omega)\ne\infty$ then $\rho=\varsigma^\gamma_{\mathcal A}(\omega)-\epsilon\leq \gamma-\epsilon<\gamma$.

Let us do an argument by contradiction, assuming that there is no $\ell$-nested transformation
 ${\mathcal A}\rightarrow {\mathcal A}'$ such that $\varsigma_{{\mathcal A}'}(\omega)> \rho$.

Up to performing a suitable $\ell$-nested transformation, the following properties hold:
\begin{itemize}
\item[a)] For any level $s$ such that $\lambda_s\leq  \rho-s\delta$ there is no new nested transformation with  $\lambda'_s >\rho-s\delta$. That is, we  assume that we have the minimum possible number of levels $s$ with $\lambda_s\leq \rho-s\delta$.
\item[b)] For any level $s$, we have that $\lambda_s\ne\rho-s\delta$. If we have a level $s$ with $\lambda_s=\rho-s\delta$, it is not a dominant level and we can apply
    Proposition \ref{prop:notfinalformsandcriticalvalue} to  obtain $\lambda'_s>\gamma-s\delta$ in contradiction with the minimality given in a).
\item[c)] $\lambda_0>\rho$. Recall that $\eta_0=\omega_0$ is $\gamma$-final.
\end{itemize}
The above properties a),b) and c) are still true under any further $\ell$-nested transformation.
 Let us choose $\tilde\rho$ such that $\rho<\tilde\rho<\varsigma^\gamma_{\mathcal A}(\omega)$ such that for any $s$ with $\lambda_s>\rho-s\delta$ we have $\lambda_s>\tilde\rho-s\delta$. Note that this property is stable under any further $\ell$-nested transformation.
 In view of the above reductions of the problem, we have the following properties under any $\ell$-nested transformation ${\mathcal A}\rightarrow {\mathcal A}'$:
  \begin{itemize}
  \item[i)] $\lambda_0>\tilde\rho$ and $\lambda'_0>\tilde\rho$.
 \item[ii)] For any $s$ such that $\lambda_s\geq \rho-s\delta$, then $\lambda_s>\tilde\rho-s\delta$ and $\lambda'_s>\tilde\rho-s\delta$.
 \item[iii)] For any $s$ such that $\lambda_s \leq \rho-s\delta$, then
 $\lambda_s<\rho-s\delta$ and $\lambda'_s<\rho-s\delta$.
 \end{itemize}
 Let us consider now the following ``Planning Property'':
 \begin{quote}
  (P): {\em For any $s$ with $\lambda_s\leq \rho-s\delta$, there is an $\ell$-nested transformation ${\mathcal A}\rightarrow {\mathcal A}'$ such that the transformed Newton Polygon $N'$ is an $s$-planning of $N$.}
 \end{quote}
If (P) is true, we end as follows. Take $\varepsilon = \tilde\rho-\rho$. By Lemma \ref{lema:planning} we can perform a $(\tilde\gamma,\delta,\varepsilon)$-transformation $N'$ of $N$ induced by an $\ell$-nested transformation ${\mathcal A}\rightarrow {\mathcal A}'$. The transformed polygon $N'$ is still in the same situation in view of properties i), ii) and iii). We repeat indefinitely the operation. This contradicts Lemma \ref{lema:finitudplanning}.

Now, let us show that Property (P) holds. Take $s$ such that $\lambda_s\leq \rho-s\delta$, note that $s\geq 1$. For any $j\geq 1$ we have that
\[
\lambda_{s+1}+\lambda_{s-1}\leq \lambda_{s+j}+\lambda_{s-j}\leq \nu_{\mathcal A}(\eta_{s+j})+\nu_{\mathcal A}(\eta_{s-j}).
\]
Let us recall that $\nu_{\mathcal A}(\Theta_{2s})\geq 2\gamma$, see Equation \eqref{eq:condicionesinttruncada2}. We deduce that
\begin{equation}
\label{eq:truncatedinductioncondition}
\nu_{\mathcal A}(\eta_s\wedge d\eta_s)\geq 2\varrho,
\end{equation}
where $2\varrho=
\min\{2\gamma, \lambda_{s+1}+\lambda_{s-1}\}$. On the other hand, let us remark that
\begin{equation}
\label{eq:abscisanivels}
\nu_{\mathcal A}(\eta_s\wedge d\eta_s)<2(\rho-s\delta).
\end{equation}
If Equation \eqref{eq:abscisanivels} does not hold, by induction hypothesis, there is a suitable $\ell$-nested transformation such  that $({\mathcal A}',\eta_s)$ is $(\rho-s\delta)$-final. Recalling that $({\mathcal A}',h_s)$ is also  $(\rho-s\delta)$-final,  the level $\omega_s$ is $(\rho-s\delta)$-final with respect to ${\mathcal A}'$. But we know that $\lambda'_s<\rho-s\delta$, in particular it should be a dominant level, contradiction with the fact that $\rho<\varsigma^\gamma_{\mathcal A}(\omega)$.

Then Equation \eqref{eq:abscisanivels} holds and we deduce that
\[
2\varrho=\lambda_{s+1}+\lambda_{s-1}\quad  \text{ and }\quad   \rho\leq \gamma-s\delta.
\]
By a new application of the induction hypothesis in view of Equation \eqref{eq:truncatedinductioncondition}, we can perform an $\ell$-nested transformation ${\mathcal A}\rightarrow {\mathcal A}'$  such  that $({\mathcal A}',\eta_s)$ is $\varrho$-final. Moreover, since $\varrho\leq \rho-s\delta$ we have that $({\mathcal A}',h_s)$ is also $\varrho$-final and hence $({\mathcal A}',\omega_s)$ is also $\varrho$-final. We know that $({\mathcal A},\omega_s)$ is not $\varrho$-final dominant, since otherwise it would be $(\rho-s\delta)$-final dominant and $\rho<\varsigma^\gamma_{\mathcal A}(\omega)$. We conclude that $\lambda'_s\geq (\lambda_{s+1}+\lambda_{s-1})/2$ and thus $N'$ has been obtained by an $s$-planning of $N$.
 \end{proof}

 \begin{corollary}[Approximate Preparation]
  \label{cor:approximatepreparation} Let us assume that $({\mathcal A},\omega)$ is  $\gamma$-pseudo prepared. We have the following properties:
 \begin{itemize}
 \item {\em Recessive Preparation:} If ${\varsigma}^\gamma_{\mathcal A}(\omega)=\infty$, there is an $\ell$-nested transformation ${\mathcal A}\rightarrow {\mathcal A}'$ such that $\varsigma_{{\mathcal A}'}(\omega)>\gamma$. In particular $({\mathcal A}',\omega)$ is $\gamma$-prepared.
 \item {\em Approximate Preparation:} Assume that ${\varsigma}^\gamma_{\mathcal A}(\omega)\ne\infty$ and consider $\epsilon>0$. There is an $\ell$-nested transformation ${\mathcal A}\rightarrow {\mathcal A}'$ such that
     $
     \varsigma_{{\mathcal A}'}(\omega)>\varsigma^\gamma_{\mathcal A}(\omega)-\epsilon
     $.
 \end{itemize}
 \end{corollary}

\subsection{Preparation}
\label{subsection:preparation}
We complete here the proof of Theorem \ref{teo:preparation}.

Denote $\varsigma=\varsigma^\gamma_{\mathcal A}(\omega)$.
By Corollary \ref{cor:approximatepreparation}, we have only to consider the case when $\varsigma\ne\infty$. Thus, we assume that $\varsigma<\infty$ and hence $\varsigma\leq \gamma$. By Corollary  \ref{cor:dominantcriticalsegment}, we have only to show that
$\varsigma_{{\mathcal A}'}(\omega)=\varsigma$ after a suitable
$\ell$-nested transformation ${\mathcal A}\rightarrow {\mathcal A}'$.

We follow an argument by contradiction, assuming that Theorem \ref{teo:preparation} does not hold for a given $({\mathcal A},\omega)$. Denote by $ s_0\leq s_1$ the levels corresponding to the vertices of the dominant critical segment and by $[s_0,s_1]=\{s; s_0\leq s\leq  s_1\}$.
There is a positive $\mu>0$ such that after a suitable $\ell$-nested transformation, the following  holds:
\begin{itemize}
\item[1)] The $s_0$-level $\omega_{s_0}$ is divisible by $\boldsymbol{x}^I$, with $\nu_{\mathcal A}(\boldsymbol{x}^I)=\varsigma-s_0\delta$. That is, we have
    \[
    \eta_{s_0}= \boldsymbol{x}^I\eta^*_{s_0},\quad h_{s_0}=\boldsymbol{x}^I h^*_{s_0} .
    \]
    (See Proposition \ref{prop:monlocprincipalization}).
\item[2)]  For any $s$-level such that $\nu_{\mathcal A}(z\omega_s)< \varsigma-s\delta$ and any $\ell$-nested transformation ${\mathcal A}\rightarrow{\mathcal A}'$, we have that
$\nu_{{\mathcal A}'}(z\omega_s)< \varsigma-s\delta$.
\item[3)] For any $s$-level we have that $\nu_{\mathcal A}(h_s)\geq \varsigma-s\delta$. Moreover, if $\nu_{\mathcal A}(h_s)>\varsigma-s\delta$ then $\nu_{\mathcal A}(h_s)>\varsigma-s\delta+\mu$.
\item[4)] For any $s\notin [s_0,s_1]$ such that $\nu_{\mathcal A}(z\omega_s)\geq \varsigma-s\delta$ we have
$\nu_{\mathcal A}(z\omega_s)> \varsigma-s\delta+\mu$ and
$\nu_{{\mathcal A}'}(z\omega_s)> \varsigma-s\delta+\mu$ after any further $\ell$-nested transformation.
\item[5)] For any $s\in [s_0,s_1]$ we have one of the following situations:
\begin{itemize}
\item[a)] $\nu_{\mathcal A}(z\omega_s)=\varsigma-s\delta$ and this property is stable under any further $\ell$-nested transformation. This is the case for $s=s_0$, $s=s_1$ and the $(\varsigma-s\delta)$-dominant levels with $s\in [s_0,s_1]$.
\item[b)] $\nu_{\mathcal A}(z\omega_s)<\varsigma-s\delta$ and this property is stable under any further $\ell$-nested transformation.
\item[c)] $\nu_{\mathcal A}(z\omega_s)>\varsigma-s\delta+\mu$ and this property is stable under any further $\ell$-nested transformation.
\end{itemize}
\end{itemize}
Moreover, in view of Corollary \ref{cor:approximatepreparation}, given $\epsilon>0$ we can perform an $\ell$-nested transformation such that $\nu_{\mathcal A}(z\omega_s)>\varsigma-s\delta-\epsilon$ for any $s$ and this property is stable under any further $\ell$-nested transformation.

Let us note that our contradiction hypothesis states that there is at least one level $\omega_s$ such that $\nu_{\mathcal A}(z\omega_s)<\varsigma-s\delta$. Let $\tilde s$ denote  the minimum of the indices $ s $ such that $\nu_{\mathcal A}(z\omega_s)<\varsigma-s\delta$. We remark that $\tilde s\geq 1$ since the level $\omega_0=\eta_0$ is $\gamma$-final.
We consider the two main situations $\tilde s<s_0$ and $\tilde s>s_0$ that we call respectively the {\em recessive case} and the {\em dominant case}.

Let us consider the recessive case $\tilde s<s_0$. Taking $0<\epsilon<\mu$ and in view of the fact that $\nu_{\mathcal A}(\Theta_{2\tilde s})\geq 2\gamma$, we deduce that
\[
\nu_{\mathcal A}(\eta_{\tilde s}\wedge d \eta_{\tilde s})\geq 2(\varsigma-\tilde s\delta).
\]
By induction hypothesis we can perform an $\ell$-nested transformation in such a way that the level $\tilde s$ should be $(\varsigma-\tilde s\delta)$-final. Since $\tilde s< s_0$, we obtain $\nu_{\mathcal A}(\omega_{\tilde s})>\varsigma-s\delta$, this is the desired contradiction.

Now, consider the dominant case $\tilde s>s_0$. For any $j<s_0$ we have
\begin{equation*}
\begin{array}{llc}
\nu_{\mathcal A}(\eta_{\tilde s+s_0-j}\wedge\eta_j)>\tilde\varsigma,&
\nu_{\mathcal A}(h_{\tilde s+s_0-j}d\eta_j)>\tilde\varsigma,&
\nu_{\mathcal A}(h_jd\eta_{\tilde s+s_0-j})>\tilde\varsigma,\\
\nu_{\mathcal A}(\eta_{\tilde s+s_0-j}\wedge dh_j)>\tilde\varsigma,&
\nu_{\mathcal A}(\eta_j\wedge dh_{\tilde s+s_0-j})>\tilde\varsigma,&
\end{array}
\end{equation*}
where $\tilde\varsigma=2\varsigma-(\tilde s+s_0)\delta+(\mu-\epsilon)$.
Recalling that $\nu_{\mathcal A}(\Delta_{\tilde s+s_0})\geq 2\gamma$, see Equation \eqref{eq:condicionesinttruncada}, we deduce that
\begin{equation}
\label{eq:cotasdelta}
\nu_{\mathcal A}\left(
\sum_{j=s_0}^{\tilde s}(j\eta_j\wedge\eta_{\tilde s+s_0-j}+h_{\tilde s+s_0-j}d\eta_j+\eta_{\tilde s+s_0-j}\wedge dh_j)
\right)> \tilde\varsigma.
\end{equation}
Moreover, for any $s_0\leq j\leq \tilde s$ we have that $\nu_{\mathcal A}(h_j)\geq \varsigma-j\delta$. Otherwise, by performing an $\ell$-nested transformation we can obtain that $\nu_{\mathcal A}(z\omega_j)>\varsigma-j\delta-\epsilon>\nu_{\mathcal A}(h_j)$, contradicting the definition of the dominant abscissa. Noting that $\nu(z\omega_j)\geq \varsigma-j\delta$ for any $j<\tilde s$, we conclude from Equation \eqref{eq:cotasdelta} that
\begin{equation}
\label{eq:cotasdelta2}
\nu_{\mathcal A}
((\tilde s-s_0)\eta_{\tilde s}\wedge\eta_{s_0}+h_{s_0}d\eta_{\tilde s}+\eta_{\tilde s}\wedge dh_{s_0})
\geq 2\varsigma-(\tilde s+s_0)\delta> 2(\varsigma-\tilde s\delta).
\end{equation}
Now, we are going to consider separately the cases where $\nu_{\mathcal A}(h_{s_0})>\varsigma-s_0\delta$ and $\nu_{\mathcal A}(h_{s_0})=\varsigma-s_0\delta$.

$\bullet$ Assume first that $\nu_{\mathcal A}(h_{s_0})>\varsigma-s_0\delta$ and hence $\nu_{\mathcal A}(h_{s_0})>\varsigma-s_0\delta+\mu$ . In this case, we have
\begin{equation}
\label{eq:cotasdelta3}
\nu_{\mathcal A}
(\eta_{\tilde s}\wedge\eta_{s_0})
\geq 2\varsigma-(\tilde s+s_0)\delta.
\end{equation}
Moreover, $ \eta_{s_0}= {\boldsymbol{x}^I}\eta^*_{s_0} $ where $\eta^*_{s_0}$ is $0$-final dominant and $\nu_{\mathcal A}({\boldsymbol{x}^I})= \varsigma-s_0\delta$. Dividing by ${\boldsymbol{x}^I}$ in Equation \eqref{eq:cotasdelta3}  we obtain
\[
\nu_{\mathcal A}(\eta_{\tilde s}\wedge\eta^*_{s_0})\geq \varsigma-\tilde s\delta.
\]
By Proposition  \ref{prop:trucateddivision}, there is
$F\in k[[\boldsymbol{x},y_1,y_2,\ldots,y_{\ell}]]$ and $\tilde\eta_{\tilde s}$ such that
\[
\eta_{\tilde s}=F\eta^*_{s_0}+\tilde\eta_{\tilde s};
\quad \nu_{\mathcal A}(\tilde\eta_{\tilde s})> \varsigma-\tilde s \delta.
\]
Up to a further $\ell$-nested transformation (that does not affect to this general situation) we have that $F$ is $(\varsigma-\tilde s\delta)$-final and hence $\eta_{\tilde s}$ is also $(\varsigma-\tilde s\delta)$-final. In particular $\nu_{\mathcal A}(\eta_{\tilde s})\geq \varsigma$, this is the desired contradiction.

$\bullet$ Assume now that $\nu_{\mathcal A}(h_{s_0})=\varsigma-s_0\delta$. Let us show that after a suitable $\ell$-nested transformation we can assume that
\begin{equation}
\label{eq:descomposicionfinal}
\eta_{\tilde s}=\theta+\tilde \eta,\quad \nu_{\mathcal A}(\tilde \eta)\geq \varsigma-\tilde s\delta.
\end{equation}
where $\theta\wedge d\theta=0$. This ends the proof as follows. By a suitable new $\ell$-nested transformation the integrable form $\theta$ may assumed to be $(\varsigma-\tilde s\delta)$-final. If
$\nu_{\mathcal A}(\theta)<\varsigma-\tilde s\delta$ we have that
$\eta_{\tilde s}$
is $(\varsigma-\tilde s\delta)$-final dominant with $\nu_{\mathcal A}(\eta_{\tilde s})<\varsigma-\tilde s\delta$. This is not possible. Then, we necessarily have that $\nu_{\mathcal A}(\theta)\geq \varsigma-\tilde s\delta$ and hence
$
\nu_{\mathcal A}(\eta_{\tilde s})\geq \varsigma-\tilde s\delta.
$
This is the desired contradiction.

Now, let us show that we can obtain the decomposition given in Equation \eqref{eq:descomposicionfinal}.
We know that
$
h_{s_0}=h^*_{s_0}\boldsymbol{x}^I
$, where
and $h^*_{s_0}\in \widehat{\mathcal O}_{M,P}$ is a unit and $\nu(\boldsymbol{x}^I)=\varsigma-s_0\delta$.
Let us take $\alpha$ given by
\begin{equation*}
\alpha= \frac{(s_0-\tilde s)\eta_{s_0}-dh_{s_0}}{h_{s_0}}= \frac{(s_0-\tilde s)\eta^*_{s_0}-dh^*_{s_0}}{h^*_{s_0}}- \frac{d\boldsymbol{x}^I}{\boldsymbol{x}^I}.
\end{equation*}
 From Equation \eqref{eq:cotasdelta2} we have that
\begin{equation}
\label{eq:cotas3}
\nu_{\mathcal A}(\alpha\wedge \eta_{\tilde s}+d\eta_{\tilde s})\geq \varsigma-\tilde s\delta.
\end{equation}
After a suitable $\ell$-nested transformation, we have two possibilities:
\begin{itemize}
\item[a)] There is $\mu'>0$ such that $\nu_{{\mathcal A}'}(\alpha)>\mu'$ for any further $\ell$-nested trasformation.
\item[b)] $\alpha$ is $0$-final dominant.
\end{itemize}
$-$ Assume that we are in the situation a), that is  $\nu_{\mathcal A}(\alpha)>\mu'>0$. By performing a suitable $\ell$-nested transformation associated to $\epsilon'<\mu'$, we obtain
$\nu_{\mathcal A}(\eta_{\tilde s})\geq \varsigma-\tilde s\delta -\epsilon'$. This implies that
$
\nu_{\mathcal A}(\alpha\wedge\eta_{\tilde s})\geq \varsigma-\tilde s\delta
$
and by Equation \eqref{eq:cotas3} we conclude that $\nu_{\mathcal A}(d\eta_{\tilde s})\geq \varsigma-\tilde s\delta$. Applying Proposition \ref{prop:logpoincare2}, we find that
\[
\eta_{\tilde s}=\theta +\tilde \eta,\quad \theta= df+\frac{d\boldsymbol{x}^{\boldsymbol{\lambda}}}{\boldsymbol{x}^{\boldsymbol{\lambda}}},
\quad \nu_{\mathcal A}(\tilde \eta)\geq \varsigma-\tilde s\delta.
\]
We see that $d\theta=0$, in particular $\theta$ is integrable.

$-$ Assume now that we are in the situation b), that is $\alpha$ is $0$-final dominant. We can write
$
\alpha=\beta+\tilde\beta
$,
where $\beta$ is $0$-final dominant with $d\beta=0$ and $\tilde\beta$ is not $0$-final dominant. To see this, it is enough to take as $\beta$ the ``residual part'' of $\alpha$.
By applying Proposition \ref{prop:notfinalformsandcriticalvalue}, after a suitable $\ell$-nested transformation, we can assume that $\nu_{\mathcal A}(\tilde \beta)>\mu'>0$. Moreover, we can perform a new $\ell$-nested transformation associated to $0<\epsilon'<\mu'$ in such a way that $\nu_{\mathcal A}(\eta_{\tilde s})\geq \varsigma-\tilde s \delta -\epsilon'$. By Equation \eqref{eq:cotas3} we conclude that
\begin{equation*}
\nu_{\mathcal A}(\beta\wedge\eta_{\tilde s}+d\eta_{\tilde s})\geq \varsigma-\tilde s\delta.
\end{equation*}
Note that $d_\beta\eta_{\tilde s}=\beta\wedge\eta_{\tilde s}+d\eta_{\tilde s}$
By Corollary \ref{cor:logpoincare3} we have that
$
\eta_{\tilde s}=\theta+\tilde\eta
$,
where $\theta\wedge d\theta=0$ and $\nu_{\mathcal A}(\tilde\eta)\geq \varsigma-\tilde s\delta$.

This ends the proof of Preparation Theorem \ref{teo:preparation}.

\section{Control by the Critical Height}
\label{Control by the Critical Height}
In this section we end the proof of Theorem \ref{teo:inductivestatement}.  We take the assumptions and notations as in Section \ref{Truncated Local Uniformization}. Thus, we fix $\ell\geq 0$, we assume  the induction  hypothesis, that is, Theorem \ref{teo:inductivestatement} is true for $1$-forms $\eta$ with  $I_{\mathcal A}(\eta)\leq \ell$.  We consider
$\gamma$-truncated formal foliated space $({\mathcal A},\omega)$ with $I_{\mathcal A}(\omega)=\ell+1$ and we intend to show the existence of a $(\ell+1)$-nested transformation ${\mathcal A}\rightarrow {\mathcal B}$ such that $({\mathcal B},\omega)$ is $\gamma$-final.

 By Preparation Theorem  \ref{teo:preparation}, we start with a strictly $\gamma$-prepared $({\mathcal A},\omega)$.

The proof of Theorem \ref{teo:inductivestatement} follows from the control  of the critical value $\varsigma_{\mathcal A}(\omega)$ and the critical height $\chi_{\mathcal A}(\omega)$ under a type of $(\ell+1)$-nested transformations that we call {\em normalized transformations}.

\begin{remark}
If $({\mathcal A},\omega)$ is strictly $\gamma$-prepared, the critical value and critical height determine the ``dominant'' ones as follows:
\begin{itemize}
\item If $\varsigma_{\mathcal A}(\omega)\leq\gamma$, then
$\varsigma^\gamma_{\mathcal A}(\omega)=\varsigma_{\mathcal A}(\omega)$
 and
 $\chi^\gamma_{\mathcal A}(\omega)=\chi_{\mathcal A}(\omega)$.
 \item If $\varsigma_{\mathcal A}(\omega)>\gamma$, then
 $\varsigma^\gamma_{\mathcal A}(\omega)=\infty$.
\end{itemize}
\end{remark}

\subsection{Normalized Transfomations}  The normalized transformations are transformations between strictly $\gamma$-prepared $\gamma$-truncated formal foliated spaces.

\begin{definition}
 \label{def:normalized transformation}
 Consider a strictly $\gamma$-prepared $({\mathcal A},\omega)$, with $I_{\mathcal A}(\omega)=\ell+1$. A transformation
$
({\mathcal A},\omega)\rightarrow({\mathcal A}^\star,\omega)
$
is said to be
\begin{itemize}
  \item A {\em normalized Puiseux's package}, if it is composition of a $(\ell+1)$-Puiseux's package ${\mathcal A}\rightarrow {\mathcal A'}$, followed by a strict  $\gamma$-preparation $({\mathcal A}',\omega)\rightarrow ({\mathcal A}^\star,\omega)$.
  \item A {\em normalized coordinate change}, if it is
composition of a $(\ell+1)$-coordinate change  ${\mathcal A}\rightarrow {\mathcal A'}$  followed by  a strict  $\gamma$-preparation $({\mathcal A}',\omega)\rightarrow ({\mathcal A}^\star,\omega)$, where the ramification index of $\mathcal A$ is equal to one, or ${\mathcal A}'={\mathcal A}$.
\end{itemize}
{\em A normalized transformation} is a finite composition of normalized Puiseux's packages and normalized coordinate changes.
\end{definition}

\begin{remark}
\label{rk: degeneratenormalizedtransforamtions}
Any $\ell$-nested transformation ${\mathcal A}\rightarrow {\mathcal A}^\star$ supports a normalized coordinate change of a degenerate type
\[({\mathcal A},\omega)\rightarrow ({\mathcal A}^\star,\omega)\]
where the $(\ell+1)$-coordinate change is just the identity ${\mathcal A}'={\mathcal A}$. In particular, any $0$-nested transformation (finite composition of independent blow-ups) supports a normalized transformation.
\end{remark}

\begin{remark}
\label{rk:nontruncatednormalizedtransformations}
In  Definition \ref{def:normalized transformation} the truncation value $\gamma$ is implicit. When we need to exhibit it, we say that $({\mathcal A},\omega)\rightarrow ({\mathcal B},\omega)$ is a {\em $\gamma$-normalized transformation}. This additional information is necessary in the last section, where we consider non-truncated situations.
\end{remark}
\subsection{Control of the Critical Height and Critical Value}
 \label{Statements for the Control of the Critical Height and Value}
 Let us state here the results we prove in next subsections, in order to prove Theorem \ref{teo:inductivestatement}. We start with a strictly $\gamma$-prepared $({\mathcal A},\omega)$.
\begin{proposition}
\label{prop:recessive case} Assume that $\varsigma_{\mathcal A}(\omega)>\gamma$ and let $({\mathcal A},\omega)\rightarrow ({\mathcal A}^\star,\omega)$ be a normalized Puiseux's package. Then $({\mathcal A}^\star,\omega)$ is $\gamma$-final.
\end{proposition}

\begin{proposition}
\label{prop:critical heightcero}
Assume that $\chi_{\mathcal A}(\omega)=0$ and let $({\mathcal A},\omega)\rightarrow ({\mathcal A}^\star,\omega)$ be a normalized Puiseux's package. Then $({\mathcal A}^\star,\omega)$ is $\gamma$-final.
\end{proposition}

\begin{proposition}
\label{prop:stability critical height}
Assume that $\varsigma_{\mathcal A}(\omega)\leq \gamma$ and let $({\mathcal A},\omega)\rightarrow ({\mathcal B},\omega)$ be a normalized transformation. Then either $\varsigma_{{\mathcal B}}(\omega)>\gamma$ or
$\chi_{{\mathcal B}}(\omega)\leq \chi_{{\mathcal A}}(\omega)$.
\end{proposition}

Once Propositions
\ref{prop:recessive case},
\ref{prop:critical heightcero} and
\ref{prop:stability critical height} are proved,
we focus in the strictly $\gamma$-prepared $({\mathcal A},\omega)$ for which we have
$\varsigma_{{\mathcal B}}(\omega)\leq\gamma$ and $\chi_{{\mathcal B}}(\omega)=\chi_{{\mathcal A}}(\omega)\geq 1$,
under any normalized transformation $({\mathcal A},\omega)\rightarrow ({\mathcal B},\omega)$. This justifies next definition
 \begin{definition}
\label{def:fixedcriticalheight} Let $\chi\geq 1$ be an integer number. We say that $({\mathcal A},\omega)$ has the property of {\em $\chi$-fixed critical height} if
$\varsigma_{{\mathcal B}}(\omega)\leq \gamma$ and $\chi_{{\mathcal B}}(\omega)=\chi$,
for any normalized transformation $({\mathcal A},\omega)\rightarrow({\mathcal B},\omega)$.
\end{definition}

Because of the results in Propositions \ref{prop:recessive case},
\ref{prop:critical heightcero} and
\ref{prop:stability critical height}, we see that Theorem \ref{teo:inductivestatement} is a consequence of Propositions
\ref{prop:critical height2} and
\ref{prop:critical height1} below:

\begin{proposition}
\label{prop:critical height2} Let us consider an integer number $\chi\geq 2$. There is no strictly $\gamma$-prepared $({\mathcal A},\omega)$ with the property of $\chi$-fixed critical height.
\end{proposition}

\begin{proposition}
\label{prop:critical height1} There is no strictly $\gamma$-prepared $({\mathcal A},\omega)$ with the property of $1$-fixed critical height.
\end{proposition}

When we have  $\varsigma_{{\mathcal A}^\star}(\omega)\leq\gamma$ and $\chi_{{\mathcal A}^\star}(\omega)=\chi_{{\mathcal A}}(\omega)\geq 1$ under a normalized Puiseux's package, we see that $({\mathcal A},\omega)$ satisfies certain {\em resonance conditions } \textbf{r1} or \textbf{r2}. Condition  \textbf{r1} occurs ``at most once'' and only when $\chi_{\mathcal A}(\omega)=1$. On the other hand,  condition \textbf{r2} implies that the ramification index of ${\mathcal A}$ is $1$.  In this way, we
 we obtain situations where it is possible to perform Tschirnhausen transformations to ``escape'' from a situation of  $\chi$-fixed critical height.

\subsection{Reduced Part of a Level}
We consider here a $\gamma$-truncated formal foliated space $({\mathcal A},\omega)$, not necessarily strictly $\gamma$-prepared.
 The {\em reduced part $\bar\omega^{\mathcal A}_s$ of the $s$-level $\omega_s$ of $({\mathcal A},\omega)$  } appears in many of our computations.

 Let us recall the definition of the abscissa
 $\lambda_{{\mathcal A},\omega}(s)$ given in Subsection \ref{planning polygons}:
 \begin{equation*}
 \lambda_{{\mathcal A},\omega}(s)=
 \lambda_{{\mathcal N}_{\mathcal A}(\omega)}(s)=
 \min\{
 \lambda;\;(\lambda,s)\in {\mathcal N}_{\mathcal A}(\omega)
 \}.
 \end{equation*}
 Note that $\lambda_{{\mathcal A},\omega}(s)=\varsigma_{\mathcal A}(\omega)-s\nu(z)$,  when $s$ corresponds to a level in the critical segment. If $\nu_{\mathcal A}(z\omega_s)>\lambda_{{\mathcal A},\omega}(s)$, we write
$\bar\omega^{\mathcal A}_s=0$. Assume $\nu_{\mathcal A}(z\omega_s)=\lambda_{{\mathcal A},\omega}(s)$ and write
\[
 \omega_s=\boldsymbol{x}^I\omega^*_s+\tilde\omega_s,\quad
 \nu_{\mathcal A}(\tilde\omega_s)> \lambda_{{\mathcal A},\omega}(s),
 \quad
 \nu(\boldsymbol{x}^I)=\lambda_{{\mathcal A},\omega}(s),
 \]
 where
\[
0\ne \omega^*_s= \sum_{i=1}^rf^*_{is}\frac{dx_i}{x_i}+\sum_{j=1}^\ell g^*_{js}dy_j+h^*_{s}\frac{dz}{z}.
\]
Denote $\bar{f}^*_{is}, \bar{h}^*_s\in k$ the classes modulo the maximal ideal of $k[[\boldsymbol{x},\boldsymbol{y}_{\leq\ell}]]$. Write
\begin{equation*}
\bar\omega^*_s=
\sum_{i=1}^r\bar{f}^*_{is}\frac{dx_i}{x_i}+\bar{h}^*_{s}\frac{dz}{z}.
\end{equation*}
The {\em reduced part $\bar{\omega}^{\mathcal A}_s$} of the level $\omega_s$  is defined by $\bar{\omega}^{\mathcal A}_s=\bar\omega^*_s$. Let us precise the nature of $\bar{\omega}^{\mathcal A}_s$. If we consider the $k$-vector space $\overline{\Omega}^1_{\mathcal A}$ defined by
\[
\overline{\Omega}^1_{\mathcal A}= \frac{\Omega^1_{\mathcal A}[\log z]}{
\hat{\mathfrak m}_{\mathcal A}\Omega^1_{\mathcal A}[\log z]+\sum_{j=1}^\ell \Omega^0_{\mathcal A}dy_j
},
\]
then $\bar{\omega}^{\mathcal A}_s\in \overline{\Omega}^1_{\mathcal A}$; we will not insist on this formalism. Anyway, we have an isomorphism of $k$-vector spaces $\overline{\Omega}^1_{\mathcal A}\rightarrow k^{r+1}$ such that the image
{$\overline{\operatorname{vec}}^{\mathcal A}_s(\omega)$} of $\overline\omega^{\mathcal A}_s$ is the vector
\begin{equation*}
\overline{\operatorname{vec}}^{\mathcal A}_s(\omega)=(\bar{f}^*_{1s}, \bar{f}^*_{2s},\ldots, \bar{f}^*_{rs},\bar{h}^*_{s})\in k^r\times k.
\end{equation*}
\begin{remark}
 \label{rk:dominantreducedpart}
 The level $\omega_s$ is $\lambda_{{\mathcal A},\omega}(s)$-dominant if and only if $\bar\omega^{\mathcal A}_s\ne 0$ or, equivalently, if and only if we have $\overline{\operatorname{vec}}^{\mathcal A}_s(\omega)\ne 0$.
\end{remark}
\subsection{Effect of Coordinate Changes}
In Proposition
\ref{prop:effect of a coordinate change},
we describe the effect of a coordinate change on the reduced part of a level higher of equal than the critical height.

First, let us give some elementary remarks on positively convex polygons ${N}\subset{\mathbb R}_{\geq 0}^2$ given by a cloud of points in ${\mathbb R}_{\geq 0}\times{\mathbb Z}_{\geq 0}$, se also Subsection \ref{planning polygons}. For any $\delta>0$, let us denote
\[
\varsigma_{\delta}(N)=\min\{\alpha+\delta \beta;\; (\alpha,\beta)\in {N}\}= \max\{\rho;\; {N}\subset H^+_{\delta}(\rho)\}.
\]
The {\em $\delta$-critical vertex} is the highest vertex of $N$ such that $\alpha+\delta\beta=\varsigma_\delta(N)$ and the {\em $\delta$-critical height $\chi_\delta(N)$} is the ordinate of the $\delta$-critical vertex.

\begin{remark} If $N={\mathcal N}_{\mathcal A}(\omega)$ and $\delta=\nu(z)$, then we have that $\varsigma_{\mathcal A}(\omega)=\varsigma_{\delta}(N)$ and $\chi_{\mathcal A}(\omega)=\chi_\delta(N)$.
\end{remark}
\begin{lemma}
\label{lema:ordenadas supracriticas} Let $N, N'\subset{\mathbb R}^2_{\geq 0}$ be two positively convex polygons with vertices in $(\mathbb R)_{\geq 0}\times{\mathbb Z}_{\geq 0}$. Let us consider $\delta_0>0$. The following statements are equivalent:
\begin{itemize}
\item[1)] $\varsigma_{\delta}(N)=\varsigma_{\delta}(N')$ for any $\delta\leq \delta_0$.
\item[2)] $\chi_{\delta_0}(N)=\chi_{\delta_0}(N')$ and $\lambda_s(N)=\lambda_s(N')$ for any $s\geq \chi_{\delta_0}(N)$.
\end{itemize}
\end{lemma}
\begin{proof} It is a standard verification on positively convex polygons.
\end{proof}

\begin{proposition}
\label{prop:effect of a coordinate change}
Consider a $\gamma$-truncated foliated space $({\mathcal A},\omega)$ and
let ${\mathcal A}\rightarrow{\mathcal A}'$ be a $(\ell+1)$-coordinate change. Denote $\delta_0=\nu(z)$, $\delta_0'=\nu(z')$.
We have $\delta_0'\geq\delta_0$ and
\begin{equation}
\label{eq:effect of a coordinate change}
\varsigma_{\delta}({\mathcal N}_{\mathcal A}(\omega))=\varsigma_{\delta}({\mathcal N}_{{\mathcal A}'}(\omega)),\quad \text{ for any }\delta\leq\delta_0.
\end{equation}
As a consequence, we have
\begin{itemize}
\item[1)]
    $\chi_{{\mathcal A}'}(\omega)=\chi_{\delta'_0}({\mathcal N}_{{\mathcal A}'}(\omega))
    \leq
    \chi_{\delta_0}({\mathcal N}_{{\mathcal A}'}(\omega))=
    \chi_{\mathcal A}(\omega)
    $.
\item[2)] $\varsigma_{{\mathcal A}'}(\omega)=
\varsigma_{\delta_0'}({{\mathcal N}_{\mathcal{A}'}}(\omega))\geq \varsigma_{\delta_0}({\mathcal N}_{\mathcal{A}'}(\omega))= \varsigma_{\delta_0}({{\mathcal N}_{\mathcal A}}(\omega))= \varsigma_{{\mathcal A}}(\omega)
    $.
\item[3)] $\lambda_{{\mathcal A}',\omega}(s)=\lambda_{{\mathcal A},\omega}(s)$, for any $s\geq \chi_{{\mathcal A}}(\omega)$.
\end{itemize}
 Moreover, we have
 $\overline{\operatorname{vec}}^{\mathcal A}_s(\omega)= \overline{\operatorname{vec}}^{{\mathcal A}'}_s(\omega)$, for any $s\geq \chi_{{\mathcal A}}(\omega)$.
\end{proposition}
\begin{proof} Let us recall that $z'=z+f$, where $\nu_{\mathcal A}(f)\geq\nu(z)=\delta_0$ and $f\in k[[\boldsymbol{x},\boldsymbol{y}]]\cap {\mathcal O}_{\mathcal A}$. In particular, we have $\nu(f)\geq \nu_{\mathcal A}(f)\geq \nu(z)$ and thus $\nu(z')\geq\nu(z)$. Let us also recall that $\nu_{\mathcal A}=\nu_{{\mathcal A}'}$.

 Note that  $\omega_0=\eta_0$ and
\[
z^t\omega_t=(z'-f)^{t-1}\left(
 z'(\eta_t+h_t\frac{dz'}{z'})- (f\eta_t+h_tdf)\right), \quad t\geq 1.
\]
If we write the level decomposition of $({\mathcal A}',\omega)$ as
$
\omega=\sum_{s\geq 0}{z'}^s\omega'_s$,   $\omega'_s=\eta'_s+h'_s{dz'}/{z'}
$, we obtain formulas for $\eta_s$, $h_s$, $\eta'_s$ and $h'_s$ as follows:
\begin{align}
\label{eq:lelvesprima1}
&
\begin{array}{ccl}
\eta'_s&=&\sum_{t\geq s}
 c^{t-1}_{s-1}(-f)^{t-s}\eta_t-\sum_{t-1\geq s}c^{t-1}_s(-f)^{t-s-1}(f\eta_t+h_tdf).
\\
h'_s&=&\sum_{t\geq s}c^{t-1}_{s-1}(-f)^{t-s}h_t.
\end{array}
\\
\label{eq:lelvesprima2}
&
\begin{array}{ccl}
\eta_s&=&\sum_{t\geq s}
 c^{t-1}_{s-1}f^{t-s}\eta'_t-\sum_{t-1\geq s}c^{t-1}_sf^{t-s-1}(f\eta'_t+h'_tdf).
\\
h_s&=&\sum_{t\geq s}c^{t-1}_{s-1}f^{t-s}h'_t.
\end{array}
\end{align}
where $c^a_b$ are the binomial coefficients.

Let us take $\delta\leq\delta_0=\nu(z)$.  In order to prove Equation \eqref{eq:effect of a coordinate change}, let us first show that $\varsigma_{\delta}({\mathcal N}_{{\mathcal A}'}(\omega))\geq \varsigma_{\delta}({\mathcal N}_{\mathcal A}(\omega))
$. Taking into account Equations \eqref{eq:lelvesprima1} and the fact that $\nu_{\mathcal A}(f)=\nu_{{\mathcal A}}(df)\geq\delta_0$, we have
\begin{align*}
s\delta+\nu_{{\mathcal A}'}(\eta'_s)
&\geq \min_{t\geq s}
\{s\delta+ (t-s)\delta_0+\min\{\nu_{\mathcal A}(\eta_t),\nu_{\mathcal A}(h_t)\}
\}
\\
&\geq \min_{t\geq s}
\{t\delta+\min\{\nu_{\mathcal A}(\eta_t),\nu_{\mathcal A}(h_t)\}\} ,
\end{align*}
and in the same way we have
$
s\delta+\nu_{{\mathcal A}'}(h'_s)
\geq  \min_{t\geq s}
\{t\delta+\nu_{\mathcal A}(h_t)\}$. Thus we get
\[
\min\{s\delta+\nu_{{\mathcal A}'}(\eta'_s), s\delta+\nu_{{\mathcal A}'}(h'_s)\}\geq
\min_{t\geq s}\{t\delta+\min\{\nu_{\mathcal A}(\eta_t),\nu_{\mathcal A}(h_t)\}\}.
\]
Now, we have
\begin{align*}
\varsigma_{\delta}({\mathcal N}_{{\mathcal A}'}(\omega))&=
\min_s\{s\delta+\min\{\nu_{{\mathcal A}'}(\eta'_s),\nu_{{\mathcal A}'}(h'_s)\}\}=\\
&=\min_s\{\min\{s\delta+\nu_{{\mathcal A}'}(\eta'_s),s\delta+\nu_{{\mathcal A}'}(h'_s)\}\}\geq\\
&\geq\min_s\{
\min_{t\geq s}\{t\delta+\min\{\nu_{\mathcal A}(\eta_t),\nu_{\mathcal A}(h_t)\}\}
\}=\\
&=
\min_{s}\{s\delta+\min\{\nu_{\mathcal A}(\eta_s),\nu_{\mathcal A}(h_s)\}\}
=\varsigma_{\delta}({\mathcal N}_{{\mathcal A}}(\omega)).
\end{align*}
By considering Equation \eqref{eq:lelvesprima2} we obtain in the same way that $\varsigma_{\delta}({\mathcal N}_{{\mathcal A}}(\omega))\geq \varsigma_{\delta}({\mathcal N}_{{\mathcal A}'}(\omega))$ and then we have $\varsigma_{\delta}({\mathcal N}_{{\mathcal A}}(\omega))= \varsigma_{\delta}({\mathcal N}_{{\mathcal A}'}(\omega))$. This proves Equation \eqref{eq:effect of a coordinate change}. Properties (1), (2) and (3) follow from Lemma \ref{lema:ordenadas supracriticas}.

Now, take $s\geq\chi_{\mathcal A}(\omega)$ and let us show that $\overline{\operatorname{vec}}^{\mathcal A}_s(\omega)= \overline{\operatorname{vec}}^{{\mathcal A}'}_s(\omega)$. Let us denote $\lambda_t=\lambda_{{\mathcal A},t}(\omega)= \lambda_{{\mathcal A}',t}(\omega)$ for any $t\geq s$. Let us take the following expressions for $t\geq s$:
\begin{itemize}
\item[1)] If $\nu_{\mathcal A}(\eta_t, h_t)=\lambda_t$, we denote
\[
(\eta_t, h_t)=\boldsymbol{x}^{I_t}(\eta_t^*, h^*_t)+(\tilde\eta_{t},\tilde h_t),
\]
where  $\nu_{\mathcal A}(\tilde \eta_t, \tilde h_t)>\lambda_t$ and $\nu_{\mathcal A}(\boldsymbol{x}^{I_t})=\lambda_t$.
\item[2)]
If $\nu_{\mathcal A}(\eta_t, h_t)>\lambda_t$, we write
$(\eta_t, h_t)= (\tilde\eta_{t},\tilde h_t)$.
\item[3)] If $\nu_{\mathcal A}f=\nu(z)$, (recall that $\nu(z)=\delta_0$) we write $f=\boldsymbol{x}^Jf^*+\tilde f$, where $\nu_{\mathcal A}(\tilde f)>\nu(z)$ and $\nu(\boldsymbol{x}^J)=\delta_0$. (Let us note that, in this situation, the only possibility is that the ramification index of $\mathcal A$ is one and the rational contact function is $\Phi=z/\boldsymbol{x^p}$, with $\boldsymbol{p}=J$).
\item[4)] If $\nu_{\mathcal A}f>\delta_0$, we write $f=\tilde f$.
\end{itemize}
Let us note that for $t>s$ we have that $\lambda_s-\lambda_t<(t-s)\delta_0$. By Equations \eqref{eq:lelvesprima1}, we conclude that
\[
\eta'_s= \boldsymbol{x}^{I_s}\eta_s^*+\tilde\eta'_{s},\quad  h'_s=\boldsymbol{x}^{I_s}h_s^*+\tilde h'_{s}, \quad \nu_{{\mathcal A}'}(\tilde\eta'_{s}),
\nu_{{\mathcal A}'}(\tilde h'_{s})>\lambda_s.
\]
This shows that $\bar\omega^{\mathcal A}_s=\bar\omega^{{\mathcal A}'}_s$, for any $s\geq \chi$. We also get that $\overline{\operatorname{vec}}^{\mathcal A}_s(\omega)=\overline{\operatorname{vec}}^{{\mathcal A}'}_s(\omega)$, since
$\overline{\Omega}^1_{\mathcal A}=\overline{\Omega}^1_{{\mathcal A}'}$ and the $k$-isomorphisms
$\overline{\Omega}^1_{\mathcal A}\rightarrow k^{r+1}$ and
$\overline{\Omega}^1_{{\mathcal A}'}\rightarrow k^{r+1}$ coincide.
\end{proof}

Next result proves Proposition \ref{prop:stability critical height} in the case of a normalized coordinate change:
\begin{corollary}
\label{cor:stabilizadcambiodecoordenadas}
Assume that $({\mathcal A},\omega)$ is strictly $\gamma$-prepared with $\varsigma_{\mathcal A}(\omega)\leq \gamma$ and that $({\mathcal A},\omega)\rightarrow ({\mathcal A}^\star,\omega)$ is a normalized coordinate change. If  $\varsigma_{{\mathcal A}^\star}(\omega)\leq\gamma$ we have
\[
\varsigma_{\mathcal A}(\omega)\leq \varsigma_{{\mathcal A}^\star}(\omega), \quad
\chi_{\mathcal A}(\omega)\geq \chi_{{\mathcal A}^\star}(\omega).
\]
\end{corollary}
\begin{proof}
Let ${\mathcal A}\rightarrow {\mathcal A}'$ be the $(\ell+1)$-coordinate change that we follow by a strict $\gamma$-preparation $({\mathcal A}',\omega)\rightarrow ({\mathcal A}^\star,\omega)$.  In view of Proposition \ref{prop:effect of a coordinate change} and by Remark  \ref{rk:dominantreducedpart}, the critical vertex
\[
(\varsigma_{\mathcal A}(\omega)-\nu(z)\chi_{\mathcal A}(\omega),\chi_{\mathcal A}(\omega))
\]
of $({\mathcal A},\omega)$ is also a dominant vertex of ${\mathcal N}_{{\mathcal A}'}(\omega)$. Since $\nu(z')\geq\nu(z)$, this vertex is preserved under the strict $\gamma$-preparation, and it is higher or equal than the new critical height.
\end{proof}

\subsection{The Critical Height under a Puiseux's package}
\label{The Critical Height under a Puiseuxs package} Let us take a strictly $\gamma$-prepared $({\mathcal A},\omega)$ and consider a normalized Puiseux's package
\[
({\mathcal A},\omega)\rightarrow ({\mathcal A}^\star,\omega).
\]
We recall that it is composed of an $(\ell+1)$-Puiseux's package ${\mathcal A}\rightarrow {\mathcal A'}$ followed by a strict $\gamma$-preparation $({\mathcal A}',\omega)\rightarrow
({\mathcal A}^\star,\omega)$.

 Let us denote by $\Phi=z^d/\boldsymbol{x}^{\boldsymbol p}$ the contact rational function for the Puiseux's package, where the number ``$d$'' is the ramification index of the Puiseux's package. We recall that $\nu(\Phi)=0$ and that $\Phi=z'+\lambda$, where $0\ne \lambda\in k$ is uniquely determined.

  Let us recall the level decomposition $\omega=\sum_sz^s\omega_s$ given in Equation \eqref{eq:omegalevels}.

 \begin{proposition}
 \label{pro:valorcriticopuiseux} We have
  $\nu_{{\mathcal A}^\star}(\omega)\geq \nu_{{\mathcal A}'}(\omega)\geq \varsigma_{\mathcal A}(\omega)$.
 \end{proposition}
 \begin{proof} The stability results under a $\gamma$-strict preparation show that  $\nu_{{\mathcal A}^\star}(\omega)\geq \nu_{{\mathcal A}'}(\omega)$. To see that $\nu_{{\mathcal A}'}(\omega)\geq \varsigma_{\mathcal A}(\omega)$ is enough to show that
  \[
 \nu_{{\mathcal A}'}(z^s\omega_s)\geq s\nu(z)+\min\{\nu_{\mathcal A}(\eta_s),\nu_{\mathcal A}(h_s)\}.
 \]
Moreover, $\nu(z)=\nu_{{\mathcal A}'}(z)$, since
$z=U'\boldsymbol{x'}^I$ where $U'$ is a unit in ${\mathcal O}_{{\mathcal A}'}$. Then, we have only to see that
 $
 \nu_{{\mathcal A}'}(\omega_s)\geq \min\{\nu_{\mathcal A}(\eta_s),\nu_{\mathcal A}(h_s)\}
 $.
 Let us note that
 \[
 \omega_s=\eta_s+h_s\frac{dz}{z}=
 \eta_s+h_s\left(\frac{d\boldsymbol{x'}^I}{\boldsymbol{x'}^I}+\frac{dU'}{U'}\right)\in \Omega^1_{{\mathcal A}'}.
 \]
 Then $\nu_{{\mathcal A}'}(\omega_s)\geq \min \{\nu_{{\mathcal A}'}(\eta_s),\nu_{{\mathcal A}'}(h_s)\}$. By Proposition \ref{prop:stabilityofcriticalvalue} we have
 $\nu_{{\mathcal A}'}(\eta_s)\geq \nu_{\mathcal A}(\eta_s)$ and
 $\nu_{{\mathcal A}'}(h_s)\geq \nu_{\mathcal A}(h_s)$.  This ends the proof.
 \end{proof}
\begin{remark} We obtain
Proposition \ref{prop:recessive case} as a consequence of Proposition \ref{pro:valorcriticopuiseux}, just by  noting that $
 \nu_{{\mathcal A}'}(\omega')\geq \varsigma_{{\mathcal A}}(\omega)>\gamma
 $.
 \end{remark}
Let us consider  a number $\rho$ lower or equal than the main abscissa $\nu_{\mathcal A}(\omega)$. We define the {\em $\rho$-dominant main height} $\hbar^\rho_{\mathcal A}(\omega)$ by
\begin{equation*}
\hbar^\rho_{{\mathcal A}}(\omega)=\min\{s;\;
\nu_{\mathcal A}(z\omega_s)=\rho \text{ and } \bar\omega^{\mathcal A}_s\ne 0 \}.
\end{equation*}
In the rest of this subsection, we put $\varsigma=\varsigma_{\mathcal A}(\omega)$ and we assume that $\varsigma\leq \gamma$.

 \begin{lemma}
     \label{lema: control by the main height} We have either $\varsigma_{{\mathcal A}^\star}(\omega)>\gamma$ or  $\chi_{{\mathcal A}^\star}(\omega)\leq \hbar^\varsigma_{{\mathcal A}'}(\omega)$.
     \end{lemma}
     \begin{proof} Let us note that $\varsigma\leq \nu_{{\mathcal A}'}(\omega)$ in view of Proposition \ref{pro:valorcriticopuiseux}. Then $\hbar^\varsigma_{{\mathcal A}'}(\omega)$ makes sense. Assume that $\varsigma_{{\mathcal A}^\star}(\omega)\leq\gamma$ and put $\hbar'=\hbar^\varsigma_{{\mathcal A}'}(\omega)$, that we suppose $\hbar'<\infty$. In particular, we have that $\varsigma=\nu_{{\mathcal A}'}(\omega)$ is the main abscissa of ${\mathcal N}_{\mathcal A}(\omega)$.
     We have two possible situations:
     \begin{itemize}
     \item[a)] $\varsigma>\gamma-\hbar'\nu(z')$. The main abscissa of ${\mathcal N}_{{\mathcal A}'}(\omega)$ is $\varsigma$ and $(\varsigma, \hbar')\in {\mathcal N}_{{\mathcal A}'}(\omega)$. In the process of $\gamma$-dominant preparation, see Section  \ref{Dominant Preparation} and Proposition \ref{prop:dominantpreparation}, we have only to consider levels strictly under $\hbar$. Since we assume that $\varsigma_{{\mathcal A}^\star}(\omega)\leq\gamma$, the critical vertex of $({\mathcal A}^\star,\omega)$ corresponds to one of that levels. Then
         $\chi_{{\mathcal A}^\star}(\omega)< \hbar'$.
     \item[b)] $\varsigma\leq\gamma-\hbar'\nu(z')$. In this case, the point $(\varsigma,\hbar')$ of the Newton Puiseux's polygon is  persistent under the process of strict preparation ${\mathcal A}'\rightarrow{\mathcal A}^\star$ as well as the main abscissa $\varsigma$.   Then $\chi_{{\mathcal A}^\star}(\omega)\leq \hbar'$.
     \end{itemize}
     This ends the proof.
     \end{proof}

\begin{remark} In the above proof, the only possibility to have  $\chi_{{\mathcal A}^\star}(\omega)= \hbar'$ is that $(\varsigma,\hbar')$ is both the main and the critical vertex of ${\mathcal N}_{{\mathcal A}^\star}(\omega)$.
\end{remark}

\begin{lemma}
\label{lema:stabilityofcriticalheight}
 We have $\hbar_{{\mathcal A}'}^\varsigma(\omega)\leq\chi_{{\mathcal A}}(\omega)$.
 \end{lemma}
\begin{remark}
\label{rk:prostabilityparapuiseux}
We obtain Proposition \ref{prop:stability critical height}
 for the case of a normalized Puiseux's package  as a corollary of  Lemmas \ref{lema:stabilityofcriticalheight} and
    \ref{lema: control by the main height}.
 We also obtain Proposition \ref{prop:critical heightcero},
 as follows: the fact that
 $\hbar^\varsigma_{{\mathcal A}'}(\omega)=0$ implies that $\varsigma$ is the main abscissa and $(\varsigma,0)$ is the main vertex that corresponds to a  $\varsigma$-final level. In view of Proposition \ref{pro:onevertex}
  we have that $({\mathcal A}',\omega)$ is $\gamma$-final and thus $({\mathcal A}^\star,\omega)$ also is $\gamma$-final.
\end{remark}
Let us start the proof of Lemma \ref{lema:stabilityofcriticalheight}.

Denote
 $
 \delta=\nu(z)$ and $\hbar'=\hbar^\varsigma_{{\mathcal A}'}(\omega)$.
Let $ s_0$ and $\chi$,  with $s_0\leq\chi$ be the ordinates of the vertices of the dominant critical segment ${\mathcal C}_{\mathcal A}(\omega)$. We recall that $\chi=\chi_{\mathcal A}(\omega)$
is the critical height of $({\mathcal A},\omega)$. The case $\chi=0$ is straighforward and we leave it to the reader. Thus we assume $\chi\geq 1$.

The proof follows from  Lemmas \ref{lema:reduccionaomegabar},
\ref{lema:reduccionaomegabar2},
\ref{lema:alphabar},
\ref{lema:calculodehbar}
and Proposition
 \ref{prop:estabilidadho}
 below.

\begin{lemma}
\label{lema:reduccionaomegabar}
For any decomposition $\omega=\omega^*+\tilde\omega$ where $\varsigma_{\mathcal A}(\tilde\omega)>\varsigma$, we have
\[
\hbar'=\hbar^\varsigma_{{\mathcal A}'}(\omega)= \hbar^\varsigma_{{\mathcal A}'}({\omega^*}).
\]
\end{lemma}
\begin{proof}It is enough to remark that
 $\nu_{{\mathcal A}'}(\tilde\omega)>\varsigma$ in view of
Proposition \ref{pro:valorcriticopuiseux}.
\end{proof}

Since $({\mathcal A},\omega)$ is strictly $\gamma$-prepared, we can decompose $\omega=\omega^*+\tilde\omega$ where $\varsigma_{\mathcal A}(\tilde \omega)>\varsigma$ and $\omega^*$ may be expressed as follows
 \begin{equation*}
 \omega^*=\sum_{s=s_0}^{\chi}z^s \boldsymbol{x}^{I_s}\omega^*_s, \quad \omega^*_s=\eta^*_s+h^*_s{dz}/{z},
 \quad
 \eta^*_s=\sum_{i=1}^r f^*_{s,i}{dx_i}/{x_i}+ \sum_{j=1}^{\ell}g^*_{s,j}dy_j ,
 \end{equation*}
 with
     $
     f^*_{s,i},g^*_{s,j}, h^*_s\in k[[\boldsymbol{y_\ell}]]
     $
 and the following properties hold:
 \begin{itemize}
  \item[1)] $\omega^*_{s_0}\ne 0\ne\omega^*_{\chi}$.
  \item[2)] If $\omega^*_s\ne 0$, then $\nu(\boldsymbol{x}^{I_s})=\varsigma-s\delta$.
 \item[3)] Each $\omega^*_s$ with $\omega^*_s\ne 0$ is $0$-final with respect to ${\mathcal A}$ and more precisely we have:
 \begin{itemize}
 \item[a)] If $\eta^*_s\ne 0$ there is a unit among the coefficients $f^*_{i,s}$, for $i=1,2,\ldots,r$.
     \item[b)] If $h^*_s\ne 0$ then $h^*_s$ is a unit in $k[[\boldsymbol{y}_\ell]]$.
 \end{itemize}
 \end{itemize}
Let us denote $\mu_s=h^*_s(0)$, the class of $h^*_s$ modulo the maximal ideal. In the same way, we denote $\lambda_{s,i}=f^*_{s,i}(0)$.
Now, we decompose $\omega^*_s=\bar\omega^*_s+\tilde\omega^*_s$, where
\begin{equation}
\label{eq:omegaese}
\bar\omega^*_s=
\frac{
d\boldsymbol{x}^{\boldsymbol{\lambda}_s}
}{
\boldsymbol{x}^{\boldsymbol{\lambda}_s}
}
+\mu_s\frac{d z}{z}=\bar\omega^{\mathcal A}_s,\quad
\boldsymbol{\lambda}_s=(\lambda_{s,1},\lambda_{s,2},\ldots,\lambda_{s,r}).
\end{equation}
We have that $\omega^*_s\ne0$ if and only if $\bar\omega^*_s\ne0$.
Note that $\tilde\omega^*_s$ is not $0$-final dominant and we can write it as
\begin{equation}
\label{eq:tildeomega}
\tilde\omega^*_s= \sum_{j=1}^{\ell}
 y_j\left(\sum_{i=1}^rf^*_{s,i,j}\frac{dx_i}{x_i}+h^*_{s,j}\frac{dz}{z}\right)+ g^*_{s,j}dy_j=\sum_{j=1}^{\ell}\left( y_j\omega^*_{s,j}+g^*_{s,j}dy_j\right).
\end{equation}
Let us write $\bar\omega^*=\sum_{s=s_0}^{\chi}z^s \boldsymbol{x}^{I_s}\bar\omega^*_s$ and
$\tilde\omega^*=\sum_{s=s_0}^{\chi}z^s \boldsymbol{x}^{I_s}\tilde\omega^*_s$. Hence $\omega^*=\bar\omega^*+\tilde\omega^*$.
\begin{definition}
\label{def:reducedcriticalpart}
We call $\bar\omega^*$ the {\em reduced critical part of $\omega$ with respect to  $\mathcal A$}.
\end{definition}

\begin{lemma}
\label{lema:reduccionaomegabar2} We have
$
 \hbar^\varsigma_{{\mathcal A}'}(\omega^*)= \hbar^\varsigma_{{\mathcal A}'}(\bar\omega^*)
$.
\end{lemma}
\begin{proof}
By Equation \eqref{eq:tildeomega}   and since $\nu_{{\mathcal A}'}(\tilde\omega^*)\geq \varsigma$, we deduce that $\hbar^\varsigma_{{\mathcal A}'}(\tilde\omega^*)=\infty$. Noting that $\omega^*=\bar\omega^*+\tilde\omega^*$, we conclude that
$
\hbar^\varsigma_{{\mathcal A}'}(\omega^*)=
\hbar^\varsigma_{{\mathcal A}'}(\bar\omega^*)
$.
\end{proof}
By Lemmas  \ref{lema:reduccionaomegabar} and \ref{lema:reduccionaomegabar2} we have
$
\hbar^\varsigma_{{\mathcal A}'}(\omega)= \hbar^\varsigma_{{\mathcal A}'}({\bar\omega^*})
$.
Now, we are going to compute $\hbar^\varsigma_{{\mathcal A}'}({\bar\omega^*})$.
Recall that the rational contact function is given by $\Phi=z^d/\boldsymbol{x}^{\boldsymbol{p}}$.

\begin{lemma}
\label{lema:alphabar}
 We have
$
\bar\omega^*= z^{s_0}\boldsymbol{x}^{I_{s_0}}\bar\alpha$,
$\nu(z^{s_0}\boldsymbol{x}^{I_{s_0}})=\varsigma
$,
where
\begin{equation*}
\bar\alpha=\sum_{m=0}^{m_1}\Phi^{m}\bar\omega^*_{s(m)}\in \Omega^1_{{\mathcal A}'},
\end{equation*}
with $s(m)=md+s_0$.
\end{lemma}
\begin{proof}
For any
index $s$ such that $\bar\omega^*_s\ne 0$, we have $\nu(z^s\boldsymbol{x}^{I_s})=\nu(z^{s_0}\boldsymbol{x}^{I_{s_0}})=\varsigma$.
 This implies that $\nu(z^{s-s_0}\boldsymbol{x}^{I_{s}-I_{s_0}})=0$.
Then, there is  $m(s)\in {\mathbb Z}_{\geq 0}$ such that
\[
z^{s-s_0}\boldsymbol{x}^{I_{s}-I_{s_0}}=\Phi^{m(s)}=(z^d\boldsymbol{x}^{\boldsymbol{-p}})^{m(s)}.
\]
That is
$
s-s_0=m(s)d$ and  $I_s-I_{s_0}=-m(s)\boldsymbol{p}
$.
In particular $m_1d=\chi-s_0$ and
 $I_{\chi}-I_{s_0}=-m_1\boldsymbol{p}$, where $m_1=m(\chi)$.
 Then, we can write $\bar\omega^*$ as
\begin{equation*}
\bar\omega^*=z^{s_0}\boldsymbol{x}^{I_{s_0}}\sum_{m=0}^{m_1}\Phi^m
\bar\omega^*_{s(m)}=
z^{s_0}\boldsymbol{x}^{I_{s_0}}\sum_{m=0}^{m_1}\Phi^{m}
\bar\omega^*_{s(m)}
,\quad s(m)=md+s_0.
\end{equation*}
Denote
$
\bar\alpha=\sum_{m=0}^{m_1}\Phi^{m}\bar\omega^*_{s(m)}\in \Omega^1_{{\mathcal A}'}.
$
We have
$
\bar\omega^*= z^{s_0}\boldsymbol{x}^{I_{s_0}}\bar\alpha
$.
\end{proof}
\begin{remark} We have that $\hbar_{{\mathcal A}'}^\varsigma(\bar\omega^*)
=\hbar_{{\mathcal A}'}^0(\bar\alpha)$, since in view of
Equations \eqref{eq:puiseux2}, there is $q\in {\mathbb Z}$ and $I'\in {\mathbb Z}_{\geq 0}^r$ with $\nu({\boldsymbol{x'}}^{I'})=\varsigma$ such that
$\bar\omega^*=(z'+\lambda)^{q}{\boldsymbol{x'}}^{I'}\bar\alpha$.
Thus, in order to compute $\hbar^\varsigma(\bar\omega^*)$ it is enough to compute $\hbar^0_{{\mathcal A}'}(\bar\alpha)$.
\end{remark}
Let us introduce the $k$-vector subspace $V_{{\mathcal A}'}$ of $\Omega^1_{{\mathcal A}'}$ whose elements are the $ 1 $-forms $\beta\in \Omega^1_{{\mathcal A}'}$ written as
\begin{equation}
\beta=\frac{d\boldsymbol{x'}^{\boldsymbol{\tau'}}}{\boldsymbol{x'}^{\boldsymbol{\tau'}}}+
\xi'\frac{d\Phi}{\Phi},\quad \boldsymbol{\tau'}\in k^{r},\;\xi'\in k.
\end{equation}
Let us recall that $\Phi=z'-\lambda$ and then $k[\Phi]\subset {\mathcal O}_{{\mathcal A}'}$ is isomorphic to a polynomial ring in one variable over $k$. Consider the $k[\Phi]$-submodule $V_{{\mathcal A}'}[\Phi]\subset\Omega^1_{{\mathcal A}'}$ given by the finite sums
$
\theta=\sum_{i\geq 0}\Phi^i\beta_i
$, where $\beta_i\in V_{{\mathcal A}'}$. Note that if $\theta'=\sum_{i\geq 0}\Phi^i\beta'_i$, we have that $\theta=\theta'$ if and only if $\beta_i=\beta'_i$ for all $i\geq 0$. Let us remark that any element $\theta\in V_{{\mathcal A}'}[\Phi]$ may be written in a unique way as a finite sum
\[
\theta=\sum_{i\geq 0}(\Phi-\lambda)^i\tilde\beta_i,\quad \tilde\beta_i\in V_{{\mathcal A}'}.
\]
Next lemma provides a way to compute $\hbar^0_{{\mathcal A}'}(\theta)$ for any $\theta\in V_{{\mathcal A}'}[\Phi]$.

\begin{lemma}
\label{lema:calculodehbar}
 Let us consider $\theta\in V_{{\mathcal A}'}[\Phi]$ written as
\[
\theta=(\Phi-\lambda)^a\sum_{i=0}^b(\Phi-\lambda)^i\beta_i, \quad
0\ne\beta_0=\frac{d\boldsymbol{x'}^{\boldsymbol{\tau'}}}{\boldsymbol{x'}^{\boldsymbol{\tau'}}}+
\xi'\frac{d\Phi}{\Phi}.
\]
We have
\begin{itemize}
\item If $\boldsymbol{\tau'}\ne\boldsymbol{0}$, then $\hbar^0_{{\mathcal A}'}(\theta)=a$.
\item If $\boldsymbol{\tau'}=\boldsymbol{0}$, then $\hbar^0_{{\mathcal A}'}(\theta)=a+1$.
\end{itemize}
\end{lemma}
\begin{proof} It is enough to recall that $\Phi-\lambda=z'$.
\end{proof}
\begin{lemma} We have
$\bar\alpha\in V_{{\mathcal A}'}[\Phi]$.
\end{lemma}
\begin{proof} It is enough to show that $\bar\omega^*_s\in V_{{\mathcal A}'}$, for any $s_0\leq s\leq\chi$. Write $\bar\omega^*_s$ as in Equation \eqref{eq:omegaese}. Recalling the matrix $C$ in Equation  \eqref{eq:puiseux4}, we have
\begin{equation*}
\bar\omega^*_s= \frac{
d\boldsymbol{x}^{\boldsymbol{\lambda}_s}
}{
\boldsymbol{x}^{\boldsymbol{\lambda}_s}
}
+\mu_s\frac{d z}{z}=
\frac{
d\boldsymbol{x'}^{\boldsymbol{\lambda}'_s}
}{
\boldsymbol{x'}^{\boldsymbol{\lambda}'_s}
}
+\mu'_s\frac{d \Phi}{\Phi}\in V_{{\mathcal A}'},
\end{equation*}
where $(\boldsymbol{\lambda'}_s,\mu'_s)=(\boldsymbol{\lambda}_s,\mu_s)C$.
\end{proof}

\begin{proposition}
 \label{prop:estabilidadho}
 We have $\hbar^0_{{\mathcal A}'}(\bar\alpha)\leq \chi$. Moreover, if $\hbar^0_{{\mathcal A}'}(\bar\alpha)=\chi$ one of the following two conditions ``r1'' or ``r2'' holds:
 \begin{itemize}
   \item[r1:] $\chi=1$, $d\geq 2$, $s_0=1$ and
   $
\bar\alpha= \xi'd\Phi/\Phi
$, where $\xi'\in k$, $\xi'\ne 0$.
   \item[r2:] $d=1$ and $\bar\alpha$ has one of the forms:
   \begin{itemize}
   \item[r2-i:] $s_0=0$,
   $\bar\alpha=(\Phi-\lambda)^\chi
   (d\boldsymbol{x}^{\boldsymbol{\tau}}/\boldsymbol{x}^{\boldsymbol{\tau}})$.
   \item[r2-ii:] $s_0=1$, $\bar\alpha= (\Phi-\lambda)^{\chi-1}
\xi'{d\Phi}/\Phi$, $\xi'\ne 0$.
 \item[r2-iii:] $s_0=0$,
 $\bar\alpha= (\Phi-\lambda)^{\chi}
 \left(
 (d\boldsymbol{x}^{\boldsymbol{\upsilon}}/\boldsymbol{x}^{\boldsymbol{\upsilon}})+
 (\xi'/\lambda)(d(\Phi-\lambda)/(\Phi-\lambda))
\right)$,
 $\xi'\ne 0$.
   \end{itemize}
 \end{itemize}
 Here ``$d$'' stands for the ramification index of the Puiseux's package.
 \end{proposition}
 \begin{proof} Denote $\hbar'=\hbar^0_{{\mathcal A}'}(\bar\alpha)$ and let us write
 \begin{equation}\label{eq:baralpha}
 \bar\alpha=\sum_{m=0}^{m_1}\Phi^{m}\bar\omega^*_{s(m)}=
 (\Phi-\lambda)^a\sum_{i=0}^{m_1-a}(\Phi-\lambda)^i\beta_i,
 \end{equation}
 where $\beta_i\in V_{{\mathcal A}'}$ and
 $
 0\ne\beta_0=
(d\boldsymbol{x'}^{\boldsymbol{\tau'}}/\boldsymbol{x'}^{\boldsymbol{\tau'}})+
\xi'd\Phi/\Phi
 $.
 Note that $a\leq m_1$. In view of Lemma \ref{lema:calculodehbar}, there are two possibilities: $\hbar'=a$ and hence $\boldsymbol{\tau'}\ne\boldsymbol{0}$ or $\hbar'=a+1$ and hence $\boldsymbol{\tau'}=\boldsymbol{0}$.

 A) {\em  Case:} $\hbar'=a$. We have $\boldsymbol{\tau'}\ne\boldsymbol{0}$.
Recall that $dm_1=\chi-s_0$ and hence $m_1\leq \chi$. Since $a\leq m_1$, we obtain that $a=\hbar'\leq m_1\leq\chi$. In this case we have that
\[
\hbar'=\chi \Leftrightarrow a=m_1\,\text { and }\, m_1=\chi \Leftrightarrow a=m_1,\, d=1\,\text{ and }\, s_0=0.
\]
This is equivalent to $ \bar\alpha=(\Phi-\lambda)^{\chi}\beta_0 $ and $\boldsymbol{\tau'}\ne\boldsymbol{0}$. Moreover, since $s_0=0$ we have that $s(0)=0$ and thus
\[
\bar\omega^*_0=\bar\omega^*_{s(0)}=(-\lambda)^\chi\beta_0.
\]
Recall that $\mu_0=0$ since $\bar\omega\in \Omega_{\mathcal A}$. Then, we have
\[
\frac{d\boldsymbol{x}^{{\boldsymbol{\lambda}_0}}}{\boldsymbol{x}^{{\boldsymbol{\lambda}_0}}}=
\frac{d\boldsymbol{x}^{{\boldsymbol{\lambda}'_0}}}{\boldsymbol{x}^{{\boldsymbol{\lambda}'_0}}}
+\mu'_0\frac{d\Phi}{\Phi}
=\bar\omega^*_0=(-\lambda)^\chi\beta_0.
\]
Since $d=1$, by Equation \eqref{eq:ecuaciondigulauno},  we have
\begin{equation}
\label{eq:ecuaciondigulauno3}
C=
\left(
\begin{array}{ccc}
{C_0}&\vline&{\boldsymbol{0}}\\
\hline
\boldsymbol{\tilde p}&\vline& 1\\
\end{array}
\right),\quad \boldsymbol{\tilde p}=\boldsymbol{p}C_0, \quad
,
 \quad (\boldsymbol{\lambda}'_0,\mu'_0)=(\boldsymbol{\lambda}_0,\mu_0)C.
\end{equation}
and then $\mu'_0=\mu_0=0$. Thus, in this case A) we have that $\hbar'=\chi$ if and only if $d=1$, $s_0=0$ and
\[
\bar\alpha=(\Phi-\lambda)^\chi
   (d\boldsymbol{x'}^{\boldsymbol{\tau'}}/\boldsymbol{x'}^{\boldsymbol{\tau'}}).
\]
Noting that $d\boldsymbol{x'}^{\boldsymbol{\tau'}}/\boldsymbol{x'}^{\boldsymbol{\tau'}}=
d\boldsymbol{x}^{\boldsymbol{\tau}}/\boldsymbol{x}^{\boldsymbol{\tau}}
$ for $\boldsymbol{\tau}=\boldsymbol{\tau}'C_0^{-1}$ we have
$
\bar\alpha=(\Phi-\lambda)^\chi
   (d\boldsymbol{x}^{\boldsymbol{\tau}}/\boldsymbol{x}^{\boldsymbol{\tau}})
$,
 hence condition {\rm r2-i} is satisfied.
\vspace{5pt}

B) {\em Case:} $\hbar'=a+1$. We have $\boldsymbol{\tau'}=\boldsymbol{0}$.
Let us show first that $\hbar'\leq\chi$. If $m_1<\chi$ or $a<m_1=\chi$
we are done, since then $a+1\leq\chi$. The only remaining case is $a=m_1=\chi$, let us show that this situation leads to a contradiction and hence it does not occur. Since $m_1=\chi$, we have $s_0=0$ and $d=1$. We have that $s(m)=dm+s_0=m$, thus
\[
\bar\alpha=\sum_{m=0}^{\chi}\Phi^{m}\bar\omega^*_{s(m)}= (\Phi-\lambda)^{\chi}\beta_0,
\quad (-\lambda)^\chi\beta_0=\bar\omega^*_0 \ .
\]
Write
\[
(-\lambda)^\chi\beta_0=\bar\omega^*_0=
 \frac{d\boldsymbol{x}^{\boldsymbol{\lambda}_0}}{\boldsymbol{x}^{\boldsymbol{\lambda}_0}}+
\mu_0\frac{dz}{z}= \frac{d\boldsymbol{x'}^{\boldsymbol{\lambda}'_0}}{\boldsymbol{x'}^{\boldsymbol{\lambda}'_0}}+
\mu'_0\frac{d\Phi}{\Phi},\quad (\boldsymbol{\lambda}'_0,\mu'_0)= (\boldsymbol{\lambda}_0,\mu_0)C.
\]
We know that $\boldsymbol{\lambda}'_0=(-\lambda)^\chi\boldsymbol{\tau'}=0$ and $(-\lambda)^\chi\xi'=\mu'_0\ne 0$. On the other hand,
note that $\omega\in \Omega^1_{\mathcal A}$ and hence $h_0=0$. This implies that $\mu_0=0$.  Now, since $d=1$, we have $C$ as in  by Equation \eqref{eq:ecuaciondigulauno3},
and then $\mu'_0=\mu_0=0$. This is a contradiction.

Let us now characterize the situations when $a+1=\chi$.  We have two possibilities, either $a=m_1=\chi-1$ or $a+1=m_1=\chi$.

$\bullet$
Assume first that $a=m_1=\chi-1$. Since $m_1d=\chi-s_0$ we have $ \chi (d-1) = d-s_0 $. If $d=1$ we have $s_0=1$ hence
\begin{equation*}
\bar\alpha=(\Phi-\lambda)^{\chi-1}\xi'd\Phi/\Phi.
\end{equation*}
We obtain r2-ii. Let us consider now the case $d\geq 2$. We have $ \chi = (d-s_0)/(d-1)  $, hence $ s_0\leq 1 $. If $ s_0 = 1 $ we obtain $ \chi = 1 $ and $ m_1=0 $. We have
\[
\bar\alpha=\bar\omega^*_{1}=\beta_0=\xi'd\Phi/\Phi \ ,
\]
and this corresponds with property r1. In the case $ s_0 = 0 $ we must have $ \chi = d = 2 $ and $ m_1=1 $. From Equation \eqref{eq:baralpha} we obtain
\[
\bar{\alpha}=\phi \omega^*_{2} + \omega^*_{0} = (\phi-\lambda) \beta_0 \ , \quad \beta_0 = \frac{d {\boldsymbol{x}'}^{\boldsymbol{\tau}'}}{{\boldsymbol{x}'}^{\boldsymbol{\tau}'}} + \xi \frac{d \phi}{\phi} \ .
\]
Write
\[
\omega^*_{0} =  \frac{d\boldsymbol{x}^{\boldsymbol{\lambda}_0}}{\boldsymbol{x}^{\boldsymbol{\lambda}_0}}+
\mu_0\frac{dz}{z}= \frac{d\boldsymbol{x'}^{\boldsymbol{\lambda}'_0}}{\boldsymbol{x'}^{\boldsymbol{\lambda}'_0}}+
\mu'_0\frac{d\Phi}{\Phi},\quad (\boldsymbol{\lambda}'_0,\mu'_0)= (\boldsymbol{\lambda}_0,\mu_0)C \ .
\]
We know that $ \mu_{0}=0 $ so $ \boldsymbol{\lambda}'_0 =  \boldsymbol{\lambda}_0 C_0 $. On the other hand, we have that $ \omega^*_{0} = -\lambda \beta_0 $ and hence
\[
 \boldsymbol{\lambda}_0 C_0 = -\lambda \boldsymbol{\tau}' \ .
\]
Since $ \beta_0 \neq 0 $ we have  $ \omega^*_{0}\neq 0 $. Then $ \mu_0 = 0 $ implies $ \boldsymbol{\lambda}_0 \neq \boldsymbol{0} $. Since $ C_0 $ is invertible, we have that $ \boldsymbol{\tau}' \neq \boldsymbol{0} $. In this situation Lemma \ref{lema:calculodehbar} assures that $ \hbar' = a = \chi-1 < \chi$.

$\bullet$ Assume now that $a+1=m_1=\chi$. Since $m_1=\chi$, we have that $s_0=0$ and $d=1$. We can write
\[
\bar\alpha=\sum_{m=0}^\chi\Phi^m\bar\omega^*_{s(m)}=
(\Phi-\lambda)^{\chi-1}\left((\phi-\lambda)\beta_1+\beta_0\right) \ .
\]
By Lemma \ref{lema:calculodehbar} we know that $\beta_0=\xi'd\Phi/\Phi$. On the other hand, we have $\bar\omega^*_0=\bar\omega^*_{s(0)}=(-\lambda)^{\chi-1}(\beta_0-\lambda\beta_1)$. Let us write
\[
\beta_1=\frac{
{d\boldsymbol{x}'}^{{\boldsymbol{\upsilon}}'}
}{
{\boldsymbol{x}'}^{{\boldsymbol{\upsilon}}'}
}
+
\xi_1\frac{d\Phi}{\Phi},\quad \beta_0-\lambda\beta_1=
\frac{
{d\boldsymbol{x}'}^{-\lambda{\boldsymbol{\upsilon}'}}
}{
{\boldsymbol{x}'}^{-\lambda{\boldsymbol{\upsilon}'}}
}
+
(\xi'-\lambda\xi_1)\frac{d\Phi}{\Phi}.
\]
Recalling that $d=1$ and $\mu_0=0$, we get that $\xi'-\lambda\xi_1=\mu'_0=0$, see Equation \eqref{eq:ecuaciondigulauno3}. Then, we have that
\[
(\Phi-\lambda)\beta_1+\beta_0=
(\Phi-\lambda)
\frac{
d\boldsymbol{x}^{\boldsymbol{\upsilon}'}
}{
\boldsymbol{x}^{\boldsymbol{\upsilon}'}
}
+
\frac{\xi'}{\lambda}\Phi\frac{d\Phi}{\Phi}=
(\Phi-\lambda)\left(\frac{
d\boldsymbol{x}^{\boldsymbol{\upsilon}'}
}{
\boldsymbol{x}^{\boldsymbol{\upsilon}'}
}
+\frac{\xi'}{\lambda}\frac{d(\Phi-\lambda)}{\Phi-\lambda}
\right).
\]
Let us denote $\boldsymbol{\upsilon}=\boldsymbol{\upsilon}'C_0^{-1}$, see Equation \eqref{eq:ecuaciondigulauno3}. Then, we have that
\[\bar\alpha= (\Phi-\lambda)^{\chi}
 \left(
 (d\boldsymbol{x}^{\boldsymbol{\upsilon}}/\boldsymbol{x}^{\boldsymbol{\upsilon}})+
 (\xi'/\lambda)(d(\Phi-\lambda)/(\Phi-\lambda))
\right)
\]
and we obtain condition {\rm r2-iii}.
\end{proof}

The proof of Lemma \ref{lema:stabilityofcriticalheight} is ended.

Let us complete the proof of Proposition \ref{prop:stability critical height}. By Corollary \ref{cor:stabilizadcambiodecoordenadas} we know that Proposition \ref{prop:stability critical height} is true for  a normalized coordinate change. By Remark \ref{rk:prostabilityparapuiseux} we know that it is true
for a normalized Puiseux's package. We deduce the result for a general normalized transformation $({\mathcal A},\omega)\rightarrow ({\mathcal B},\omega)$, by noting that we always have that $\varsigma_{\mathcal B}(\omega)\geq \varsigma_{\mathcal A}(\omega)$.

\subsection{Resonances} Conditions ``r1'' and ``r2'' in Proposition \ref{prop:estabilidadho} are the properties that may produce a stabilization of the critical height. Here we present them in terms of the reduced critical part $\bar\omega^*$ and we describe their behaviour under normalized transformations.

\begin{definition}
\label{def:resonances}
Let $({\mathcal A},\omega)$ be
strictly $\gamma$-prepared with $\varsigma^\gamma_{\mathcal A}(\omega)\leq \gamma$. Consider the  $(\ell+1)$-rational contact function $\Phi=z^d/{\boldsymbol{x}^{\boldsymbol{q}}}$ of $\mathcal A$ and let $\lambda$ be the only scalar $0\ne\lambda\in k$ such that $\nu(\Phi-\lambda)>0$. Denote by $\chi$ the critical height $\chi=\chi_{\mathcal A}(\omega)$ and assume $\chi\geq 1$. Let
$\bar\omega^*$ be the  reduced critical part of $({\mathcal A},\omega)$.
\begin{itemize}
\item We say that $({\mathcal A},\omega)$ satisfies the {\em resonant property \textbf{r1}} if and only if $d\geq 2$, $\chi=1$ and
    $\bar\omega^*= \xi \boldsymbol{x}^{I}zd\Phi/\Phi,$ with $\xi\in k$, $\xi\ne 0$ and $I\in {\mathbb Z}^r_{\geq 0}$.
\item We say that $({\mathcal A},\omega)$ satisfies the {\em resonant property \textbf{r2}} if and only if $d=1$ and there are $\xi\ne0$, $I\in {\mathbb Z}^r$,  $\boldsymbol{0}\ne\boldsymbol{\tau}\in {\mathbb C}^r$ and $\boldsymbol{\upsilon}\in {\mathbb C}^r$ such that one of the following expressions holds:
   \begin{description}
    \item[r2a]
    $\bar\omega^*=\boldsymbol{x}^{I}(\Phi-\lambda)^\chi
    d\boldsymbol{x}^{\boldsymbol{\tau}}/\boldsymbol{x}^{\boldsymbol{\tau}}
    $.
    \item[r2b-$\boldsymbol{\upsilon}$]
    $\bar\omega^*=\xi \boldsymbol{x}^{I} (\Phi-\lambda)^{\chi}
 \left(
 d\boldsymbol{x}^{\boldsymbol{\upsilon}}/\boldsymbol{x}^{\boldsymbol{\upsilon}}+
 d(\Phi-\lambda)/(\Phi-\lambda)\right).
 $
    \end{description}
\end{itemize}
\end{definition}
\begin{remark}
\label{rk:resonanciasunificadas}
Condition \textbf{r2b}-$\boldsymbol{0}$
is equivalent to
\begin{equation*}
\bar\omega^*=\xi \boldsymbol{x}^{I-\boldsymbol{p}}z (\Phi-\lambda)^{\chi-1} \frac{d\Phi}{\Phi}, \quad \Phi=z/\boldsymbol{x}^{\boldsymbol{p}}.
\end{equation*}
In a general way, condition \textbf{r2b}-$\boldsymbol{\upsilon}$ is equivalent to
\begin{equation}
\label{eq:condicionr2unificada}
\bar\omega^*
 =\xi
 \boldsymbol{x}^{I-\chi\boldsymbol{p}} (z-\lambda\boldsymbol{x}^{\boldsymbol{p}})^{\chi}
 \left(
 \frac{
 d\boldsymbol{x}^{\boldsymbol{\upsilon}-\boldsymbol{p}}
 }{
 \boldsymbol{x}^{\boldsymbol{\upsilon}-\boldsymbol{p}}
 }
 +
 \frac{d(z-\lambda\boldsymbol{x}^{\boldsymbol{p}})}{(z-\lambda\boldsymbol{x}^{\boldsymbol{p}})}\right)
 ,\quad \Phi=z/\boldsymbol{x}^{\boldsymbol{p}}.
\end{equation}
Note that any $\bar\omega^*\in \Omega^1_{{\mathcal A}}$ written as in Equation \eqref{eq:condicionr2unificada} with $I\in {\mathbb Z}^r$ satisfies automatically that $I-\chi\boldsymbol{p}\in {\mathbb Z}^r_{\geq 0}$, since the coefficient of $dz$ is $\xi\boldsymbol{x}^{I-\chi\boldsymbol{p}} (z-\lambda\boldsymbol{x}^{\boldsymbol{p}})^{\chi-1}$.

We also have that \textbf{r2b}-$\boldsymbol{\upsilon}$ is equivalent to
\begin{equation}
\label{eq:condicionr2unificada2}
\bar\omega^*= \xi\boldsymbol{x}^{I-\chi\boldsymbol{p}} (z-\lambda\boldsymbol{x}^{\boldsymbol{p}})^{\chi-1}
 \left(z\left(
 \frac{
 d\boldsymbol{x}^{\boldsymbol{\upsilon-p}}
 }{
 \boldsymbol{x}^{\boldsymbol{\upsilon-p}}
 }
 +\frac{dz}{z}
 \right)-\lambda\boldsymbol{x^p}
 \frac{
 d\boldsymbol{x}^{\boldsymbol{\upsilon}}
 }{
 \boldsymbol{x}^{\boldsymbol{\upsilon}}
 }
 \right), \quad \Phi=z/\boldsymbol{x}^{\boldsymbol{p}}.
\end{equation}
The reader can verify that conditions $\textbf{r1}$, respectively $\textbf{r2}$, coincide with the conditions {\rm r1}, respectively {\rm r2}, stated in Proposition \ref{prop:estabilidadho} in terms of $\bar\alpha$.
\end{remark}
\begin{remark}
\label{rk:resonantreducedcriticallevel}
Let us write $\bar\omega^*=\sum_{s=0}^\chi z^s\bar\omega^*_s$. Under the resonance conditions, the reduced critical level $\bar\omega^*_\chi$ is obtained as follows:
\begin{align}
\label{eq:chilevelsr1}
\textbf{r1}:\qquad 
\bar\omega^*_\chi& = \bar\omega^*_1=\xi\boldsymbol{x}^{I}\frac{d\Phi}{\Phi}
    =\xi\boldsymbol{x}^{I}
    \left(
    \frac{dz^d}{z^d}-
    \frac{
    d\boldsymbol{x}^{\boldsymbol{p}}
    }{
    \boldsymbol{x}^{\boldsymbol{p}}
    }
    \right),\;  d\geq 2.
    \\
    \label{eq:chilevelsr2a}
    \textbf{r2a}: \qquad
\bar\omega^*_\chi& = \boldsymbol{x}^{I-\chi\boldsymbol{p}}\frac{
    d\boldsymbol{x}^{\boldsymbol{\tau}}
    }{
    \boldsymbol{x}^{\boldsymbol{\tau}}
    }
    \\
    \label{eq:chilevelsr2b}
    \textbf{r2b}\text{-}\boldsymbol{\upsilon}: \qquad
\bar\omega^*_\chi& = \xi\boldsymbol{x}^{I-\chi\boldsymbol{p}}\left(
 \frac{
 d\boldsymbol{x}^{\boldsymbol{\upsilon-p}}
 }{
 \boldsymbol{x}^{\boldsymbol{\upsilon-p}}
 }
 +\frac{dz}{z}
 \right)
\end{align}
\end{remark}

There is an essential difference between the resonance conditions \textbf{r1} and \textbf{r2}. In the case of \textbf{r1}, the ramification index is $d\geq 2$ and in the case of \textbf{r2}, we have $d=1$.  On the other hand, the ``bad resonance'' \textbf{r1} only occurs when $\chi=1$. In  Propositions \ref{prop:r1ycambiodecoordenadas} and \ref{prop:r1ypaquetepuiseux}, we describe the possible transitions between the two types of resonance when $\chi=1$. Roughtly speaking, the resonance \textbf{r1} occurs ``at most once'' during our local uniformization procedure.
\begin{definition}
\label{def:anegativo}
We say that $\boldsymbol{\upsilon}\in {\mathbb Q}^r$ is {\em $\mathcal A$-negative} if and only if $\nu(\boldsymbol{x}^{\boldsymbol{\upsilon}})< 0$.
\end{definition}

\begin{proposition}
\label{prop:r1ycambiodecoordenadas} Let $({\mathcal A},\omega)$ be strictly $\gamma$-prepared with $\chi=\chi_{\mathcal A}(\omega)=1$ and let $({\mathcal A},\omega)\rightarrow ({\mathcal A}^\star,\omega)$ be a normalized coordinate change. Assume that $\chi_{{\mathcal A}^\star}(\omega)=1$ and $({\mathcal A}^\star,\omega)$ is resonant.
We have
\begin{itemize}
  \item[1)] $({\mathcal A},\omega)$ satisfies \textbf{r2a} if and only if  $({\mathcal A}^\star,\omega)$ satisfies \textbf{r2a}.
   \item[2)] If $({\mathcal A},\omega)$ satisfies \textbf{r2b}-$\boldsymbol{\upsilon}$,  where  $\boldsymbol{\upsilon}$ is not ${\mathcal A}$-negative, then $({\mathcal A}^\star,\omega)$ satisfies \textbf{r2b}-$\boldsymbol{\upsilon}^\star$, with $\boldsymbol{\upsilon}^\star$ being also not ${\mathcal A}^\star$-negative.
   \item[3)] If $({\mathcal A},\omega)$ satisfies \textbf{r2b}-$\boldsymbol{\upsilon}$,  where  $\boldsymbol{\upsilon}$ is not ${\mathcal A}$-negative and $\boldsymbol{\upsilon}\ne 0$, then $({\mathcal A}^\star,\omega)$ satisfies \textbf{r2b}-$\boldsymbol{\upsilon}^\star$, with $\boldsymbol{\upsilon}^\star\ne 0$.
\end{itemize}
\end{proposition}
\begin{proof} Recall that the coordinate change is given by $z'=z+f$, where we have $f\in k[[\boldsymbol{x},\boldsymbol{y}_{\leq \ell}]]\cap{\mathcal O}_{\mathcal A}$ and $\nu_{\mathcal A}(f)\geq \nu(z)$. In this situation, we have that $\bar\omega^{\mathcal A}_1$ and $\bar\omega^{{\mathcal A}'}_1$ have the same coefficients. More precisely
\[
\bar\omega^{\mathcal A}_1
=\frac{d\boldsymbol{x}^{\boldsymbol{\tau}}}{\boldsymbol{x}^{\boldsymbol{\tau}}}
+\xi\frac{dz}{z},\quad
\bar\omega^{{\mathcal A}'}_1
=\frac{d\boldsymbol{x}^{\boldsymbol{\tau}}}{\boldsymbol{x}^{\boldsymbol{\tau}}}
+\xi\frac{dz'}{z'},\quad
\bar\omega^{{\mathcal A}^\star}_1
=\frac{d{\boldsymbol{x}^\star}^{\boldsymbol{\tau}^\star}}{{\boldsymbol{x}^\star}^{\boldsymbol{\tau}^\star}}
+\xi\frac{dz'}{z'},
\]
with $\boldsymbol{\tau}^\star=\boldsymbol{\tau}A$, where $A$ is a matrix of non-negative integer coefficients such that $\det A=1$, given by the property that $\nu(\boldsymbol{x}^{\boldsymbol{\mu}})=\nu({\boldsymbol{x}^\star}^{\boldsymbol{\mu}A})$.

In view of Equations \eqref{eq:chilevelsr1}, \eqref{eq:chilevelsr2a} and
 \eqref{eq:chilevelsr2b} we conclude that Statement
(1) is true and it corresponds to the case $\xi=0$.

Assume now that  $({\mathcal A},\omega)$ satisfies \textbf{r2b}-$\boldsymbol{\upsilon}$. In particular the ramification index is one and $\Phi=z/\boldsymbol{x}^{\boldsymbol{p}}$. Since $\xi\ne 0$, we deduce that $({\mathcal A}^\star,\omega)$ satisfies either \textbf{r1} or \textbf{r2b}-$\boldsymbol{\upsilon}^\star$. We have
\[
\boldsymbol{\upsilon}-\boldsymbol{p}=\boldsymbol{\tau}/\xi,\quad  \boldsymbol{\upsilon}A-\boldsymbol{p}A=\boldsymbol{\tau}^\star/\xi.
\]
Write $f=\boldsymbol{x}^{\boldsymbol{p}}(\mu+\tilde f^*)+\tilde f$, where $\mu\in k$, $\nu_{\mathcal A}(\tilde f)>\nu(z)=\nu(\boldsymbol{x}^{\boldsymbol{p}})$ and  $\tilde f^*$ is in the maximal ideal of $\Omega^0_{\mathcal A}$. If $\mu\ne \lambda$, we have that $\nu(z')=\nu(z)$ and  $\Phi'=z'/\boldsymbol{x}^{\boldsymbol{p}}$.
After a strict $\gamma$-preparation the new ramification index is still one and we have $\Phi^\star=z'/{\boldsymbol{x^\star}}^{\boldsymbol{p^\star}}$, where
${\boldsymbol{p^\star}}= {\boldsymbol p}A$. We obtain the resonance condition \textbf{r2b}-$\boldsymbol{\upsilon}^\star$, where $\boldsymbol{\upsilon}^\star=\boldsymbol{\upsilon}A$.
Note that $\boldsymbol{\upsilon}^\star$ is not ${\mathcal A}^\star$-negative, otherwise we would have
\[
0>\nu({\boldsymbol x^\star}^{\boldsymbol{\upsilon^\star}})=\nu(\boldsymbol x^{\boldsymbol{\upsilon}})
\]
and $\boldsymbol{\upsilon}$ should be $\mathcal A$-negative, moreover if $\boldsymbol{\upsilon}\ne 0$ we have
$0\ne \boldsymbol{\upsilon^\star}=\boldsymbol{\upsilon}A $.
This proves Statements (2) and (3) in this case.

Assume now that $\mu=\lambda$ and hence $\nu(z')>\nu(z)$. Write the new rational contact function as $\Phi^\star={z'}^{d^\star}/{\boldsymbol{x}^\star}^{\boldsymbol{p}^\star}$.
If $({\mathcal A}^\star,\omega)$ satisfies \textbf{r1}, then $d^\star\geq 2$. By Equation \eqref{eq:chilevelsr1}, we necessarily have that
\[
-{\boldsymbol{p}^\star}=
d^\star\boldsymbol{\tau}^\star/\xi=d^\star(\boldsymbol{\upsilon} A-\boldsymbol{p}A).
\]
This implies that $\boldsymbol{\upsilon} A=\boldsymbol{p}A-(\boldsymbol{p}^\star/d^\star)$. Taking values, we obtain that
\begin{align*}
\nu(\boldsymbol{x}^{\boldsymbol{\upsilon}})&=
\nu({\boldsymbol{x}^\star}^{\boldsymbol{\upsilon}A})=
\nu({\boldsymbol{x}^\star}^{\boldsymbol{p}A})-
\nu({\boldsymbol{x}^\star}^{\boldsymbol{p}^\star/d^\star})\\
&=
\nu({\boldsymbol{x}}^{\boldsymbol{p}})-\nu(z')=\nu(z)-\nu(z')<0.
\end{align*}
Thus ${\boldsymbol{\upsilon}}$ should be ${\mathcal A}$-negative. Since $\boldsymbol{\upsilon}$ is not ${\mathcal A}$-negative, we necessarily have condition \textbf{r2b}-${\boldsymbol{\upsilon}^\star}$ for $({\mathcal A}^\star,\omega)$, in particular $d^\star=1$. In order to prove Statements (2) and (3), we have only to show that $\boldsymbol{\upsilon}^\star\ne 0$ and  $\boldsymbol{\upsilon}^\star$ is not ${\mathcal A}^\star$-negative.
From Equation \eqref{eq:chilevelsr2b}, we deduce that
$
\boldsymbol{\upsilon}^\star=
\boldsymbol{\tau}^\star/\xi+\boldsymbol{p}^\star
=\boldsymbol{\upsilon}A-\boldsymbol{p}A+\boldsymbol{p}^\star
$ and hence
\[
\boldsymbol{\upsilon}A= \boldsymbol{\upsilon}^\star +\boldsymbol{p}A-\boldsymbol{p}^\star.
\]
We have that $\boldsymbol{\upsilon}\in {\mathbb Q}^r$
 if and only
 if $\boldsymbol{\upsilon}^\star\in {\mathbb Q}^r$. In this case , taking values we have
\[
\nu(\boldsymbol{x}^{\boldsymbol{\upsilon}})=
\nu({\boldsymbol{x}^\star}^{\boldsymbol{\upsilon}A})=
\nu({\boldsymbol{x}^\star}^{\boldsymbol{\upsilon}^\star})
+\nu({\boldsymbol{x}}^{\boldsymbol{p}})
-\nu({\boldsymbol{x}^\star}^{\boldsymbol{p}^\star})=
\nu({\boldsymbol{x}^\star}^{\boldsymbol{\upsilon}^\star})+\nu(z)-\nu(z').
\]
That is, we have
$
 \nu(\boldsymbol{x}^{\boldsymbol{\upsilon}})-
\nu({\boldsymbol{x}^\star}^{\boldsymbol{\upsilon}^\star})=
\nu(z)-\nu(z')<0
$.
 If $\boldsymbol{\upsilon}^\star$ is ${\mathcal A}^\star$-negative or $\boldsymbol{\upsilon}^\star=0$, then $\boldsymbol{\upsilon}$ is  ${\mathcal A}$-negative; hence $\boldsymbol{\upsilon}^\star\ne 0$ and it is not ${\mathcal A}^\star$-negative. We have Statements (2) and (3).
\end{proof}

\begin{remark} In the situation of Proposition \ref{prop:r1ycambiodecoordenadas}, note that if $({\mathcal A}^\star,\omega)$ satisfies \textbf{r2a}, then also $({\mathcal A},\omega)$ satisfies \textbf{r2a}. In particular,  If $({\mathcal A},\omega)$ satisfies \textbf{r2b}-$\boldsymbol{\upsilon}$ it is not possible that $({\mathcal A}^\star,\omega)$ would satisfy \textbf{r2a}, hence it satisfies \textbf{r2}-${\boldsymbol{\upsilon}}^*$ or \textbf{r1}, this last possibility case only occurs when $\boldsymbol{\upsilon}$
is $\mathcal A$-negative.
\end{remark}

\begin{proposition}
\label{prop:r1ypaquetepuiseux}
 Let $({\mathcal A},\omega)$ be strictly $\gamma$-prepared with $\chi_{\mathcal A}(\omega)=1$. Consider a normalized Puiseux's package
 $({\mathcal A},\omega)\rightarrow({\mathcal A}^\star,\omega)$ and
  assume that $({\mathcal A}^\star,\omega)$ is resonant with $\chi_{{\mathcal A}^\star}(\omega)=1$.
We have
\begin{itemize}
  \item[1)] If $({\mathcal A},\omega)$ satisfies \textbf{r2a}, then $({\mathcal A}^\star,\omega)$ satisfies \textbf{r2a}.
  \item[2)] If $({\mathcal A},\omega)$ satisfies \textbf{r2b}-$\boldsymbol{\upsilon}$,  then $({\mathcal A}^\star,\omega)$ does not satisfy \textbf{r2a}.
   \item[3)] If $({\mathcal A},\omega)$ satisfies \textbf{r2b}-$\boldsymbol{\upsilon}$ and $\boldsymbol{\upsilon}$ is not ${\mathcal A}$-negative, then $({\mathcal A}^\star,\omega)$ satisfies \textbf{r2b}-$\boldsymbol{\upsilon}^\star$, where $\boldsymbol{\upsilon}^\star\ne 0$ is not ${\mathcal A}^\star$-negative.
    \item[4)] If $({\mathcal A},\omega)$ satisfies \textbf{r1}, then $({\mathcal A}^\star,\omega)$ satisfies \textbf{r2b}-$\boldsymbol{\upsilon}^\star$,  where $\boldsymbol{\upsilon}^\star\ne 0$ is not ${\mathcal A}^\star$-negative.

  \end{itemize}
  As a consequence, if $({\mathcal A}^\star,\omega)$ satisfies the resonance condition \textbf{r1}, then we necessarily have that $({\mathcal A},\omega)$ satisfies \textbf{r2b}-$\boldsymbol{\upsilon}$ where $\boldsymbol{\upsilon}$ is ${\mathcal A}$-negative.
\end{proposition}
\begin{proof} With our hypothesis, the reduced $1$-level $\bar{\omega}^{{\mathcal A}^\star}_1$ is obtained only from $\bar{\omega}^{{\mathcal A}}_1$.  More precisely, we know that
\[
\bar\omega^{\mathcal A}_1
=\frac{d\boldsymbol{x}^{\boldsymbol{\tau}}}{\boldsymbol{x}^{\boldsymbol{\tau}}}
+\xi\frac{dz}{z},
\quad
\bar\omega^{{\mathcal A}'}_1
=\frac{d{\boldsymbol{x}'}^{\boldsymbol{\tau}'}}{{\boldsymbol{x}'}^{
\boldsymbol{\tau}'}}
+\xi'\frac{dz'}{z'},
\quad
\bar\omega^{{\mathcal A}^\star}_1
=\frac{d{\boldsymbol{x}^\star}^{\boldsymbol{\tau}^\star}}{{\boldsymbol{x}^\star}^{\boldsymbol{\tau}^\star}}
+\xi^\star\frac{dz'}{z'},
\]
where the following holds:
\begin{itemize}
\item[a)] $(\boldsymbol{\tau}',\xi')= (\boldsymbol{\tau},\xi) C$, where $C$ is the matrix that appears in Equation \eqref{eq:puiseux2}.
\item[b)] $(\boldsymbol{\tau}^\star,\xi^\star)= (\boldsymbol{\tau'}A,\xi')$, where $A$ is a matrix with nonnegative integer coefficients and $\det A=1$.
\end{itemize}

Assume first that $({\mathcal A},\omega)$ satisfies \textbf{r2}. Then the ramification index $d$ of $\mathcal A$ is one.  Hence, the rational contact function is $\Phi=z/\boldsymbol{x}^{\boldsymbol{p}}$ and the matrix $C$ has the form given in Equation \eqref{eq:ecuaciondigulauno}
\begin{equation}
\label{eq:ecuaciondigulauno2}
C=
\left(
\begin{array}{ccc}
{C_0}&\vline&{\boldsymbol{0}}\\
\hline
\boldsymbol{\tilde p}&\vline& 1\\
\end{array}
\right),\quad \boldsymbol{\tilde p}=\boldsymbol{p}C_0, \quad
C^{-1}=
\left(
\begin{array}{ccc}
{C_0^{-1}}&\vline&{\boldsymbol{0}}\\
\hline
-\boldsymbol{ p}&\vline& 1\\
\end{array}
\right).
\end{equation}
In particular, we have $\xi^\star=\xi'=\xi$. Put $C_0^\star=C_0A$. Let us recall that
\[
\nu({\boldsymbol{x}}^{\boldsymbol{\mu}})=
\nu({\boldsymbol{x}'}^{\boldsymbol{\mu}C_0})=
\nu({\boldsymbol{x}^\star}^{\boldsymbol{\mu}C_0^\star})
.
\]

$\bullet$ If $({\mathcal A},\omega)$ satisfies \textbf{r2a}, then $\xi=\xi'=\xi^\star=0$. Then $({\mathcal A}^\star,\omega)$ satisfies \textbf{r2a}. This proves Statement (1).

$\bullet$ If $({\mathcal A},\omega)$ satisfies \textbf{r2b}, then $\xi=\xi'=\xi^\star\ne0$. Hence $({\mathcal A}^\star,\omega)$ does not satisfy \textbf{r2a}.  This proves Statement (2).

$\bullet$ Let us prove Statement (3). We assume that $({\mathcal A},\omega)$ satisfies \textbf{r2b}-$\boldsymbol{\upsilon}$, where $\boldsymbol{\upsilon}$ is not ${\mathcal A}$-negative. We know that $\xi=\xi^\star\ne 0$ and that $({\mathcal A}^\star,\omega)$ does not satisfy  \textbf{r2a} nor \textbf{r1}, hence it satisfies \textbf{r2b}-$\boldsymbol{\upsilon}^\star$. Let us show that $\boldsymbol{\upsilon}^\star\ne 0$ and it is not ${\mathcal A}^\star$-negative, this proves (2). Assume that $\boldsymbol{\upsilon}^\star$ is ${\mathcal A}^\star$-negative or $\boldsymbol{\upsilon}^\star=0$. Let us note that $\boldsymbol{\upsilon}^\star-\boldsymbol{p}^\star=
\boldsymbol{\tau}^\star/\xi^\star$, on the other hand, we have
\[
\boldsymbol{\tau}^\star/\xi^\star=\boldsymbol{\tau}^\star/\xi= (\boldsymbol{\tau}+\xi\boldsymbol{p})C_0^\star/\xi=
(\boldsymbol{\tau}/\xi+\boldsymbol{p})C_0^\star=
\boldsymbol{\upsilon}C^\star_0.
\]
Then we have $\boldsymbol{\upsilon}^\star-\boldsymbol{p}^\star=\boldsymbol{\upsilon}C^\star_0$.
We conclude that $\boldsymbol{\upsilon}\in {\mathbb Q}^r$, moreover, taking values, we have that
\[
0\geq \nu({\boldsymbol{x}^\star}^{\boldsymbol{\upsilon}^\star})=
\nu({\boldsymbol{x}^\star}^{\boldsymbol{\upsilon}C^\star_0})
+ \nu({\boldsymbol{x}^\star}^{\boldsymbol{p}^\star})
=\nu({\boldsymbol{x}}^{\boldsymbol{\upsilon}})
+
\nu(z')
>
\nu({\boldsymbol{x}}^{\boldsymbol{\upsilon}}).
\]
Then $\boldsymbol{\upsilon}$ should be ${\mathcal A}$-negative. This shows Statement (2).

$\bullet$ It remains to prove Statement (4). We assume that $({\mathcal A},\omega)$ satisfies \textbf{r1}. Hence the ramification index is $d\geq 2$ and $\Phi=z^d/\boldsymbol{x}^{\boldsymbol{p}}$. By Equation \eqref{eq:chilevelsr1}, we have
\[
\boldsymbol{\tau}/\xi=-\boldsymbol{p}/d.
\]
The matrix $C$ has not the form in Equation \eqref{eq:ecuaciondigulauno2}. Anyway, we have that
\[
(\boldsymbol{\tau}',\xi')= (\xi/d) (-\boldsymbol{p},d)C= (\xi/d)(\boldsymbol{0},1).
\]
This implies that $\xi^\star=\xi'\ne 0$ and $\boldsymbol{\tau}'=\boldsymbol{0}$, thus $\boldsymbol{\tau^\star}=\boldsymbol{\tau}'A=\boldsymbol{0}$.  In particular we have neither \textbf{r1} nor \textbf{r2a}. We get the resonance condition \textbf{r2b}-$\boldsymbol{\upsilon}^\star$, with $\boldsymbol{\upsilon}^\star=\boldsymbol{p}^\star$ and $\nu(z')=\nu({\boldsymbol{x}}^{\boldsymbol{p}^\star})>0$. Thus $\boldsymbol{\upsilon}^\star$ is not ${\mathcal A}^\star$-negative and $\boldsymbol{\upsilon}^\star\ne 0$.
\end{proof}

\subsection{Never Ramified Cases and Horizontal Functions}
\label{Normalized Transformations and Conditionated Local Uniformization}
 Some of the results in this subsection may be considered as a truncated version of Classical Zariski's Local Uniformization \cite{Zar1} for the case of formal functions, although we add the property of respecting the normalized transformations used for the case of foliations.

\begin{definition}
\label{def:noramification} Consider $({\mathcal A},\omega)$ with the property of { $\chi$-fixed critical height}. We say that $({\mathcal A},\omega)$ is {\em  never ramified} if and only if for any normalized transformation
$({\mathcal A},\omega)\rightarrow ({\mathcal B},\omega)$ the $(\ell+1)$-ramification index of ${\mathcal B}$ is equal to one.
\end{definition}
In next subsections  we will reduce our study to never ramified cases.

Let us define the {\em horizontal coefficient $H_{{\mathcal A},\omega}$} as follows.  Write $\omega$ in $\mathcal A$ as
\begin{equation*}
\omega=\eta+hdz, \quad \eta=\sum_{i=1}^rf_i\frac{dx_i}{x_i}+\sum_{j=1}^{\ell}g_jdy_j.
\end{equation*}
Let $H_{\mathcal{A},\omega}$ denote the coefficient $h=\omega(\partial/\partial z)$.

  One important feature of never ramified situations is that the $H_{{\mathcal A},\omega}$ is stable in the following sense:

\begin{lemma}
 \label{lema:stableh} Consider $({\mathcal A},\omega)$ with the property of $\chi$-fixed critical height and never ramified. We have:
 \begin{itemize}
 \item[1)] If $({\mathcal A},\omega)\rightarrow ({\mathcal B},\omega)$ is a normalized coordinate change, then $H_{{\mathcal B},\omega}=H_{{\mathcal A},\omega}$.
 \item[2)] If $({\mathcal A},\omega)\rightarrow ({\mathcal B},\omega)$ is  a normalized Puiseux's package and $\Phi=z/{\boldsymbol{x}}^{\boldsymbol{p}}$ is the contact rational function in ${\mathcal A}$, then
     $H_{{\mathcal B},\omega}= {\boldsymbol{x}}^{\boldsymbol{p}}H_{{\mathcal A},\omega}$.
 \end{itemize}
\end{lemma}
\begin{proof} The cases of a coordinate change or a $\gamma$-strict preparation are straighforward and left to the reader. If we consider a normalized Puiseux package, we necessarily have that resonant condition \textbf{r2} holds, since we are non ramified and $\chi$-fixed. The contact rational function is given by $\Phi=z/\boldsymbol{x}^{\boldsymbol{p}}$, where $z^{\mathcal B}=\Phi-\lambda$. Noting that
\[
hdz= zh\left(
\frac{d\Phi}{\Phi}+
\frac{d{\boldsymbol{x}}^{\boldsymbol{p}}}{{\boldsymbol{x}}^{\boldsymbol{p}}}
\right)=
\boldsymbol{x}^{p}h
dz^{\mathcal B}+(z^{\mathcal B}+\lambda)
h{d{\boldsymbol{x}}^{\boldsymbol{p}}},
\]
the result follows.
\end{proof}

Let us consider a formal function $F\in \Omega^0_{\mathcal A}$. Looking at Definition \ref{def:truncated_final_function}, we recall that $({\mathcal A}, F)$ is $\gamma$-final if either $\nu_{\mathcal A}(F)>\gamma$ or $F$ has the form
\[
F=\boldsymbol{x}^IU+\tilde F, \quad \nu(\boldsymbol{x}^I)=\nu_{\mathcal A}(F)\leq\gamma,\quad  \nu_{\mathcal A}(\tilde F)>\nu_{\mathcal A}(F),
\]
where $U\in \Omega^0_{\mathcal A}$ is a unit. We need a strong version of this  concept, where we can assume that $\tilde F=0$. We precise it in Definition
\ref{def:stronggammafinal} below:
\begin{definition}
\label{def:stronggammafinal}
 Consider a pair $({\mathcal A},F)$, where $F\in \Omega^0_{\mathcal A}$ is a formal function in $\mathcal A$. We say that $({\mathcal A}, F)$ is {\em strongly $\gamma$-final} if either $\nu_{\mathcal A}(F)>\gamma$ or $F=\boldsymbol{x}^IU$, where $U$ is a unit in $\Omega^0_{\mathcal A}$ that can be written as
 \[
 U=\xi+\tilde U,\quad \nu_{\mathcal A}(\tilde U)>0,\quad 0\ne \xi\in k.
 \]
\end{definition}
Let us note that to be strongly $\gamma$-final  is stable under any nested transformation.

 \begin{proposition}
 \label{prop:conditionatedluofafunction} Consider $({\mathcal A},\omega)$ with the property of $\chi$-fixed critical height and never ramified.
  Take a formal function $F\in k[[\boldsymbol{x},\boldsymbol{y}_{\leq\ell},z]]\subset \Omega^0_{\mathcal A}$. There is a normalized transformation
  $
  ({\mathcal A},\omega)\rightarrow ({\mathcal B},\omega)
  $
  such that $({\mathcal B}, F)$
  is strongly $\gamma$-final.
 \end{proposition}
 \begin{proof} We give quick indications of the proof, that follows the same lines of the case of $1$-forms, without the difficulty of the preparation.

  Let us note that the induction hypothesis allows us to make $\rho$-final any formal function $G\in k[[{\boldsymbol{x}},\boldsymbol{y}_{\leq\ell}]]$ by means of an $\ell$-nested transformation; to see this is enough to consider the differential $dG$, in view of Corollary \ref{cor:gammafinalfdf}. Moreover, these $\ell$-nested transformations may be ``integrated'' in a $\gamma$-strict preparation of $({\mathcal A},\omega)$, just by completing the $\gamma$-preparation.

 Now, write
 \[
 F=\sum_{s\geq 0}z^sF_s;\quad F_s\in k[[{\boldsymbol{x}},\boldsymbol{y}_{\leq\ell}]].
 \]
 We can perform a $\gamma$-preparation of $F$ by means of an $\ell$-nested transformation. To do this is enough to make $\gamma$-final one by one the levels $F_s$ with $s\leq \gamma\nu(z)$. This can be done with a normalized transformation and hence we have ramification index equal to one. Now, we have a critical vertex at height $\xi$.
Assume that the critical height does not drop.  By a new normalized Puiseux's package, we get the property that the main vertex coincides with the critical vertex; moreover, this property remains true since we assume that the critical height does not drop.

 By making combinatorial independent blow-ups as in Proposition \ref{prop:monlocprincipalization}, we  divide $F$ by $\boldsymbol{x}^I$ with $\nu(\boldsymbol{x}^I)=\nu_{\mathcal A}F$ and hence we obtain that $F_{\xi}$ is a unit.
  Up to performing additional $\ell$-nested transformations, we can assume that
 \[
 F_{\xi-i}=G_i+\tilde G_i, \quad \nu_{\mathcal A}(\tilde G_i)>\gamma, \; G_i\in k[\boldsymbol{x},\boldsymbol{y}_{\leq\ell}]\subset \mathcal {O}_{\mathcal A},
 \]
 for $i=0,1$.
 This property is stable under any new $\ell$-nested transformation. This allows us to make a Tschirnhausen coordinate change
 \[
 z'=z+\frac{1}{\xi}\frac{G_{\xi-1}}{G_{\xi}}.
 \]
 By classical arguments, when we perform a new $({\mathcal A},\omega)$- normalized Puiseux's package ${\mathcal A}\rightarrow{\mathcal B}$, either the critical height $\xi$ drops ( when we are touching the ``empty part'' of the $\xi-1$-level in the Newton Polygon ) or the value $\nu(z^{\mathcal B})\geq \min\{\gamma, \nu(z)\}$.
 If $\nu(z^{\mathcal B})\geq \gamma$ we are done, in the next transformation we get $\varsigma_{\mathcal A}(F)\geq \gamma$. Otherwise, we see that $\varsigma_{\mathcal B}(F)\geq\varsigma_{\mathcal A}(F)+\nu(z)$. We end in a finite number of steps.

 Once $F$ is $\gamma$-final, we write $F=\sum_{I} \boldsymbol{x^I}F_I$. Now we perform a monomialization of the ideal generated by the $\boldsymbol{x^I}$ by means of independent combinatorial blow-ups. This gives an expression of $F$ as $F=\boldsymbol{x}^IU$, where $U$ is a unit that we write $U=\xi+\tilde U$ and $\tilde U$ is in the maximal ideal, by performing a normalized transformation containing a $j$-Puiseux's package for any $j=1,2,\ldots,\ell+1$ we obtain that $\nu_{\mathcal A}(\tilde U)>0$, see Proposition \ref{prop:notfinalformsandcriticalvalue}.
 \end{proof}

 \begin{corollary}
 \label{cor:hachereducido} Consider $({\mathcal A},\omega)$ with the property of $\chi$-fixed critical height and never ramified.
There is a normalized transformation
$
({\mathcal A},\omega)\rightarrow ({\mathcal B},\omega)
$
such that $({\mathcal B}, H_{{\mathcal B},\omega})$ is strongly $\gamma$-final.
Moreover, this situation is stable under further normalized transformations.
\end{corollary}
\begin{proof} It follows
 from Lemma \ref{lema:stableh}
and
Proposition \ref{prop:conditionatedluofafunction}.
\end{proof}
\begin{remark}
\label{rk:Hstability} Take $({\mathcal A},\omega)$ with the property of $\chi$-fixed critical height and never ramified. Assume that $H_{{\mathcal A},\omega}$ is strongly $\gamma$-final and consider an  $({\mathcal A},\omega)$-normalized transformation ${\mathcal A}\rightarrow{\mathcal B}$. By Lemma \ref{lema:stableh}, we see that
\[
 \nu_{{\mathcal B}}(H_{{\mathcal B},\omega})\geq
 \nu_{{\mathcal A}}(H_{{\mathcal A},\omega}).
\]
In particular, if $({\mathcal A}, H_{{\mathcal A},\omega})$ is $\gamma$-final recessive, then $({\mathcal B}, H_{{\mathcal B},\omega})$ is also $\gamma$-final recessive. Let us say that
$({\mathcal A},\omega)$ has {\em $\gamma$-recessive horizontal stability} if $H_{{\mathcal A},\omega}$ is $\gamma$-final recessive. On the other hand, we say that $({\mathcal A},\omega)$ has {\em $\gamma$-dominant horizontal stability} if $({\mathcal B},H_{{\mathcal B},\omega})$ is $\gamma$-final dominant under any normalized transformation $({\mathcal A},\omega)\rightarrow({\mathcal B},\omega)$.
We conclude that there is a normalized transformation
$({\mathcal A},\omega)\rightarrow({\mathcal B},\omega)$ such that $({\mathcal B},\omega)$
has $\gamma$-recessive or $\gamma$-dominant horizontal stability.
\end{remark}

\subsection{First Steps in the Reduction to Critical Height One}
\label{Reduction to Critical Height One}
Let us start the proof of
Proposition \ref{prop:critical height2}. We look for a contradiction with the existence of a $\gamma$-truncated formal foliated space
 $({\mathcal A},\omega)$ with the property of $\chi$-fixed critical height, where $\chi\geq 2$. Thus, we assume we have such an $({\mathcal A},\omega)$.

If $({\mathcal A},\omega)\rightarrow ({\mathcal B},\omega)$ is a normalized transformation, then $({\mathcal B},\omega)$ also have the property of the $\chi$-fixed critical height. Then $({\mathcal B},\omega)$ has the resonance property \textbf{r2}, since $\chi\geq 2$. In particular $({\mathcal A},\omega)$ is never ramified, accordingly with Definition \ref{def:noramification}.

By Corollary  \ref{cor:hachereducido} and Remark \ref{rk:Hstability}, we can make the following assumption:
\begin{itemize}
\item[A1:] The horizontal coefficient $H_{{\mathcal A},\omega}$ is strongly $\gamma$-final and $({\mathcal A},\omega)$ has the $\gamma$-recessive or $\gamma$-dominant horizontal stability.
\end{itemize}
Note that A1 is stable under any new normalized transformations and thus $({\mathcal B},\omega)$ also satisfies A1, when $({\mathcal A},\omega)\rightarrow ({\mathcal B},\omega)$ is a normalized transformation.

\begin{lemma}  Let $({\mathcal A},\omega)\rightarrow({\mathcal B},\omega)$ be a normalized transformation that contains at least one normalized Puiseux's package.  Then
\begin{equation}
\label{eq:vertice principal}
\nu_{{\mathcal B}}(\omega)=\varsigma_{\mathcal B}(\omega)- \chi\nu(z^{\mathcal B}),
\end{equation}
where $z^{\mathcal B}$ is the $(\ell+1)$-th dependent parameter in $\mathcal B$. Moreover, the property in Equation \eqref{eq:vertice principal} is stable under further normalized transformations.
\end{lemma}
\begin{proof}  We leave to the reader the details of this verification, based on the arguments in the proof of Proposition \ref{prop:stability critical height}.
\end{proof}
 Then, we can make another assumption:
 \begin{itemize}
\item[A2:] $\nu_{\mathcal A}(\omega)=\varsigma_{\mathcal A}(\omega)-\chi\nu(z)$.
\end{itemize}
Note also that A2 is stable under new normalized transformations.
\begin{remark}
\label{rk:alturacriticaligualprincipal}
The property A2 is equivalent to saying that the critical vertex and the main vertex coincide.
\end{remark}
Let us decompose $\omega$ into levels
$
\omega=\sum_{s\geq 0}z^s\omega_s$, $\omega_s=\eta_s+h_sdz/z
$, where we recall that $h=\sum_{s\geq 1}h_sz^{s-1}$.
\begin{lemma}
\label{lema:nuhachechi}
We have
$\nu_{\mathcal A}(h_\chi)>\varsigma_{\mathcal A}(\omega)-\chi\nu(z)$.
\end{lemma}
\begin{proof} This is obvious if $\nu_{\mathcal A}(h)>\gamma$. When $h=U\boldsymbol{x}^J$, we have
\begin{equation*}
\nu_{\mathcal A}(h_\chi)\geq\nu_{\mathcal A}(h)=\nu(\boldsymbol{x}^J)\geq\varsigma_{\mathcal A}(\omega)-\nu(z)>\varsigma_{\mathcal A}(\omega)-\chi\nu(z),
\end{equation*}
recalling that $\chi\geq 2$.
\end{proof}
\begin{lemma}
\label{lema:descrition of the reduced critical part} The resonance condition \textbf{r2a} is satisfied for $({\mathcal A}, \omega)$. That is, the reduced critical part $\bar\omega^*$ of $\omega$ has the form
\begin{equation}
\label{eq:segmentocriticoconr2a}
\bar\omega^*
 =
 \boldsymbol{x}^{I} (z-\lambda\boldsymbol{x}^{\boldsymbol{p}})^{\chi}
 \frac{
 d\boldsymbol{x}^{\boldsymbol{\tau}}
 }{
 \boldsymbol{x}^{\boldsymbol{\tau}}
 }.
\end{equation}
\end{lemma}
\begin{proof} We already know that the resonant condition \textbf{r2} is satisfied. We have only to show that condition \textbf{r2b} cannot occur. If \textbf{r2b} holds, the reduced part $\bar\omega^*_\chi$ of the $\chi$-level $\omega_\chi$ is given by
 \[
 \bar\omega^*_\chi=\xi\boldsymbol{x}^{I-\chi\boldsymbol{p}}\left(
 \frac{
 d\boldsymbol{x}^{\boldsymbol{\upsilon-p}}
 }{
 \boldsymbol{x}^{\boldsymbol{\upsilon-p}}
 }
 +\frac{dz}{z}
 \right),
 \]
 see Equation \eqref{eq:chilevelsr2b}. This implies that $\nu_{\mathcal A}(h_\chi)=\nu(\boldsymbol{x}^{I-\chi\boldsymbol{p}})=\varsigma_{\mathcal A}(\omega)-\chi\nu(z)$.
This is not compatible with the fact that $\nu_{\mathcal A}(h_\chi)>\nu(\boldsymbol{x}^I)=\varsigma_{\mathcal A}(\omega)-\chi\nu(z)$ stated in Lemma \ref{lema:nuhachechi}. Then we have \textbf{r2a} and Equation \eqref{eq:segmentocriticoconr2a} holds.
\end{proof}

By expanding the binomial $(z-\lambda\boldsymbol{x}^{\boldsymbol{p}})^{\chi}$ in Equation \eqref{eq:segmentocriticoconr2a}, we see
that each level $\omega_s$ is $(\varsigma_{\mathcal A}(\omega)-s\nu(z))$-final dominant, for $s=0,1,\ldots,\chi$.  Let us write
\[
\omega_s=
\boldsymbol{x}^{I_s}\omega_{I_s,s}+
\sum_{\nu(\boldsymbol{x}^K) > \nu(\boldsymbol{x}^{I_s})}\boldsymbol{x}^K\omega_{K,s},
\]
where $\omega_{I_s,s}$ is $0$-final dominant and $\nu(\boldsymbol{x}^{I_s})=\varsigma_{\mathcal A}(\omega)-s\nu(z)$, for $s=0,1,\ldots,\chi$.
Up to performing combinatorial blow-ups in the independent parameters $\boldsymbol{x}$ as in  Proposition \ref{prop:monlocprincipalization}, we obtain that $\boldsymbol{x}^{I_s}$ divides each one of the $\boldsymbol{x}^{K}$ and then we have
\begin{equation*}
\omega_s=\boldsymbol{x}^{I_s}\beta_s \quad \text{and} \quad \nu(\boldsymbol{x}^{I_s})=\varsigma_{\mathcal A}(\omega)-s\nu(z) \quad \text{for} \quad s=0,1,\ldots,\chi,
\end{equation*}
where $\beta_s$ is $0$-final dominant for $s=0,1,\ldots,\chi$. Let us write
\begin{equation*}
\beta_s = \alpha_s+H_s\frac{dz}{z} , \quad \eta_s = \boldsymbol{x}^{I_s}\alpha_s \quad \text{and} \quad h_s=\boldsymbol{x}^{I_s}H_s \quad \text{for} \quad s=0,1,\dots,\chi.
\end{equation*}
\begin{lemma}
\label{lema:sinetiqueta}
Let $\epsilon$ denote the value
\begin{equation*}
\epsilon=\min\{\gamma,\nu_{\mathcal A}(h)\}-\varsigma_{\mathcal A}(\omega)+\nu(z).
\end{equation*}
We have $\epsilon>0$ and $\nu_{\mathcal A}(H_s)\geq \epsilon$ for $s=0,1,\ldots,\chi$.
\end{lemma}
\begin{proof} If $\nu_{\mathcal A}(h)>\gamma$, we have $\epsilon= \gamma-\varsigma_{\mathcal A}(\omega)+\nu(z)\geq \nu(z)>0$. Assume that $\nu_{\mathcal A}(h)\leq \gamma$. Then $h=U\boldsymbol{x}^{J}$ where $U$ is a unit in $\Omega^0_{\mathcal A}$ and hence $\nu_{\mathcal A}(h)=\nu(\boldsymbol{x}^J)$. Write $h=\sum_{s\geq 1}z^{s-1}h_{s}$, with $h_s\in k[[\boldsymbol{x},\boldsymbol{y}_{\leq \ell}]]$. We have that $h_1=\boldsymbol{x}^JV_1$, where $V_1$ is a unit in $k[[\boldsymbol{x},\boldsymbol{y}_{\leq \ell}]]$. The $1$-level $\omega_1$ of $\omega$ is given by
\[
\omega_1=\eta_1+h_1dz/z=\boldsymbol{x}^{I_1}(\alpha_1+\boldsymbol{x}^{J-I_1}V_1dz/z).
\]
We necessarily have that $\epsilon=\nu(\boldsymbol{x}^{J-I_1})>0$, otherwise we obtain an incompatibility with Equation \eqref{eq:segmentocriticoconr2a}.
In fact, for any $s=0,1,\ldots,\chi$, we have that $h_s=\boldsymbol{x}^{J}V_s$, where $V_s\in k[[\boldsymbol{x},\boldsymbol{y}_{\leq \ell}]]$ and $V_0=0$. We obtain that
\[
H_s= \boldsymbol{x}^{J-I_s}V_s,  \quad s=0,1,\ldots,\chi.
\]
We know that $H_0=0$ and hence $\nu_{\mathcal A}(H_0)=\infty>\epsilon$. For $s\geq 1$ we have
\[
\nu_{\mathcal A}(H_s)=\nu_{\mathcal A}(\boldsymbol{x}^{J-I_s})+\nu_{\mathcal A}(V_s)\geq \nu_{\mathcal A}(\boldsymbol{x}^{J-I_s})=\nu_{\mathcal A}(h)-\varsigma_{\mathcal A}(\omega)+s\nu(z)\geq\epsilon.
\]
This ends the proof.
\end{proof}
\begin{remark} Since $\beta_s=\alpha_s+H_sdz/z$ is $0$-final dominant and $\nu_{\mathcal A}(H_s)>\epsilon$, we conclude that $\alpha_s$ is $0$-final dominant for $s=0,1,\ldots,\chi$.
\end{remark}

Lemma \ref{lema:proporcionalidad} below provides the key property we need to perform a ``useful''  Tschirnhausen transformation. Let us note that we need to invoke the $\gamma$-truncated integrability condition.
\begin{lemma}
\label{lema:proporcionalidad} For any $1\leq s\leq \chi$, we have  $\alpha_s=U_s\alpha_0+\tilde\alpha_s$, where $U_s$ is a unit in
$k[[\boldsymbol{x},\boldsymbol{y}_{\leq \ell}]]$  and $\nu_{\mathcal A}(\tilde\alpha_s)\geq \min\{\epsilon,\nu(z)\}$.
\end{lemma}
\begin{proof} Let us consider the description of the truncated integrability condition given in Subsection \ref{Truncated Integrability Condition}. In particular, recall that $\nu_{\mathcal A}(\Delta_s)\geq 2\gamma$, where
\begin{equation*}
 \Delta_s=\sum_{i+j=s}\Delta_{ij},
\quad\Delta_{ij}=j\eta_j\wedge\eta_i+h_id\eta_j+\eta_i\wedge dh_j ,
\end{equation*}
see Equation \eqref{eq:condicionesinttruncada}. For any $0\leq i,j\leq s$ such that $i+j=s$ we have
\[
\Delta_{ij}=\boldsymbol{x}^{I_i+I_j}\Xi_{ij},\quad \Xi_{ij}=
\left(j\alpha_j\wedge\alpha_i+
H_j(\frac{d\boldsymbol{x}^{I_i}}{\boldsymbol{x}^{I_i}}\wedge\alpha_i)
+\alpha_i\wedge
(H_j\frac{d\boldsymbol{x}^{I_j}}{\boldsymbol{x}^{I_j}}+dH_j)
\right).
\]
Then
$
\Xi_{ij}=\left(j\alpha_j\wedge\alpha_i+\vartheta_{ij}
\right)
$, where $\nu_{\mathcal A}(\vartheta_{ij})\geq \epsilon$. If $i+j=s$ we have that
\[
\nu(\boldsymbol{x}^{I_i+I_j})=2\varsigma_{\mathcal A}(\omega)-s\nu(z).
\]
In particular, there is $J_s$ such that $J_s=I_i+I_j$ when $i+j=s$.  Then
\begin{equation*}
\Delta_s=\boldsymbol{x}^{J_s}\Xi_s,\quad \Xi_s=\sum_{i+j=s}(j\alpha_j\wedge\alpha_i)+\vartheta_s,\quad \vartheta_s=\sum_{i+j=s}\vartheta_{ij},
\end{equation*}
where we have $\nu_{\mathcal A}(\vartheta_s)\geq \epsilon$. Let us consider the following computation:
\begin{align*}
\nu_{\mathcal A}(\Xi_s)&=\nu_{\mathcal A}(\Delta_s)-\nu(\boldsymbol{x}^{J_s})=
\\&=\nu_{\mathcal A}(\Delta_s)-2\varsigma_{\mathcal A}(\omega)+s\nu(z)\geq 2(\gamma-\varsigma_{\mathcal A}(\omega))+s\nu(z)\geq s\nu(z).
\end{align*}
Then we have $\nu_{\mathcal A}(\Xi_s)\geq s\nu(z)$.

Now, let us start the proof by finite induction on $1\leq s\leq \chi$. If $s=1$, we have
$
\Xi_1=\alpha_1\wedge\alpha_0+\vartheta_1
$.
Hence $\alpha_1\wedge\alpha_0=\Xi_1-\vartheta_1$. We deduce that \[\nu_{\mathcal A}(\alpha_1\wedge\alpha_0)\geq\min\{\epsilon,\nu(z)\}.\]  Recalling that $\alpha_0$ is $0$-final dominant, we obtain $U_1$ and $\tilde\alpha_1$
by Proposition
 \ref{prop:trucateddivision}.

 Assume  that $1<s\leq\chi$ and that the result is true for $1\leq s'<s$. Denote
 \[
 \varpi_s=  \sum_{1\leq i,j\leq s-1; i+j=s}j\alpha_j\wedge\alpha_i=
 \sum_{1\leq i,j\leq s-1; i+j=s}j(U_j\alpha_0+\tilde\alpha_j)\wedge(U_i\alpha_0+\tilde\alpha_i).
 \]
 Note that $\nu_{\mathcal A}(\varpi_s)\geq \min\{\epsilon,\nu(z)\}$. Now, we have
 $s\alpha_s\wedge\alpha_0=
 \Xi_s-\vartheta_s-\varpi_s$.
 This implies that $\nu_{\mathcal A}(\alpha\wedge\alpha_0)\geq\min\{\epsilon,\nu(z)\}$ and we obtain $U_s$ and $\tilde\alpha_s$ as for $s=1$.
\end{proof}

\subsection{Reduction to Critical Height One. Tschirnhausen Transformation}
\label{Reduction to Critical Height One 2} In this subsection,  we end the proof of Proposition \ref{prop:critical height2}.
Let us take the notations and reductions in Subsection \ref{Reduction to Critical Height One}. We know that we may assume the following additional properties:
\begin{itemize}
\item[A1:]Then main vertex and the critical vertex coincide. This is equivalent to saying that $\nu_{\mathcal A}(\omega)=\varsigma_{\mathcal A}(\omega)-\chi\nu(z)$.
\item[A2:] $H_{{\mathcal A},\omega}$ is strongly $\gamma$-final and $({\mathcal A},\omega)$ has the $\gamma$-recessive or $\gamma$-dominant horizontal stability.
\item[A3:] Up to performing a $0$-nested transformation, for any $0\leq s\leq \chi$, we have $\omega_s=\boldsymbol{x}^{I_s}\beta_s$, with $\beta_s$ being $0$-final dominant and $\nu(\boldsymbol{x}^{I_s})=\varsigma_{\mathcal A}(\omega)-s\nu(z)$.
\end{itemize}
The above properties are stable under any further normalized transformations.
Let us recall that a $0$-nested transformation can be considered a normalized transformation, see Remark \ref{rk: degeneratenormalizedtransforamtions}.

By Lemma \ref{lema:sinetiqueta}, for $0\leq s\leq\chi$ we know that
\[
\beta_s=\alpha_s+H_sdz/z,
\]
with $\nu_{\mathcal A}(H_s)\geq\min\{\nu(z),\epsilon\}$,
where $\epsilon>0$ is given by
\begin{equation*}
\epsilon=\min\{\gamma,\nu_{\mathcal A}(h)\}-\varsigma_{\mathcal A}(\omega)+\nu(z).
\end{equation*}
In particular, each $\alpha_s$ is $0$-final dominant. We also know that there are units $U_s\in k[[\boldsymbol{x},\boldsymbol{y}_{\leq \ell}]]$ such that
\[
\alpha_s=U_s\alpha_0+\tilde\alpha_s,\quad 1\leq s\leq \chi, \quad \nu_{\mathcal A}(\tilde\alpha_s)\geq\min\{\nu(z),\epsilon\}.
\]
Let us prepare the units $U_s$ to obtain a ``Tschirnhausen coordinate change''.
\begin{lemma}
\label{lema:preparingtheunits}
Up to performing an additional $\ell$-nested transformation, the formal units $U_s$ can be written as
\[
U_s=W_s+\boldsymbol{x}^{K}\widetilde{W}_s,
\]  with $\nu(\boldsymbol{x}^K)>\gamma$, $W_s\in
k[\boldsymbol{x},\boldsymbol{y}_{\leq\ell}]\subset {\mathcal O}_{\mathcal A}$ and $\widetilde{W}_s\in k[[\boldsymbol{x},\boldsymbol{y}_{\leq\ell}]]$.
\end{lemma}
\begin{proof}
Write
$U_s\in k[[\boldsymbol{x},\boldsymbol{y}_{\leq \ell}]]$ in a $\gamma$-truncated way as
$
U_s=W_s+\widetilde W_s$, where  $W_s\in k[\boldsymbol{x},\boldsymbol{y}_{\leq \ell}]\subset {\mathcal O}_{\mathcal A}$,
and $\widetilde W_s$ is a formal series that we write as
\[
\widetilde W_s=\sum_{\nu(\boldsymbol{x}^I\boldsymbol{y}_{\leq \ell}^K)> \gamma}c_{IK}\boldsymbol{x}^I\boldsymbol{y}_{\leq \ell}^K.
\]
Let us perform an $\ell$-nested transformation containing a $j$-Puiseux's package for each $1\leq j\leq \ell$. Then each of the $c_{IK}\boldsymbol{x}^I\boldsymbol{y}_{\leq \ell}^K$ becomes a unit times a monomial in the independent variables, see Proposition \ref{prop:notfinalformsandcriticalvalue}. We principalize the list of of such monomials by using Proposition
\ref{prop:monlocprincipalization} and we are done.
\end{proof}

Let us perform the ``Tschirnhausen type'' coordinate change ${\mathcal A}\rightarrow {\mathcal A}'$ defined by
\begin{equation*}
z'=z-F{x}^{\boldsymbol{p}},\quad  F={W_{\chi-1}}(\chi W_\chi)^{-1}\in {\mathcal O}_{\mathcal A}\cap k[[\boldsymbol{x},\boldsymbol{y}_{\leq \ell}]]
\end{equation*}
and let us consider a $\gamma$-strict preparation $({\mathcal A}',\omega)\rightarrow ({\mathcal A}^\star,\omega)$. Lemma \ref{lema:avance epsilon} below is a key observation for finding the desired contradiction:
\begin{lemma}
\label{lema:avance epsilon} We have
$\nu(z^{\star})\geq \nu(z)+\min\{\epsilon,\nu(z)\}$.
\end{lemma}
\begin{proof}
	Let us write $ \omega = \eta + hdz = \eta' + h dz' $, where $\eta' = \eta + hd(\boldsymbol{x}^{\boldsymbol{p}}F) $. Consider the decomposition in levels $ \eta =  \sum_{s=0}^{\infty}z^s\eta_s $ and $ \eta' =  \sum_{s=0}^{\infty}z'^s\eta'_s $. We have that
	\begin{align}
	\eta'_{\chi} & =  \eta_{\chi} + \boldsymbol{x}^{\boldsymbol{p}} \vartheta_{\chi} + d(\boldsymbol{x}^{\boldsymbol{p}}F) \xi_{\chi} \ , \nonumber \\
	\eta'_{\chi-1} & =  \eta_{\chi-1} - \chi \boldsymbol{x}^{\boldsymbol{p}}F \eta_{\chi} + \boldsymbol{x}^{2\boldsymbol{p}}\vartheta_{\chi-1} + d(\boldsymbol{x}^{\boldsymbol{p}}F) \xi_{\chi-1} \ , \label{eq:nuevonivelchi-1}
	\end{align}
	where $ \vartheta_{i} $ denotes the contribution to the new corresponding level coming from the levels $ \eta_s $ with $ s>\chi $, and $ \xi_{i} $ denotes the contribution coming from $ h $. We have that
	\begin{equation}\label{eq:valor_contribucion_eta}
		\nu_{\mathcal A}(\vartheta_{i}) \geq \nu_{\mathcal A}(\omega) = \nu_{\mathcal A}(\eta_{\chi}) \quad \text{for} \quad i=\chi-1,\chi \ .
	\end{equation}
	Since $ \nu_{\mathcal A}(h)\geq \varsigma_{\mathcal{A}}(\omega) - \nu(z) + \epsilon > \nu_{\mathcal A}(\eta_{\chi}) + \epsilon $ we also have that
	\begin{equation}\label{eq:valor_contribucion_h}
	\nu_{\mathcal A}(\xi_{i}) \geq \nu_{\mathcal A}(\omega) = \nu_{\mathcal A}(\eta_{\chi}) \quad \text{for} \quad i=\chi-1,\chi \ .
	\end{equation}
	In particular we obtain
	\begin{equation}\label{eq:valor_chi_estable}
		 \nu_{\mathcal A}(\eta'_{\chi}) = \nu_{\mathcal A}(\eta_{\chi}) \ .
	\end{equation}
	By Lemmas \ref{lema:proporcionalidad} and \ref{lema:preparingtheunits}, and taking into account that $ \eta_{\chi}= \boldsymbol{x}^{I_{\chi}}\alpha_{\chi}$ and $ \eta_{\chi}= \boldsymbol{x}^{I_{\chi}}\alpha_{\chi}$, where $ \boldsymbol{x}^{I_{\chi-1}}=\boldsymbol{x}^{I_{\chi}+\boldsymbol{p}} $,  we see that $\eta_{\chi-1} - \chi \boldsymbol{x}^{\boldsymbol{p}}F \eta_{\chi}$ can be written as	
	\begin{align*}
		&\boldsymbol{x}^{I_{\chi-1}}\Big((W_{\chi-1}+\boldsymbol{x}^{K}\widetilde{W}_{\chi-1})\alpha_0 + \tilde{\alpha}_{\chi-1} - \chi F\big((W_{\chi}+\boldsymbol{x}^{K}\widetilde{W}_{\chi}+ )\alpha_0+\tilde{\alpha}_{\chi}\big)\Big)= \\
		&\boldsymbol{x}^{I_{\chi-1}}\big( \boldsymbol{x}^{K}(\widetilde{W}_{\chi-1}-\chi F\widetilde{W}_{\chi})\alpha_0+  \tilde{\alpha}_{\chi-1}-\chi F\tilde{\alpha}_{\chi}\big) \ .
	\end{align*}
	Since $ \nu_{\mathcal A}(\boldsymbol{x}^{K})>\gamma $ and $ \nu_{\mathcal A}(\tilde{\alpha}_{i})>\min\{\epsilon,\nu(z)\} $ we have
	\begin{align}\label{eq:valornivelchi-1}
		\nu_{\mathcal A}(\eta_{\chi-1} - \chi \boldsymbol{x}^{\boldsymbol{p}}F \eta_{\chi}) & \geq  \nu(\boldsymbol{x}^{I_{\chi-1}}) + \min\{\gamma,\epsilon,\nu(z)\} \\
		& =  \nu_{\mathcal A}(\eta_{\chi})+\nu(z)+ \min\{\epsilon,\nu(z)\} \ . \nonumber
	\end{align}
	Now, using  \eqref{eq:valor_contribucion_eta}, \eqref{eq:valor_contribucion_h}, \eqref{eq:valor_chi_estable} and \eqref{eq:valornivelchi-1} in \eqref{eq:nuevonivelchi-1} we obtain
	\[
	\nu_{\mathcal A}(\eta'_{\chi-1})\geq \nu_{\mathcal A}(\eta'_{\chi}) + \nu(z)+\min\{ \nu(z),\epsilon \} \ .
	\]
	Finally, after performing the strict $ \gamma $-preparation, we obtain
	\[
	\nu_{\mathcal A}(\eta^{\star}_{\chi-1})\geq \nu_{\mathcal A}(\eta^{\star}_{\chi}) + \nu(z)+\min\{ \nu(z),\epsilon \} \ .
	\]
Note that condition \textbf{r2} must be satisfied for $({\mathcal A}^\star,\omega)$. In particular, we necessary have that $ \nu(z^{\star})=\nu_{\mathcal A}(\eta^{\star}_{\chi-1})-\nu_{\mathcal A}(\eta^{\star}_{\chi}) $, which is the desired property.
\end{proof}
Now, let us see how to obtain the desired contradiction. We know that $({\mathcal A}^\star,\omega)$ has the property of $\chi$-fixed critical height, and moreover the main vertex and the critical vertex are the same ones for $({\mathcal A},\omega)$ and for $({\mathcal A}^\star,\omega)$. More precisely, they are both the point
\[
(\varsigma_{\mathcal A}(\omega)-\chi\nu(z),\chi)= (\varsigma_{{\mathcal A}^\star}(\omega)-\chi\nu(z^\star),\chi).
\]
Since $ \varsigma_{\mathcal A}(\omega)-\chi\nu(z)=\varsigma_{{\mathcal A}^\star}(\omega)-\chi\nu(z^\star) $, and taking into account that $ \chi \geq 2 $, we have that
\begin{align}\label{eq:2-fixed_contradiction_1}
	\varsigma_{{\mathcal A}^\star}(\omega)-\nu(z^\star) & = \varsigma_{{\mathcal A}}(\omega)-\nu(z)+(\chi-1)(\nu(z^\star)-\nu(z))  \nonumber \\
	& \geq \varsigma_{{\mathcal A}}(\omega)-\nu(z)+(\nu(z^\star)-\nu(z)) .
\end{align}
By Lemma \ref{lema:avance epsilon}, Inequality \eqref{eq:2-fixed_contradiction_1} gives
\begin{align}\label{eq:2-fixed_contradiction_2}
	\varsigma_{{\mathcal A}^\star}(\omega)-\nu(z^\star) & \geq \varsigma_{{\mathcal A}}(\omega)-\nu(z)+\min\{\epsilon,\nu(z)\} =  \min \left\{ \epsilon +  \varsigma_{{\mathcal A}}(\omega)-\nu(z), \varsigma_{{\mathcal A}}(\omega)  \right\} \nonumber \\
	& =  \min \big\{ \min \{ \gamma, \nu_{\mathcal{A}}(h) \}, \varsigma_{{\mathcal A}}(\omega)  \big\} = \min\{ \gamma, \nu_{\mathcal{A}}(h), \varsigma_{{\mathcal A}}(\omega) \}, 
\end{align}
where we recall that $\epsilon=\min\{\gamma,\nu_{\mathcal A}(h)\}-\varsigma_{\mathcal A}(\omega)+\nu(z)$. On the other hand, since $ (\mathcal{A}^{\star},\omega) $ satisfies condition \textbf{r2}, we must have
\begin{equation}\label{eq:2-fixed_contradiction_3}
	\varsigma_{{\mathcal A}^\star}(\omega)-\nu(z^\star) < \min\{\gamma,\nu_{\mathcal{A}^{\star}}(h)\} = \min\{\gamma,\nu_{\mathcal{A}}(h)\},
\end{equation}
where the last equality is due to the fact that $ h $ is $ \gamma $-final. From Inequalities \eqref{eq:2-fixed_contradiction_2} and \eqref{eq:2-fixed_contradiction_3} we obtain
\begin{equation*}
\varsigma_{{\mathcal A}^\star}(\omega)-\nu(z^\star) \geq \varsigma_{{\mathcal A}}(\omega) .
\end{equation*}
Now, by Lemma \ref{lema:avance epsilon} we know that $ \nu(z^\star) > \nu(z) $, hence
\begin{equation*}
\varsigma_{{\mathcal A}^\star}(\omega) > \varsigma_{{\mathcal A}}(\omega) + \nu(z).
\end{equation*}
Assuming that $ (\mathcal{A},\omega) $ has the property of $ \chi $-fixed critical height, we can ite\-ra\-te the procedure performing Tschirnhausen transformations followed by strict $ \gamma $-preparations. We obtain a sequence of transformations
\[
\mathcal{A}^{\star}=\mathcal{A}^{1\star}\rightarrow \mathcal{A}^{2\star}\rightarrow \cdots \rightarrow \mathcal{A}^{n\star}\rightarrow \cdots
\]
such that
\[
\nu(z^{(i+1)\star})  > \nu(z^{i\star})  \quad \text{and} \quad  \varsigma_{{\mathcal A}^{(i+1)\star}}(\omega) > \varsigma_{{\mathcal A}^{i\star}}(\omega) + \nu(z^{i\star})
\]
for each index $ i\geq1$. Thus we have
\[
\varsigma_{{\mathcal A}^{(i+1)\star}}(\omega) > \varsigma_{{\mathcal A}}(\omega) + i\nu(z) .
\]
However, for any index $ j $ with $ \varsigma_{{\mathcal A}}(\omega) + (j-1)\nu(z) > \nu_{\mathcal{A}}(h)= \nu_{\mathcal{A}^{j\star}}(h) $, the condition \textbf{r2} can not be satisfied in $ (\mathcal{A}^{j\star},\omega) $.

We have just proved that $ (\mathcal{A},\omega) $ does not satisfy  the property of $ \chi $-fixed critical height, hence the proof of Proposition \ref{prop:critical height2} is ended.

 \subsection{Critical Height One}
 \label{Critical Height One}
  Here we give a  proof of Proposition \ref{prop:critical height1}. Thus, we assume that $({\mathcal A},\omega)$ has the property of $1$-fixed critical height and we look for a contradiction.

  As a consequence of the study of the evolution of resonances, we can do a first reduction. Let us first consider the following definitions:
  \begin{itemize}
  \item We say that $({\mathcal A},\omega)$ is  of an {\em\textbf{r2a}-resonant persistent type} if and only if $({\mathcal B},\omega)$ is \textbf{r2a} for any normalized transformation $({\mathcal A},\omega)\rightarrow({\mathcal B},\omega)$.
   \item We say that $({\mathcal A},\omega)$ is  of an {\em \textbf{r2b}$^+$-resonant persistent type} if and only if for any normalized transformation $({\mathcal A},\omega)\rightarrow({\mathcal B},\omega)$, we have that
      $({\mathcal B},\omega)$ is \textbf{r2b}-$\boldsymbol{\upsilon}$, where $\boldsymbol{\upsilon}$ is not $\mathcal B$-negative and $\boldsymbol{\upsilon}\ne 0$.
  \item We say that $({\mathcal A},\omega)$ is  of an {\em \textbf{r2b}$^{\times}$-resonant persistent type} if and only if for any normalized transformation $({\mathcal A},\omega)\rightarrow({\mathcal B},\omega)$, we have that
      $({\mathcal B},\omega)$ is \textbf{r2b}-$\boldsymbol{\upsilon}$, where $\boldsymbol{\upsilon}$ is $\mathcal B$-negative.
  \end{itemize}
  \begin{lemma}
  \label{lema:reduccionalcasor2}
  There is a normalized transformation $({\mathcal A},\omega)\rightarrow({\mathcal B},\omega)$ such that $({\mathcal B},\omega)$ is of  a resonant persistent type \textbf{r2a},  \textbf{r2b}$^{+}$ or \textbf{r2b}$^{\times}$.
  \end{lemma}
  \begin{proof} Let us recall the statements of Propositions \ref{prop:r1ycambiodecoordenadas} and \ref{prop:r1ypaquetepuiseux}.
  If we can find  a normalized transformation $({\mathcal A},\omega)\rightarrow({\mathcal B},\omega)$ such that
  $({\mathcal B},\omega)$ satisfies the resonant condition \textbf{r2a}, we get the persistent type \textbf{r2a}.
  If there is a normalized transformation $({\mathcal A},\omega)\rightarrow({\mathcal B},\omega)$ such that
  $({\mathcal B},\omega)$ is \textbf{r2b}-$\boldsymbol{\upsilon}$, with $\boldsymbol{\upsilon}$ being  not ${\mathcal B}$-negative, we obtain the persistent type \textbf{r2b}$^+$. If we can get the condition \textbf{r1}, we are in the  case \textbf{r2b}$^+$ after just one normalized Puiseux's package. The only remaining possibility is to be persistently in the case \textbf{r2b}-$\boldsymbol{\upsilon}$ with ``negative'' $\boldsymbol{\upsilon}$, this is the case of the persistent type \textbf{r2b}$^\times$.
  \end{proof}
 Let us make the following assumption from now on:
\begin{itemize}
	\item[P0:]$({\mathcal A},\omega)$ is of a persistent type  \textbf{r2a}, \textbf{r2b}$^+$ or \textbf{r2b}$^\times$.
\end{itemize}
  \begin{remark} In particular, for any normalized transformation $({\mathcal A},\omega)\rightarrow ({\mathcal B},\omega)$ the ramification index of ${\mathcal B}$ is equal to one. That is $({\mathcal A},\omega)$ is never ramified.
  \end{remark}

   With the same arguments as in Subsection
\ref{Reduction to Critical Height One 2}, we assume the following
additional properties (that are stable under any further normalized transformation):
\begin{itemize}
\item[P1:] Then main vertex and the critical vertex coincide. This is equivalent to saying that $\nu_{\mathcal A}(\omega)=\varsigma_{\mathcal A}(\omega)-\nu(z)$.
\item[P2:] $H_{{\mathcal A},\omega}$ is strongly $\gamma$-final and $({\mathcal A},\omega)$ has the $\gamma$-recessive or $\gamma$-dominant horizontal stability.
\item[P3:] Up to a $0$-nested transformattion, for $s=0,1$, we have $\omega_s=\boldsymbol{x}^{I_s}\beta_s$, with $\beta_s$ being $0$-final dominant and $\nu(\boldsymbol{x}^{I_s})=\varsigma_{\mathcal A}(\omega)-s\nu(z)$.
\end{itemize}
Once we have the above reductions, we are no more interested in performing normalized Puiseux's packages. We look for a contradiction with the existence of $({\mathcal A},\omega)$ with the stated properties by performing only normalized coordinate changes. More precisely, we will contradict the property $\varsigma_{\mathcal A}(\omega)<\gamma$.

Let us introduce a new reduction that is necessary in order to apply the truncated cohomological results in Section \ref{Truncated Cohomological Statements}. Recall that the levels $\omega_0$ and $\omega_1$ of $({\mathcal A},\omega)$ are written as $ \omega_0 = \eta_0 $ and $ \omega_1=\eta_1+h_1 dz/z$, where $ h=\sum_{s\geq 1}z^{s-1}h_s $. Moreover, we have that $\omega_1=\boldsymbol{x}^{I_1}\beta_1$, $\eta_1=\boldsymbol{x}^{I_1}\alpha_1$, $h_1=\boldsymbol{x}^{I_1}H_1$, $\omega_0=\eta_0=\boldsymbol{x}^{I_0}\alpha_0$, with  $\nu(\boldsymbol{x}^{I_0})=\varsigma_{\mathcal A}(\omega)\leq\gamma$ and
 $\nu(\boldsymbol{x}^{I_0-I_1})=\nu(z)=\nu(\boldsymbol{x^p})$, where $\Phi=z/\boldsymbol{x^p}$. In particular, we have that $I_0-I_1=\boldsymbol{p}$.
 \begin{lemma}
 	Up to performing an $\ell$-nested transformation, we have a decomposition
 	\[
 	\alpha_1={\alpha}^*_1+\tilde\alpha_1,
 	\]
 	where $d{\alpha}_1^*=0$ and $\nu_{{\mathcal A}}(\tilde\alpha_1)>0$.
 \end{lemma}
 \begin{proof}
 Write $\alpha_1=\alpha^*_1+\tilde\alpha_1$, where $d\alpha_1^*=0$ and $\tilde \alpha_1$ is not $0$-final dominant. By Proposition \ref{prop:notfinalformsandcriticalvalue}, we obtain $\nu_{\mathcal A}(\tilde\alpha_1)>0$ by a suitable $\ell$-nested transformation.
 \end{proof}
Recall that the $\ell$-nested transformations are a degenerate type of normalized coordinate changes, see Remark \ref{rk: degeneratenormalizedtransforamtions}.
 \begin{lemma}
 Assume that  $
 \alpha_1={\alpha}^*_1+\tilde\alpha_1
 $, where
 $d{\alpha_1^*}=0$
 and
 $\nu_{{\mathcal A}}(\tilde\alpha_1)>0$.
 Let $({\mathcal A},\omega)\rightarrow ({\mathcal A}^\star,\omega)$ be a normalized coordinate change, containing  a suitable $0$-nested transformation in order to get the property {\rm P3}. The $1$-level $\omega_1^\star$ of $({\mathcal A}^\star,\omega)$ is given by $\omega^\star_1={\boldsymbol x^\star}^{I_1^\star}\alpha_1^\star$, where
 \[
 \alpha^\star_1={\alpha^\star}^*_1+\tilde\alpha^\star_1,\quad \nu_{{\mathcal A}^\star}(\tilde\alpha^\star_1)\geq \min\{\nu(z), \nu_{{\mathcal A}}(\tilde\alpha_1)\}
 \]
and $d({\alpha^\star}^*_1)=0$
 \end{lemma}
 \begin{proof} Left to the reader.
 \end{proof}
\begin{remark}
\label{rk:alpha1finaldominant}
Write $\alpha_1=\alpha^*_1+\tilde\alpha_1$, where $\nu_{\mathcal A}(\tilde\alpha_1)>0$. Recall that $({\mathcal A},\omega)$ is \textbf{r2a} or \textbf{r2b}-$\boldsymbol{\upsilon}$, and that the contact rational function is $\Phi=z/\boldsymbol{x^p}$. In view of Equation \eqref{eq:condicionr2unificada2}, we have that the following statements are equivalent
\begin{itemize}
 \item $\alpha_1$ is $0$-final dominant.
 \item $\alpha_1^*$ is $0$-final dominant.
 \item $({\mathcal A},\omega)$ is \textbf{r2a} or it is \textbf{r2b}-$\boldsymbol{\upsilon}$, with $\boldsymbol{\upsilon}\ne\boldsymbol{p}$.
 \end{itemize}
\end{remark}
From now on, we additionally assume the following property:
\begin{itemize}
	\item[P4:]$\alpha_1=\alpha^*_1+\tilde\alpha_1$, where $d\alpha^*_1=0$ and $\nu_{\mathcal A}(\tilde\alpha_1)>0$.
\end{itemize}
Proposition \ref{prop:epsilonalturauno} below is the key observation we need to find the desired contradiction:
\begin{proposition}\label{prop:epsilonalturauno}
Assume that $({\mathcal A},\omega)$ has the property of the $1$-fixed critical height and properties \textnormal{P0}, \textnormal{P1}, \textnormal{P2}, \textnormal{P3} and \textnormal{P4}. There is a procedure to obtain a positive real number $\epsilon_{\mathcal A}(\omega)>0$ from $({\mathcal A},\omega)$ with the following property: there is a normalized coordinate change $({\mathcal A},\omega)\rightarrow ({\mathcal A}^\star,\omega)$ such that $({\mathcal A}^\star,\omega)$ satisfies the properties \textnormal{P0}, \textnormal{P1}, \textnormal{P2}, \textnormal{P3}, \textnormal{P4}, and moreover
  \[
  \nu(z^\star)\geq \nu(z)+\epsilon_{\mathcal A}(\omega), \quad \epsilon_{{\mathcal A}^\star}(\omega)\geq\epsilon_{\mathcal A}(\omega).
  \]
\end{proposition}
\begin{proof} In view of Equation \eqref{eq:condicionesinttruncada2}, we have that
$\nu_{\mathcal A}(\Delta_1)\geq 2\gamma$, where $\Delta_1= \eta_1\wedge \eta_0 + h_1 d \eta_0 + \eta_0 \wedge d h_1 $. By expanding $\Delta_1$, we obtain
  \begin{align*}
  	\Delta_1
  	& =  \boldsymbol{x}^{I_1+I_0}\Big( \alpha_1\wedge \alpha_0 + H_1\big(\frac{d\boldsymbol{x}^{I_0}}{\boldsymbol{x}^{I_0}}\wedge \alpha_0 + d\alpha_0\big) + \alpha_{0}\wedge \big( H_1 \frac{d\boldsymbol{x}^{I_1}}{\boldsymbol{x}^{I_1}} + dH_1 \big) \Big) \\
  	& =  \boldsymbol{x}^{I_1+I_0} \Big( \big(\alpha_1 + H_1 \frac{d\boldsymbol{x^p}}{\boldsymbol{x^p}} - dH_1 \big)\wedge \alpha_0 + H_1 d \alpha_{0} \Big) \ ,
  \end{align*}
where we recall that $\boldsymbol{p} = I_0-I_1$. Since $ 2\gamma - \nu(\boldsymbol{x}^{I_1+I_0}) \geq \nu(z) $, we have
\begin{equation}
\label{eq:valorDelta1}
  	\nu_{\mathcal A}\Big( \big(\alpha_1 + H_1 \frac{d\boldsymbol{x^p}}{\boldsymbol{x^p}} - dH_1 \big)\wedge \alpha_0 + H_1 d \alpha_{0}\Big) \geq \nu(z) \ .
  \end{equation}
Let us consider the cases when $({\mathcal A},\omega)$ is \textbf{r2a},  \textbf{r2b}$^{+}$ or \textbf{r2b}$^{\times}$. We know that they are independent situations and it is enough to show the existence of $\epsilon_{\mathcal A}(\omega)$ with the desired properties in each of the three cases.

 Assume that $({\mathcal A},\omega)$ is \textbf{r2a}. Let us see that $\nu_{\mathcal A}(H_1)\geq 0$. Since the reduced form $\bar\omega^*_1$ of the level is given by $\bar\omega^*_1=d\boldsymbol{x^\tau}/\boldsymbol{x^\tau}$, it is not possible for  $H_1$ to be a unit, hence $\nu_{\mathcal A}(H_1)>0$, since $h=H_{{\mathcal A},\omega}$ is strongly $\gamma$-final. Then, in this case $\alpha_1$ is $0$-final dominant.
Let us fix $\epsilon_1=\epsilon_{\mathcal A}(\omega)>0$ such that
\[
\epsilon_1\leq\min\{\nu(z),\nu_{\mathcal A}(H_1)\}.
\]
We have that $\nu_{\mathcal A}(\alpha_1\wedge\alpha_0)\geq\epsilon_1$, in view of Equation \eqref{eq:valorDelta1}.
By the truncated proportionality considered in Proposition \ref{prop:trucateddivision}, there is a unit $U\in k[[\boldsymbol{x},\boldsymbol{y}_{\leq\ell}]]$ and a $ 1 $-form $\tilde\alpha_0$ such that
\[
\alpha_0=U\alpha_1+\tilde\alpha_0,\quad \nu_{\mathcal A}(\tilde\alpha_0)\geq\epsilon_1.
\]
Up to perform an $\ell$-nested transformation containing $j$-Puiseux's packages for any $j=1,2,\ldots,\ell$, we can assume that $U\in k[\boldsymbol{x},\boldsymbol{y}_{\leq\ell}]$ and hence $U$ is a unit in ${\mathcal O}_{\mathcal A}$ (to see this, write $U=U^*+\tilde U$, where $U^*$ is a polynomial and $\tilde U$ belongs to a enough bigger power the maximal ideal such that we get $\nu_{\mathcal A}(\tilde U)>\epsilon_1$ by applying Proposition \ref{prop:notfinalformsandcriticalvalue}). Now, we consider the coordinate change $ \mathcal{A} \rightarrow \mathcal{A}' $ given by
\[
 z'= z + U^*\boldsymbol{x^p}.
\]
We obtain that $ \nu_{{\mathcal A}'}(\eta'_0) \geq \nu_{{\mathcal A}'}(\eta'_1)+\nu(z)+\epsilon_1 $. After a strict $ \gamma $-preparation $ \mathcal{A}' \rightarrow \mathcal{A}^{\star} $ we must have
\[
\nu(z^{\star}) = \nu_{{\mathcal A}^{\star}}(\eta^{\star}_0)-\nu_{{\mathcal A}^{\star}}(\eta^{\star}_1)\geq \nu(z)+\epsilon_1.
\]
Moreover, we still have $\epsilon_1\leq\min\{\nu(z^\star),\nu_{{\mathcal A}^\star} (H_1)\}$. This ends the proof of this case.

 Assume that $({\mathcal A},\omega)$ is  \textbf{r2b}$^{+}$ or \textbf{r2b}$^\times$. In particular $({\mathcal A},\omega)$ is \textbf{r2b}-$\boldsymbol{\upsilon}$ with $\boldsymbol{\upsilon}\ne 0$.
  In this situation $H_1$ is a unit that we can write $H_1=\xi+\tilde H_1$, with $\nu_{\mathcal A}(\tilde H_1)>0$ and $0\ne\xi\in k$. We see this looking, as before, to the reduced part $\bar\omega^*_1$ of the $1$-level $\omega_1$ of $\omega$ in $\mathcal A$. Let us recall the decomposition $\alpha_1=\alpha_1^*+\tilde\alpha_1$, where $d\alpha^*_1=0$ and $\nu_{\mathcal A}(\tilde\alpha_1)>0$. From Equation \eqref{eq:valorDelta1}, we obtain that
 \begin{equation*}
  	  \nu_{\mathcal A}\big((\xi^{-1}\alpha^*_1 +  \frac{d\boldsymbol{x^p}}{\boldsymbol{x^p}} )\wedge \alpha_0 +  d \alpha_{0}\big) \geq \epsilon_2 \ ,
  \end{equation*}
  where we denote $ \epsilon_2= \min\{\nu_{\mathcal A}(\tilde{\alpha}_1),\nu_{\mathcal A}(\tilde{H}_1),\nu(z)\}$. Let us consider the $1$-form $\sigma$ given by
  \[
  \sigma = \xi^{-1}\alpha^*_1 + d\boldsymbol{x^p}/\boldsymbol{x^p}.
  \]
 Looking at Equation \eqref{eq:condicionr2unificada2} we see that
  $
  \sigma = {d \boldsymbol{x^\upsilon}}/{\boldsymbol{x^\upsilon}} + \tilde{\sigma}
  $,
  where $ \tilde{\sigma} $ is not $ 0 $-final dominant. Since $\boldsymbol{\upsilon}\ne \boldsymbol{0}$, then $\sigma$ is $0$-final dominant. Moreover, we have that
  $ d\sigma=0$ and
  $
  \nu_{\mathcal A}(d_{\sigma}(\alpha_0))\geq \epsilon_2
  $.
  Thus, we can apply Corollary \ref{cor:logpoincare3}.  We have two possibilities to consider:

 {\em First case:} $\boldsymbol{\upsilon}\notin {\mathbb Z}^r_{\leq 0}$. By
  Corollary \ref{cor:logpoincare3}, there is a formal function $U$ such that
  \[
  \alpha_0= U\sigma+dU+\tilde\alpha_0,\quad \nu_{\mathcal A}(\tilde\alpha_0)\geq \epsilon_2.
  \]
  Note that $U$ must be a unit, since $({\mathcal A},\omega)$ is \textbf{r2b}-$\boldsymbol{\upsilon}$. As before, we can assume that $U\in {\mathcal O}_{\mathcal A}$ up to an additional $\ell$-nested transformation and modulo $\gamma$-truncation, we do not detail this part. Now, we perform the coordinate change
  \[
  z'=z+\xi^{-1}U\boldsymbol{x^p}.
  \]
  We obtain that
  $\nu_{{\mathcal A}'}(\eta'_0) \geq \nu_{{\mathcal A}'}(\eta'_1)+\nu(z)+\epsilon_2
  $. After a strict $ \gamma $-preparation $ \mathcal{A}' \rightarrow \mathcal{A}^{\star} $ we have
  \[
   \nu(z^{\star}) = \nu_{{\mathcal A}^{\star}}(\eta^{\star}_0)-\nu_{{\mathcal A}^{\star}}(\eta^{\star}_1)\geq \nu(z)+\epsilon_2.
  \]
The situation repeats with the same $\epsilon_2$.

{\em Second case: }$ \boldsymbol{\upsilon} \in \mathbb{Z}_{\leq 0}^{r}$. Then Corollary \ref{cor:logpoincare3} allows to write  $ \alpha_0 $ as
  \[
  \alpha_0 = U \sigma + dU + \boldsymbol{x}^{-\boldsymbol{\upsilon}}\sum_{i=1}^{r}\mu_i \frac{dx_i}{x_i} +\tilde{\alpha}_0 \ , \quad \nu_{\mathcal A}(\tilde{\alpha}_0)\geq \epsilon_2 \ .
  \]
If all the $\mu_i=0$, we proceed as before. Assume that $\boldsymbol{\mu}\ne\boldsymbol{0}$ and $ \nu(\boldsymbol{x}^{-\boldsymbol{\upsilon}}) \geq \epsilon_2$. Consider the coordinate change $ \mathcal{A} \rightarrow \mathcal{A}' $ given by $ z'= z + \xi^{-1} U\boldsymbol{x^p} $. We have that
  \begin{equation}\label{eq:upsilon_negativo}
  	 z \big(\alpha_1+\xi \frac{dz}{z}\big) + \boldsymbol{x^p}\alpha_0= z' \big(\alpha^*_1+\xi \frac{dz'}{z'}\big) + \boldsymbol{x}^{\boldsymbol{p}-\boldsymbol{\upsilon}}\sum_{i=1}^{r}\mu_i \frac{dx_i}{x_i} + \boldsymbol{x}^{\boldsymbol{p}}(\tilde{\alpha}_0-\xi^{-1}U\tilde{\alpha}_1) \ .
  \end{equation}
  Since $ \nu(\boldsymbol{x}^{-\boldsymbol{\upsilon}}) \geq \epsilon_2$, by Equation \eqref{eq:upsilon_negativo}, we have that
  $ \nu_{{\mathcal A}'}(\eta'_0)\geq \nu_{{\mathcal A}'}(\eta'_0) + \nu(z) + \epsilon_2.
  $
  Therefore, after a strict $ \gamma $-preparation $ \mathcal{A}' \rightarrow \mathcal{A}^{\star} $ we must have
  \[
  \nu(z^{\star}) = \nu_{E^{\star}}(\eta^{\star}_0)-\nu_{E^{\star}}(\eta^{\star}_1)\geq \nu(z)+\epsilon_2.
  \]
  Moreover, we can restart with the same $\epsilon_2$.
It remains to consider the possibility $\boldsymbol{\mu}\ne\boldsymbol{0}$ and $ \nu(\boldsymbol{x}^{-\boldsymbol{\upsilon}}) <\epsilon_2$. Let us show that this situation does not happen. Since $ \nu(\boldsymbol{x}^{-\boldsymbol{\upsilon}}) < \epsilon_2$, by Equation \eqref{eq:upsilon_negativo} the new $ 0 $-level has the form
  \[
  \eta'_0 = \boldsymbol{x}^{I_1 + \boldsymbol{p}-\boldsymbol{\upsilon}}\sum_{i=1}^{r}\mu_i \frac{dx_i}{x_i} + \tilde{\eta}'_0 \ , \quad \nu_{{\mathcal A}'}( \tilde{\eta}'_0) >  \nu(\boldsymbol{x}^{I_1 + \boldsymbol{p}-\boldsymbol{\upsilon}}) \ .
  \]
  Then $\eta'_0$ is $ \nu(\boldsymbol{x}^{I_1+ \boldsymbol{p}-\boldsymbol{\upsilon}}) $-final dominant and thus $(\mathcal{A}',\omega) $ is strictly $ \gamma $-prepared without  performing any $ \ell $-nested transformation. Since condition \textbf{r2b} is satisfied, we must have that \[ \nu(z')= \nu_{{\mathcal A}'}(\eta'_0)-\nu_{{\mathcal A}'}(\eta'_1) =  \nu(\boldsymbol{x}^{\boldsymbol{p}-\boldsymbol{\upsilon}}).  \]
  Hence $ \boldsymbol{p}' = \boldsymbol{p}-\boldsymbol{\upsilon} $. This implies that $ \boldsymbol{\upsilon}'=\boldsymbol{0} $, that is we obtain a resonance \textbf{rb2}-$\boldsymbol{0}$, but this does not occur when we are in the persistent situations  \textbf{r2b}$^{+}$ or \textbf{r2b}$^\times$.
The proof is completed.
\end{proof}

By applying finitely many times Proposition  \ref{prop:epsilonalturauno}, we obtain a normalized transformation $({\mathcal A},\omega)\rightarrow ({\mathcal B},\omega)$ such that
$\nu_{\mathcal B}(z^{\mathcal B})>\gamma$. But we know that
\[
\varsigma_{\mathcal B}(\omega)\geq \chi_{\mathcal B}(\omega)\nu(z^{\mathcal B})=\nu(z^{\mathcal B})>\gamma.
\]
This is a contradiction. In this way, the proof of Proposition \ref{prop:critical height1} is ended.

\section{Local Uniformization of  Rational Foliations}
\label{Local Uniformization of  Rational Foliations}
In this section we show that Theorem \ref{teo:formalforms} implies Theorem \ref{teo:mainmain}. More precisely, we show how to obtain pre-simple points as they were defined in the Introduction. In view of the results in \cite{FeD}, this is enough,  since we can pass from pre-simple to simple points.

Before considering the proof of Theorem \ref{teo:mainmain}, we are going to develop a ``non-truncated'' version of the results in Section \ref{Control by the Critical Height}.

\subsection{Stable Foliated Spaces} Consider a parameterized formal foliated space $({\mathcal A},\omega)$ with $I_{\mathcal A}(\omega)=\ell+1$. Recall that $\omega\wedge d\omega=0 $.
We say that $({\mathcal A},\omega)$ is {\em $\gamma$-stable} if and only if $({\mathcal A},\omega)$ is strictly $\gamma$-prepared and $\varsigma_{\mathcal A}(\omega)\leq\gamma$. Note that if $({\mathcal A},\omega)$ is $\gamma$-stable and $\gamma'\geq\gamma$, then we also have that
$({\mathcal A},\omega)$ is $\gamma'$-stable. We say that $({\mathcal A},\omega)$ is {\em stable} if and only if there is $\gamma$ such that  $({\mathcal A},\omega)$ is  $\gamma$-stable. A {\em stable normalized transformation} is a transformation
\[
({\mathcal A},\omega)\rightarrow ({\mathcal B},\omega)
\]
that is a normalized transformation with respect to a truncation value $\gamma$ for which $({\mathcal A},\omega)$ and $({\mathcal B},\omega)$ are $\gamma$-stable. In Lemma \ref{lema:nontruncatedstability} we show that we have ``enough'' stable normalized transformations:
\begin{lemma}
\label{lema:nontruncatedstability}
Let us assume that $({\mathcal A},\omega)$ is $\gamma_0$-stable and let ${\mathcal A}\rightarrow {\mathcal A}'$ be a $(\ell+1)$-Puiseux's package or a $(\ell+1)$-coordinate change, assuming that the coordinate change is trivial or the ramification index of $\mathcal A$ is one. There is $\gamma_1\geq\gamma_0$ such that if we perform a strict $\gamma_1$-preparation $({\mathcal A}',\omega)\rightarrow ({\mathcal A}_1,\omega)$, then $({\mathcal A}_1,\omega)$ is $\gamma_1$-stable and moreover
$
\chi_{\mathcal A}(\omega)\geq\chi_{{\mathcal A}_1}(\omega)
$.
\end{lemma}
\begin{proof}
Let $\rho$ and $\rho'$ denote the values $\nu(z)$ and $\nu(z')$ respectively, and let $ \chi $ denote the critical height $\chi_{{\mathcal A}}(\omega)$. In the case of a Puiseux's package, it is enough to take $\gamma_1\geq\gamma_0$ with
\[
\gamma_1\geq \nu_{\mathcal A}(\omega)+ \chi_0(\rho+\rho').
\]
In the case of a coordinate change, we take $\gamma_1\geq\gamma_0$ with $\gamma_1\geq \nu_{\mathcal A}(\omega)+ \chi_0\rho'$. The inequality $
\chi_{\mathcal A}(\omega)\geq\chi_{{\mathcal A}_1}(\omega)
$ follows from Proposition \ref{prop:stability critical height}, since $({\mathcal A},\omega)\rightarrow ({\mathcal A}_1,\omega)$ is a $\gamma_1$-normalized transformation.
\end{proof}

By Proposition \ref{prop:stability critical height}, we know that $\chi_{\mathcal B}(\omega)\leq \chi_{\mathcal A}(\omega)$, for any stable normalized trasformation $({\mathcal A},\omega)\rightarrow({\mathcal B},\omega)$.
We say that a stable $({\mathcal A},\omega)$ has the property of {\em non-truncated $\chi_0$-fixed critical height} if and only if $\chi_{{\mathcal B}}(\omega)=\chi_0$ for any stable normalized transformation $({\mathcal A},\omega)\rightarrow({\mathcal B},\omega)$.

\begin{proposition}
\label{pro:nontruncated reduction to height one} Consider a stable $({\mathcal A},\omega)$ with the property of {non-truncated $\chi_0$-fixed critical height}, where $\chi_0\geq 2$. Up to performing a stable normalized transformation $({\mathcal A},\omega)\rightarrow ({\mathcal A}^\star,\omega)$ we get that $({\mathcal A}^\star,\omega)$ satisfies the following property:
\begin{quote}
``For any stable normalized transformation $({\mathcal A}^\star,\omega)\rightarrow ({\mathcal B},\omega)$ and any $\gamma\in {\mathbb R}$, we have that $({\mathcal B},{H}_{\mathcal B}(\omega))$ is not $\gamma$-final dominant.''
\end{quote}
\end{proposition}
\begin{proof}
It is enough to work like in Subsections \ref{Reduction to Critical Height One} and \ref{Reduction to Critical Height One 2}, with the following adaptations:
\begin{itemize}
\item About the results in Subsection \ref{Reduction to Critical Height One}:
We assumed that the horizontal coefficient  $H_{\mathcal{B}}(\omega)$ was $ \gamma $-final. Instead of that, we use that $H_{\mathcal{B}}(\omega)$ is a monomial times a unit.
 In Lemma \ref{lema:sinetiqueta}, when we define the value $ \epsilon $, we must put \[
    \epsilon = \nu_{\mathcal A}(H_{\mathcal{B}}(\omega))-\varsigma_{\mathcal B}(\omega)+\nu(z^{\mathcal{B}}), \]
    instead of $ \epsilon=\min\{\gamma,\nu_{\mathcal A}(H_{\mathcal{B}}(\omega))\}-\varsigma_{\mathcal B}(\omega)+\nu(z^{\mathcal{B}}) $.

\item About results in Subsection \ref{Reduction to Critical Height One 2}:
Every time $ \epsilon $ appears, recall that it is defined without taking any value $ \gamma $ into account. In the last part of Subsection \ref{Reduction to Critical Height One 2}, we find a contradiction by increasing the value of the $ (l+1) $-th dependent parameter until $ \varsigma_{{\mathcal B}^{(n)}}(\omega)>\gamma $. Instead of this, the contradiction appears when \[  \varsigma_{{\mathcal B}^{(n)}}(\omega)-\nu(z^{{\mathcal B}^{(n)}}) \geq \nu_{{\mathcal B}^{(n)}}(H_{\mathcal{B}}(\omega)).\] 
\end{itemize}
\end{proof}
\begin{proposition}
\label{pro:nontruncated height one} Consider a stable $({\mathcal A},\omega)$ with the property of {non-truncated $1$-fixed critical height}. There is a stable normalized transformation $({\mathcal A},\omega)\rightarrow ({\mathcal B},\omega)$  such that  $({\mathcal B},\omega)$ has the resonance property \textbf{r2a} or \textbf{r2b}.
\end{proposition}
\begin{proof}
Consequence of Lemma \ref{lema:reduccionalcasor2}.
\end{proof}
\begin{remark}
 \label{rk:alturaunofin}
 The cases of \textbf{r2a} or \textbf{r2b} with $\chi_{\mathcal A}(\omega)=1$ correspond to pre-simple points, when the critical vertex coincide with the main vertex, up to perform independent blow-ups to allow a division by a monomial.
\end{remark}
\subsection{Rational Foliations}
Let us start with a projective variety variety $ M $, with $K=k(M)$  and an rational codimension one foliation ${\mathcal F}\subset\Omega_{K/k}$. Up to performing appropriate blow-ups, we can assume that the center $P$ of $R$ in $M$ is a non-singular point of $M$ and we have a locally parameterized model  $\mathcal{A}=(\mathcal{O}_\mathcal{A};\boldsymbol{x},\boldsymbol{y})$ adapted to $R$ such that ${\mathcal O}_{M,P}={\mathcal O}_{\mathcal A}$. Now, we consider a nonzero differential $1$-form $\omega\in {\Omega}^1_{\mathcal A}\cap\mathcal F$.  Let us remark that
\[
\omega\in \Omega^1_{{\mathcal O}_{\mathcal A}/k}[\log \boldsymbol{x}]\subset \Omega^1_{\mathcal A}.
\]
Then $\omega$ is a Frobenius integrable rational differential $1$-form. In particular, we see that the coefficients of $\omega$ belong to ${\mathcal O}_{\mathcal A}\subset K$.
Our objective is to perform blow-ups of $M$ to get that $\omega$ satisfies the definition
of pre-simple point given in the Introduction. The kind of blow-ups we perform are indicated by the corresponding allowed transformations of locally parameterized models, hence the centers of the blow-ups are non-singular and of codimension two, when we localise the situation at the center of the valuation.

 We have a parameterized formal foliated space $(\mathcal{A},\omega)$, hence it is
 $\gamma$-truncated for any $\gamma \in \mathbb{R}$. Thus we are in the conditions of Theorem \ref{teo:formalforms}, for any $ \gamma\in {\mathbb R} $. Applying Theorem \ref{teo:formalforms}, there are two possibilities:
\begin{itemize}
\item[a)] There are $ \rho \in \mathbb{R} $ and an allowed transformation $ \mathcal{A}\rightarrow \mathcal{A}' $ such that $ (\mathcal{A}',\omega) $ is $ \rho $-final dominant.
\item[b)] For any $ \rho \in \mathbb{R} $, there is an allowed transformation $ \mathcal{A}\rightarrow \mathcal{A}' $ such that $ (\mathcal{A}',\omega) $ is $ \rho $-final recessive.
\end{itemize}
In case a), we perform independent blow-ups $\mathcal{A}'\rightarrow \mathcal{A}''$, as in Proposition \ref{prop:monlocprincipalization}, to obtain that $ \omega = {\boldsymbol{x''}^I}{\tilde\omega} $, where $ \tilde{\omega} \in \Omega^1_{\mathcal{A}''} $ is $0$-final dominant. This situation corresponds with a pre-simple corner as defined in the Introduction.

It remains to consider the case b). This case corresponds to a situation of ``infinite value'' of $\omega$, in particular it does not appear when $\omega=df$, for $f\in {\mathcal O}_{\mathcal A}$. Nevertheless, it can happen for instance for an Euler's Equation, where the differential form is of polynomial type.

Let us work by induction on ${I}_{\mathcal A}(\omega)$, see Subsection \ref{Induction Structure}. If $I_{\mathcal A}(\omega)=0$, we end by Corollary \ref{cor:principalizacion} and in fact we are in the case a). Assume ${I}_{\mathcal A}(\omega)=\ell+1$, and take notations as in Subsection \ref{Induction Structure}.
Let us write
\[
\omega = \sum_{i=1}^{r} f_i \frac{dx_i}{x_i} + \sum_{j=1}^{\ell} g_j dy_j + h dz \ ,
\]
where $z$ is the last dependent parameter we consider. Recall that $h\in {\mathcal O}_{\mathcal A}\subset K$ and thus we have $\nu(h)<\infty$, or $h=0$.

If $h=0$, the integrability condition $\omega\wedge d\omega=0$ shows the existence of rational function $F$ and a rational $1$-differential form $\eta$ such that $I_{\mathcal A}(\eta)\leq\ell$ and
\[
\omega=F\eta.
\]
We end by induction.

Then, we can suppose that $h\ne 0$ and hence $\nu(h)<\infty$. Consider the (formal) level decomposition
\[
\omega =\sum_{s=0}^\infty\omega_s= \sum_{s=0}^{\infty} z^s ( \eta_s + h_s \frac{dz}{z}) \ .
\]
For any index $ s $, we can apply Theorem \ref{teo:functions} to $ h_s $. We have two options:
\begin{itemize}
\item[i)] There are $ \rho \in \mathbb{R} $ and an $ \ell $-nested transformation $ \mathcal{A}\rightarrow \mathcal{B} $ such that $ (\mathcal{B},h_s) $ is $ \rho $-final dominant.
\item[ii)] For any $ \rho \in \mathbb{R}, $ there is an $ \ell $-nested transformation $ \mathcal{A}\rightarrow \mathcal{B} $ such that $ (\mathcal{B},h_s) $ is $ \rho $-final recessive.
\end{itemize}
Let us denote $ k_0 = \min\{s; \ h_s \text{ satisfies i)}\} $. Let us check that $ k_0 < \infty $. Suppose that $ k_0 = \infty $ and let $ \rho = \nu(h) $. Consider an $ \ell $-nested transformation $ \mathcal{A}\rightarrow \mathcal{B} $ such that $ h_s $ is $ \rho $-final recessive for any $ s \leq  \rho / \nu(z) $. After a $(\ell+1)$-Puiseux's package $ \mathcal{B}\rightarrow \mathcal{C} $ we get $ \nu_{{\mathcal C}}(h)>\rho=\nu(h)$, which is a contradiction.

By performing an $\ell$-nested transformation, we assume that $({\mathcal A},h_{k_0})$ is $\rho_0$-final dominant, where we take $\rho_0=\nu_{\mathcal A}(h_{k_0})$.

Take $ \gamma_1 > \rho_0 + k_0 \nu(z) $. Since $ (\mathcal{A},\omega) $ is satisfies the $\gamma_1$-truncated integrability condition, by Theorem \ref{teo:preparation}
there is a strict $\gamma_1$-preparation $ \mathcal{A}\rightarrow \mathcal{A}_1 $ of $ \omega $. The abscissa of the $k_0$-level of $({\mathcal A}_1,\omega)$ is smaller or equal than $\rho_0=\nu_{{\mathcal A}}(h_{k_0})=\nu_{{\mathcal A}_1}(h_{k_0})$.  This implies that
\[
\varsigma_{{\mathcal A}_1}(\omega)\leq\gamma_1.
\]
Then $({\mathcal A}_1,\omega)$ is stable and we are able to apply Propositions \ref{pro:nontruncated reduction to height one} and \ref{pro:nontruncated height one}.

Up to a stable normalized transformation, we can assume that $({\mathcal A}_1,\omega)$ has the property of $\chi_0$-fixed critical height.

The case $\chi_0\geq 2$ does not happen. Indeed, by
Proposition \ref{pro:nontruncated reduction to height one}, the critical height cannot stabilize in $\chi_0\geq 2$, since the horizontal coefficient $h$ is nonzero and belongs to the local ring ${\mathcal O}_{\mathcal A}\subset K$. Then, it has a well defined finite value and thus we can transform it into a unit times a monomial. To see this, we can apply Proposition \ref{prop:conditionatedluofafunction} with respect to $\gamma$ with $\gamma>\nu(h)$.

 Note that $\chi_0\geq 1$ since we are in case  b), see also Proposition \ref{prop:critical heightcero}. If $\chi_0=1$, by Proposition \ref{pro:nontruncated height one} and  Remark \ref{rk:alturaunofin} we get a pre-simple point.

The proof of Theorem \ref{teo:mainmain} is ended.

\end{document}